\documentclass[11pt,a4paper]{amsart}
\usepackage{mathpazo}  
\usepackage{amsmath}
\usepackage{amssymb}

\usepackage{graphicx}
\usepackage[all]{xy}
\usepackage{latexsym}
\usepackage{mathrsfs}
\usepackage{fancyhdr}
\usepackage{enumerate}
\usepackage[english]{babel}
\usepackage{lipsum}
\usepackage{etoolbox}
\usepackage[colorlinks=true,linkcolor=Maroon, citecolor=magenta, pagebackref]{hyperref} 

\usepackage[usenames,dvipsnames]{xcolor}

\usepackage{tikz, tkz-euclide, float, pgfplots, overpic}

\newcommand*\circled[1]{\tikz[baseline=(char.base)]{
            \node[shape=circle,draw,inner sep=2pt] (char) {#1};}}

\usetikzlibrary{arrows,calc, decorations.markings, positioning, fadings}
\usetkzobj{all}
\tikzset{snake it/.style={decorate, decoration=snake}}

\newdimen\tlx
\newdimen\tlx
\newdimen\brx
\newdimen\bry

\usepackage[left=2.5cm, top=3cm, right=2.5cm, bottom=3cm, bindingoffset=0.5cm]{geometry}

\newcommand{\Z}{\ensuremath{\mathbb{Z}}}

\newcommand{\N}{\ensuremath{\mathbb{N}}}

\newcommand{\CC}{\ensuremath{\mathcal{C}}} 

\DeclareMathOperator{\Hom}{Hom}
\DeclareMathOperator{\Ext}{Ext}
\DeclareMathOperator{\End}{End}
\DeclareMathOperator{\add}{add}

\DeclareMathOperator{\ind}{ind} 

\newcommand{\img}{\operatorname{Im}\nolimits}
\newcommand{\modcat}{\operatorname{mod}\nolimits}

\newcommand{\rad}{\operatorname{{rad}}}
\newcommand{\soc}{\operatorname{{soc}}}

\newcommand{\supp}{\operatorname{supp}\nolimits}

\newcommand{\mc}[1]{\ensuremath{\mathcal{#1}}}   

\usepackage{collectbox}    

\makeatletter
\newcommand{\sqbox}{%
    \collectbox{%
        \@tempdima=\dimexpr\width-\totalheight\relax
        \ifdim\@tempdima<\z@
            \fbox{\hbox{\hspace{-.5\@tempdima}\BOXCONTENT\hspace{-.5\@tempdima}}}%
        \else
            \ht\collectedbox=\dimexpr\ht\collectedbox+.5\@tempdima\relax
            \dp\collectedbox=\dimexpr\dp\collectedbox+.5\@tempdima\relax
            \fbox{\BOXCONTENT}%
        \fi
    }%
}
\makeatother

\theoremstyle{theorem}
\newtheorem{theorem}{Theorem}[section]         
\newtheorem{lemma}[theorem]{Lemma}
\newtheorem{cor}[theorem]{Corollary}
   
\newtheorem{prop}[theorem]{Proposition}

\theoremstyle{remark}
\newtheorem{remark}[theorem]{Remark}
\newtheorem*{rem}{Remark}
\newtheorem{ex}[theorem]{Example}
\newtheorem{Construction}[theorem]{Construction}

\theoremstyle{definition}
\newtheorem{defi}[theorem]{Definition} 

\title{Mutation of friezes} 

\author[K.~Baur, E.~Faber, S.~Gratz, K.~Serhiyenko, G.~Todorov]{Karin Baur, Eleonore Faber, Sira Gratz, Khrystyna Serhiyenko, Gordana Todorov}

\address{Institut f\"{u}r Mathematik und Wissenschaftliches Rechnen, 
Universit\"{a}t Graz, NAWI Graz, Heinrichstrasse 36, 
A-8010 Graz, Austria}
\email{baurk@uni-graz.at}

\address{
Department of Mathematics, University of Michigan, Ann Arbor, MI 48109, USA
}
\email{emfaber@umich.edu}

\address{
Mathematical Institute, University of Oxford, Oxford, OX2 6GG, UK
}
\email{gratz@maths.ox.ac.uk}

\address{Department of Mathematics, University of California at Berkeley, Berkeley, CA 94720, USA}
\email{khrystyna.serhiyenko@berkeley.edu}

\address{Department of Mathematics, Northeastern University, Boston, MA 02115, USA}
\email{g.todorov@neu.edu}
\date{\today}

\begin{document}

\begin{abstract}
We study mutations of Conway-Coxeter friezes which are compatible with mutations of cluster-tilting objects in the associated cluster category of Dynkin type $A$. More precisely, we provide a formula, relying solely on the shape of the frieze, describing how each individual entry in the frieze changes under cluster mutation. We observe how the frieze can be divided into four distinct regions, relative to the entry at which we want to mutate, where any two entries in the same region obey the same mutation rule. Moreover, we provide a combinatorial formula for the number of submodules of a string module, and with that a simple way to compute the frieze associated to a fixed cluster-tilting object in a cluster category of Dynkin type $A$ in the sense of Caldero and Chapoton.
\end{abstract}


\maketitle

\section{Introduction}\label{sec:intro}

Coxeter introduced friezes in \cite{Coxeter} in the early 1970's, inspired by Gauss's {\em pentagramma mirificum}. A {\em frieze} is a grid of positive integers, with a finite number of infinite rows, where the top and bottom rows are bi-infinite repetition of $0$s and the second to top and the second to bottom row are bi-infinite repetitions of $1$s

\[
 \xymatrix@=0.5em{ 
 &\ldots && 0 && 0 && 0 && 0 && \ldots &\\
   && 1 && 1 && 1 && 1 && 1 && \\
  &\ldots && m_{-1,-1} && m_{00} && m_{11} && m_{22} && \ldots& \\
   && m_{-2,-1} && m_{-1,0} && m_{01} && m_{12} && m_{23} && \\
   & \ldots && \ldots && \ldots && \ldots && \ldots && \ldots \\
    && 1 && 1 && 1 && 1 && 1 && \\
   &\ldots && 0 && 0 && 0 && 0 && \ldots &
  }
\]
satisfying the {\em frieze rule}: for every set of adjacent numbers arranged in a {\em diamond}
\[
 \xymatrix@=0.5em{
 & b & \\
 a && d \\
 & c &
 }
\]
we have
\[
 ad - bc = 1.
\]
Every frieze is uniquely determined by its quiddity sequence, that is the sequence $(m_{ii})_{i \in \Z}$ of integers in the third row 
(the first non-trivial row) of the frieze. Moreover, it is invariant under a glide reflection and thus 
periodic. 
The {\em order of the frieze} is one less than the number of rows. If the order 
is $n$, then the frieze is necessarily $n$-periodic. 

Coxeter conjectured that friezes are in bijection with triangulations of convex polygons -- a 
result that was proven shortly thereafter by Conway and Coxeter in \cite{CoCo1} and \cite{CoCo2}: 
Friezes of order $n$ 
are in bijection with triangulations of a convex $n$-gon. This provides a first link to cluster combinatorics of Dynkin type $A$. 
Indeed, the theory of friezes has gained fresh momentum in the last decade in 
relation to cluster theory, as observed by Caldero and Chapoton \cite{CalderoChapoton}.

Cluster algebras were introduced by Fomin and Zelevinsky in \cite{FZ1}. They are commutative rings with a certain combinatorial structure. The basic set-up of a cluster algebra is the following: We have a distinguished set of generators, called {\em cluster variables}, which can be grouped into overlapping sets of a fixed cardinality, called {\em clusters}. A process called {\em mutation} allows us to jump from one cluster to another by replacing one cluster variable by a unique other cluster variable. To construct a cluster algebra, we start with an {\em initial cluster} and some combinatorial rule, for our purposes encoded in a quiver $Q$ without loops or two-cycles, which determines how to mutate clusters. By iterated mutation of our initial cluster we then obtain every cluster variable of our cluster algebra after finitely many steps.

A particularly well-behaved example of cluster algebras are cluster algebras of Dynkin type $A$. They are of finite type, that is, there are only finitely many clusters. In fact, clusters in the cluster algebra of Dynkin type $A_{n-3}$ are in bijection with triangulations of a convex $n$-gon ($n\ge 3$). 
That is, we obtain a bijection between clusters in the cluster algebra of Dynkin type $A_{n-3}$ and friezes of order $n$.

The goal of this paper is to complete the picture of cluster combinatorics in the context of friezes. More precisely, we determine how mutation of a cluster affects the associated frieze, thus effectively introducing the notion of a mutation of friezes that is compatible with mutation in the associated cluster algebra. This provides a useful new tool to study cluster combinatorics of Dynkin type $A$. 

We approach this problem via (generalized) cluster categories. Cluster categories associated to finite dimensional hereditary algebras were introduced by Buan, Marsh, Reineke, Reiten and Todorov \cite{BMRRT} as certain orbit categories of the bounded derived category of the hereditary algebra; a generalized version for algebras of global dimension $2$ has been introduced by Amiot \cite{[Am]}. Cluster categories are certain triangulated categories which mirror the combinatorial behaviour of the associated cluster algebras. The role of cluster variables is taken on by the indecomposable objects, while clusters correspond to the so-called {\em cluster-tilting objects}. There exists a notion of mutation of cluster-tilting objects, which formally relies on the category's triangulated structure. Briefly put, mutation replaces an indecomposable summand of a cluster-tilting object by a unique other indecomposable object such that we again get a cluster-tilting object.

Caldero and Chapoton \cite{CalderoChapoton} have provided a formal link between cluster categories and cluster algebras by introducing what is now most commonly known as the {\em Caldero Chapoton map} or {\em cluster character}. Fixing a cluster-tilting object (which takes on the role of the initial cluster), it associates to each indecomposable in the cluster category a unique cluster variable in the associated cluster algebra, sending the indecomposable summands in the cluster-tilting object to the initial cluster. Postcomposing the Caldero Chapoton map with the specialization of all initial cluster variables to one gives rise to the {\em specialized Caldero Chapoton map}, whose values are positive integers. 

Let now $\CC$ be a cluster category of Dynkin type $A$, that is $\CC = D^b(kQ)/\tau^{-1}[1]$ where $Q$ is an orientation of a Dynkin diagram of type $A$, $\tau$ denotes the Auslander-Reiten translation and $[1]$ denotes the suspension in the bounded derived category $D^b(kQ)$. The specialized Caldero Chapoton map allows us to jump directly from the cluster category with a fixed cluster-tilting object $T$ to a frieze $F(T)$:  For each vertex of the Auslander-Reiten quiver of $\CC$ we pick a representant of the associated isomorphism class of indecomposable objects and label each vertex in the Auslander-Reiten quiver of $\CC$ by the image of this representant under the specialized Caldero Chapoton map. Completing accordingly with rows of $0$s and $1s$ at the top and bottom this yields a frieze, cf.\ \cite[Proposition~5.2]{CalderoChapoton}.

Fixing a cluster-tilting object $T$ in $\CC$, we consider the associated cluster-tilted algebra 
$B_T = \End(T)$. It has been shown by Buan, Marsh and Reiten that there is an equivalence of categories $\CC \big/ \add(T[1]) \cong\modcat (B_T)$. 
Each indecomposable object in $\CC$ thus either lies in $T[1]$ and can be viewed as the suspension of an indecomposable projective $B_T$-module, or it can be identified with a unique indecomposable $B_T$-module. The specialized Caldero Chapoton map sends each indecomposable summand of $T[1]$ to 1 and each indecomposable $B_T$-module $M$ to the sum, taken over dimension vectors of submodules of $M$, of the Euler-Poincar\'e characteristic of the Grassmannians of submodules of $M$ of a given dimension vector. Since in our setting all modules are string modules, all Grassmannians appearing in this sum are points. Hence, the specialized Caldero Chapoton map sends an indecomposable $B_T$-module to the number of its submodules. The frieze $F(T)$ associated to the cluster-tilting object $T$ thus has entries of $1$ in the positions of the vertices associated to indecomposable summands of $T[1]$ and all the other entries (that do not lie in the mandatory rows of $0$s and $1$s at the top and bottom) are given by the number of submodules of the indecomposable $B_T$-module sitting in the same position in the Auslander-Reiten quiver of $\CC$.

Understanding the Caldero Chapoton map in Dynkin type $A$ thus amounts to knowing the number of submodules of $B_T$-modules where $B_T$ is cluster-tilted algebra of Dynkin type $A$. In Theorem  \ref{Prop:explicitsubmodules}, our first main result, we provide a combinatorial formula for the number of submodules of any given indecomposable $B_T$-module: Each $B_T$-module is a string module and hence has a description in terms of the lengths of the individual legs. If $(k_1,\dots, k_m)$ are these lengths for an indecomposable $M$, then the number $s(M)$ of submodules of $M$ is given by 
\[
 s(M) = 1 + \sum_{j=0}^m\sum_{|I|=m-j}\prod_{i\in I}k_i 
\]
where the second sum runs over all admissible subsets $I$ of $\{1,\dots, m\}$ (for details see Section~\ref{S:Number of submodules}). 
This formula relies on the shape $(k_1, \ldots, k_m)$ of the module, which in turn can be directly read of from its position in the Auslander-Reiten quiver of the cluster category, and allows for a straight-forward combinatorial way to compute the number of submodules of any string module, and in particular, obtain the frieze associated to a given cluster-tilting object in $\CC$.
It has been brought to our attention that parallel to our work, \cite{CSch} established a formula for the number of submodules in the context of snake graphs and continued 
fractions. 

Assume now that our cluster-tilting object $T$ in $\CC$ is of the form $T = \bigoplus_{i=1}^n T_i$, where the $T_i$ are mutually non-isomorphic indecomposable objects. Mutating $T$ at $T_i$ for some $1 \leq i \leq n$ yields a new cluster tilting object $T' = T/T_i \oplus T'_i$, to which we can associate a new 
frieze $F(T')$. 
In terms of the frieze, we can think of this mutation as a mutation at an entry of value $1$, 
namely the one sitting in the position of the indecomposable object $T_i[1]$.

We describe how, using graphic calculus, we can obtain each entry of the frieze $F(T')$ independently and directly from the frieze $F(T)$, thus effectively introducing the concept of mutations of friezes at entries of value $1$ that do not lie in the second or second-to-last row of $1$s. Our second main result in this paper, Theorem \ref{thm01} provides an explicit formula of how each entry in the frieze $F(T)$ changes under mutation at the entry corresponding to $T_i$. We observe that each frieze can be divided into four separate regions, relative to the entry of value $1$ at which we want to mutate. Each of these regions gets affected differently by mutation. Theorem \ref{thm01} provides an explicit formula, relying solely on the shape of the frieze and the entry at which we mutate, that determines how each entry of the frieze individually changes under mutation.

{\footnotesize
\tableofcontents 
}

\section{Description of the modules}\label{S:Description of the modules}
In this section we describe the objects in the cluster categories associated to the quivers $Q$ of type $A_n$, which will include quivers mutation equivalent to $A_n$. 
More precisely, we describe the structure of the indecomposable modules depending on their position in the Auslander-Reiten quivers of the cluster categories.\\

First we review the definition and some basic properties of the cluster category $\mathcal C_Q$ associated to an acyclic quiver $Q$ \cite{BMRRT}, and after that we consider quivers $Q'$ which are mutation equivalent to the quivers of type $A_n$ and which may have nontrivial potential $W$. In that case we consider generalized cluster categories $\mathcal C_{(Q',W)}$, which are shown to be triangle equivalent to $\mathcal C_Q$ \cite{[Am]}. 

\subsection{Acyclic quivers: cluster categories and AR quivers} 
$\ $

Let $Q=(Q_0,Q_1)$ be a finite quiver, with vertices $Q_0=\{1,\dots,n\}$ and arrows $Q_1$. Recall that for any finite quiver $Q$ with no oriented cycles, the cluster category $\mathcal C_Q$ is defined as the orbit category $\mathcal D^b(kQ)/\tau^{-1}[1]$ where  $\mathcal D^b(kQ)$ is the associated derived category of bounded complexes, $[1]$ is the shift functor in the triangulated category $\mathcal D^b(kQ)$ and $\tau$ is the Auslander-Reiten translation functor. It is possible to choose representatives of the indecomposable objects in $\mathcal C_Q$ to be the indecomposable $kQ$-modules and shifts of the indecomposable projective $kQ$-modules, i.e. $\ind \mathcal C_Q$ can be viewed as $\ind (kQ) \cup \{P_i[1]\}_{i=1}^n$; here $P_i$ is the projective cover of the simple $S_i$, and  simple $S_i$ is the module/representation supported at the vertex $i$.\\

The Auslander-Reiten (AR) quiver for $\modcat (kQ)$ is defined as: vertices correspond to 
(isoclasses of) 
indecomposable objects,  and arrows between vertices correspond to irreducible 
map between the associated objects (in general there might be several arrows, however for the quivers mutation equivalent to $A_n$ there will be at most one arrow); the mesh relations correspond to 
AR sequences in $\modcat (kQ)$. We now recall one of 
the basic theorems about almost split sequences, which will be used in both the 
acyclic and mutation equivalent to acyclic case.

\begin{theorem} \cite{AR-artin} 
Let $\Lambda$ be an artin algebra. \\
(a) Let $P$ be an indecomposable projective, not injective $\Lambda$-module and let $P_{t}$ be the  indecomposable projective modules such that $P$ is isomorphic to a direct summand of $\rad P_{t}$. Then there exists an almost split sequence (AR sequence):
$$0\rightarrow P \rightarrow \tau^{-1}(\rad P) \oplus (\oplus P_{t}) 
\rightarrow \tau^{-1}P\rightarrow0.$$
(b) Consider the  AR sequence,  where $B, C$ have no injective summands and $I$ is injective: 
$0\rightarrow A \rightarrow B\oplus I\rightarrow C\rightarrow0. \ \text{Then}$
 $0\rightarrow \tau^{-1}A \rightarrow \tau^{-1}B\oplus (\oplus P_t)\rightarrow \tau^{-1}C\rightarrow0$   
is an almost split sequence where the $P_t$ are the 
indecomposable projectives such that $\tau^{-1}A$ is isomorphic to a direct summand of $\rad P_t$.
\end{theorem}
As a special case, when $Q$ is acyclic quiver, and hence $kQ$ is hereditary algebra, we have a more precise description as follows.
If $P_i$ is projective, not injective, then there is an almost split sequence (AR sequence):
$$0\rightarrow P_i \rightarrow \tau^{-1}(\oplus_{t\in a(i)} P_t) \oplus (\oplus_{t'\in a'(i)} P_{t'}) \rightarrow \tau^{-1}P_i\rightarrow0,$$
where $a(i):=\{t\in Q_0\ |\ \exists (t\leftarrow i)\in Q_1\} \ \text{ and }\ a'(i):=\{t'\in Q_0\ |\ \exists (i\leftarrow t')\in Q_1\}.$\\

 The AR quiver for the cluster category $\mathcal C_Q$ is very closely related to the AR quiver of the module category $\modcat (kQ)$ in the following sense: all AR sequences of $kQ$-modules are still AR triangles in the cluster category $\mathcal C_Q$.
 The only new objects are $\{P_i[1]\}_{i=1}^n$ and the new AR triangles are the following. 
For each $i\in Q_0$ let 
$$I_i \rightarrow (\oplus_{t\in a(i)}  P_t[1])\oplus (\oplus_{t'\in a'(i)}  I_{t'})\rightarrow P_i[1] \rightarrow \ \ \ \ \ \ \ \ \ \ \ (*)$$
$$ P_i[1] \rightarrow (\oplus_{t\in a(i)}  P_t) \oplus (\oplus_{t'\in a'(i)}  P_{t'}[1]) \rightarrow P_i \rightarrow.\ \ \ \ \ \ \ \ \ \ \ (**)$$
\sloppy The triangles (*) are the connecting triangles in the derived category $D^b(kQ)$ between $\ind (kQ)$ and $\ind(kQ)[1]$, while the triangles (**) are the new triangles which appear in the orbit category. (The modules $I_i$ are injective envelopes of the simple modules $S_i$.) 

\subsection{Quivers with potential which are mutation equivalent to $A_n$}
$\ $

Cluster mutations of acyclic quivers, in general, do not produce another acyclic quiver, but instead quivers with potential are obtained. We will now describe the AR quiver of the cluster category of quivers with potential which are obtained by mutations of quivers of type $A_n$.\\

{\bf Generalized cluster categories $\mathcal C_{(Q,W)}$:} Given a quiver $Q$ with potential $W$, the generalized cluster category $\mathcal C_{(Q,W)}$ is defined as the quotient category $per \Gamma/D^b(\Gamma)$ where $\Gamma$ is the associated Ginzburg algebra which is a differential graded  algebra, $per \Gamma$ is the category of complexes of projective $\Gamma$-modules, and $D^b(\Gamma)$ is the category of differential graded $\Gamma$-modules with finite cohomology. While this description of this category is somewhat complicated, there is a beautiful theorem of Amiot, about the triangulated equivalence.

\begin{theorem} \cite{[Am]} 
Let $Q$ be an acyclic quiver. Let $(\mu Q,W)$ be the quiver with potential obtained after mutation $\mu$ of $Q$. Then the generalized cluster category $\mathcal C_{(\mu Q, W)}$ is triangle equivalent to the cluster category $\mathcal C_Q$.
\end{theorem}
This theorem tells us that the cluster category obtained after mutation still has the same AR quiver. 
However we need more precise description of the objects of the new cluster category. In order to do that we now recall the definition of mutation. \\

{\bf Mutations:} A mutation of a cluster-tilting object in the cluster category $\mathcal C_Q$ corresponds to the mutation of the quiver of the endomorphism algebra of the corresponding cluster-tilting objects. We recall that an object $T$ in the cluster category $\mathcal C_Q$ is called {\it cluster-tilting object} if 
$\Ext^1_{\mathcal C_Q}(T,T)=0$ and $T=\oplus_{j=1}^nT_j$, with $T_j$ indecomposable and pairwise 
non-isomorphic. Notice that $n=|Q_0|$.

\begin{defi} Let $T=\oplus_{j=1}^nT_j$ be a cluster-tilting object in a cluster category $\mathcal C_Q$. 
A {\it mutation of $T$ in direction $i$} is a new cluster-tilting object $\mu_iT= T/T_i\oplus T_i'$ where $T_i'$ is defined as the pseudokernel of the right $\add(T/T_i)$-approximation of $T_i$ or pseudocokernel of the left $\add(T/T_i)$-approximation of $T_i$: 
$$\rightarrow T_i'\rightarrow B\xrightarrow{f_i} T_i\rightarrow \ \ \ \ \  \text{ or } \ \ \ \ 
\rightarrow T_i\xrightarrow {f_i} B'\rightarrow T_i'\rightarrow.$$
\end{defi}
Cluster mutation can also be viewed as a mutation of quivers: for each cluster-tilting object $T$, let $B_T:= \End_{\mathcal C_Q}(T)$ and let $Q_T$ be the quiver of $B_T$. It follows from 
\cite[Theorem 5.1]{BIRS} that the quiver $Q_{\mu_i T}$ of 
$B_{\mu_iT}=\End_{\mathcal C_Q}(\mu_iT)$ can be obtained from the quiver $Q_T$ by applying DWZ-quiver mutation as in \cite{DWZ1}. 

\subsection{ AR quivers for $\mathcal C_{(Q,W)}$.} \label{sec:AR quiver}
$\ $

We now describe the AR quiver of the generalized cluster category $\mathcal C_{(Q, W)}$ 
which is obtained after a sequence of mutations. We need to use the following theorem. 

\begin{theorem} \cite{BMR} \label{T:KR} Let $\mathcal C_{(Q,W)}$ 
be a generalized cluster category. Let $T$ be a cluster-tilting object in $\mathcal C_{(Q,W)}$ and let $B_T:=\End_{\mathcal C_{(Q,W)}}(T)$. 
\begin{enumerate}
\item The functor $\Hom_{\mathcal C_{(Q,W)}}(T,-): \mathcal C_{(Q,W)} \to \modcat(B_T)$ induces an equivalence of categories $\mathcal C_{(Q,W)}/(\add T[1]) \cong \modcat(B_T)$, 
\item The kernel of $\Hom_{\mathcal C_{(Q,W)}}(T,-)$ is the subcategory $\add(T[1])\subset \mathcal C_{(Q,W)}$. 
\item The modules $\{\Hom_{\mathcal C_{(Q,W)}}(T,T_j)\}_{j=1}^n$ form a complete set of non-isomorphic indecomposable projective $B_T$-modules.
\end{enumerate}
\end{theorem}

Using the above equivalence, the AR quiver of $\modcat B_T$ can be viewed as a full sub\-quiver of the AR quiver of $\mathcal C_{(Q,W)}$, however we need the following precise functorial correspondence between the cluster categories $\mathcal C_{(Q,W)}$ and $\mathcal C_{(Q',W')}$ which are related by 
a (sequence of) mutations. 
We write $P_i$ for the indecomposable projectives of $\mathcal C_{(Q,W)}$ and $P_i'$ 
for the ones of $\mathcal C_{(Q',W')}$.

\begin{prop} \label{psi} Let $\mathcal C_{(Q,W)}$ 
and $\mathcal C_{(Q',W')}$ be  generalized cluster categories 
where $(Q',W')$ is obtained after a sequence of DWZ-mutations. 
Let $T$ be the cluster-tilting object obtained from $\oplus_{j=1}^nP_j$ after the same sequence of cluster-tilting mutations.
Then there is a triangulated functor $\psi$ 
making the following commutative diagram:
\[
\xymatrix{
\mathcal C_{(Q,W)}\ar[rr]^{\Hom_{\mathcal C_{(Q,W)}}(T,-)}& &\modcat(B_T)\\
\mathcal C_{(Q',W')} \ar[urr]_{\phantom{aaaaaa}\Hom_{\mathcal C_{(Q',W')}} (\oplus_{j=1}^n P_j',- )}\ar@{..>}[u]^ {\psi}
       }
       \]
       
\end{prop}
\begin{proof} We first prove the statement for a single DWZ-mutation $\mu_i$. Let $(Q',W')=\mu_i(Q,W)$ and $T=\mu_i(\oplus_{j=1}^n P_j)$. 
We need a triangulated functor $\psi_i$ so that 
$$
\Hom_{\mathcal C_{(Q',W')}} (\oplus P_j',- )=\Hom_{\mathcal C_{(Q,W)}} (T,-)\cdot\psi_i,
$$
i.e. such that $\psi_i(P_j')=P_j$ for $j\neq i$ 
and $\psi_i(P'_i)=cone(P_i\xrightarrow{f}\oplus P_t)$ where $f$ is a minimal $\add(\oplus_{j\neq i}P_j)$-approximation of $P_i$. Such a triangulated functor exists as a consequence of the 
Keller--Yang theorem \cite{KY}. \
They consider  the corresponding Ginzburg algebras $\Gamma$ and $\Gamma'$ 
and prove existence of a triangulated functor 
$\Psi_i :per\,\Gamma' \to per\,\Gamma$ such that $\Psi_i(\Gamma'e'_j)=\Gamma e_j$ for 
$j\neq i$ and 
$\Psi_i(\Gamma'e'_i)=cone(\Gamma e_i\xrightarrow{F} \oplus \Gamma e_t)$ where 
$F$ is a minimal $\add(\oplus_{j\neq i}\Gamma e_j)$-approximation of $\Gamma e_i$. It is also 
shown that $\Psi_i$ restricts to a functor $D^b(\Gamma')\to D^b(\Gamma)$, inducing the desired 
functor $\psi_i$ between the cluster categories.

In order to get the general statement apply the above argument to the given sequence of mutations.
\end{proof}

The AR quiver for the cluster category $\mathcal C_{A_n}$ and any generalized cluster category  mutation equivalent to $A_n$ is of the following form: 
\[ 
\xymatrix@!C=0pt@!R=0pt{
&&&\ar[rd]&&a_{1,n}\ar[rd]&&a_{2,n}\ar[rd]&&a_{1,1}\ar[rd]&&\\
&&\ar@{.}[rd]&&a_{1,n-1}\ar@{.}[rd]\ar[ru]&&a_{2,n-1}\ar[ru]\ar@{.}[rd]&&a_{3,n-1}\ar[ru]\ar@{.}[rd]&&a_{1,2}\ar@{.}[rd]\ar[ru]&&\\
\dots &\ar@{.}[rd]&&\ar@{.}[ur]\ar@{.}[dr]&&\ar@{.}[ur]\ar@{.}[dr]&&\ar@{.}[ur]\ar@{.}[dr]&&\ar@{.}[ur]\ar@{.}[dr]&&\ar@{.}[dr]\ar@{.}[ur]&&&\dots&&&&\\
\ar[rd]&&a_{1,3}\ar@{.}[ur]\ar[dr]&&a_{2,3}\ar[dr]\ar@{.}[ur] &&a_{3,3}\ar@{.}[ru]&&a_{n-2,3}\ar[rd]\ar@{.}[ru]&&a_{n-1,3}\ar[rd]\ar@{.}[ru]&&a_{1,n-2}\ar[rd]\ar@{.}[ur]&&\\
&a_{1,2}\ar[ur]\ar[dr]&&a_{2,2}\ar[ur]\ar[dr]&&a_{3,2}\ar[ur]\ar[dr] &&\dots&&a_{n-1,2}\ar[rd]\ar[ru]&&a_{n,2}\ar[ru]\ar[rd]&&a_{1,n-1}\ar[rd]\ar[ur]\\
a_{1,1}\ar[ur]&&a_{2,1}\ar[ur]&&a_{3,1}\ar[ru]&&&&\ar[ru]&&a_{n,1}\ar[ur]&&a_{n+1,1}\ar[ru]&&a_{1,n} \ .
}  \]

\begin{remark} \label{remark} Some properties of the AR quiver for the cluster categories $\mathcal C_{(Q,W)}$:
\begin{enumerate}
\item All of the objects encompassed by $\{a_{1,1}, \ a_{1,n}, \ a_{2,n}, \ a_{n+1,1}\}$ are mutually non-isomorphic and form a fundamental domain. 
\item \sloppy Each maximal rectangle starting at the point $a_{1,t}$ is bounded by the corners $\{a_{1,t}, \ a_{1,n}, \ a_{t,1}, \ a_{t,n-t+1}\}$ and is contained entirely within the fundamental domain, hence all the points are distinct.
\item Each maximal rectangle starting at any point $a_{s,t}$ can be viewed as 
starting at $a_{1,t}$
by relabelling, and hence all the points within the rectangle are distinct.
\item All the points, within any rectangle starting at any point, are distinct.
\item Additional copies of the same points of the AR quiver are included in the above diagram, in order to be able to see and describe supports of the functors 
$\Hom_{\mathcal C_{(Q,W)}}(-,- )$ and $\Ext^1_{\mathcal C_{(Q,W)}}(-,- )$.
\end{enumerate}
\end{remark}

\subsection{Supports of Hom- and Ext- functors} 
$\ $

From this point on, $\mathcal C$ will denote a generalized cluster category 
$\mathcal C_{(Q,W)}$ of type $A_n$.

Even though it is known that all indecomposable modules over cluster-tilted algebras of type 
$A_n$ are string modules (\cite{ASCJP,BR})
we need to describe the precise shape of indecomposable modules depending on their relative 
position to the projectives in the AR quiver of $\mathcal C$. Since the indecomposable 
projective modules have no (non-trivial) 
extensions, the possible configurations of the projectives can be determined using the  supports of the functors $\Ext^1_{\mathcal C}(P,-)$ which we now describe. 

\begin{prop} \label{dimHom1}Let $\mathcal C$ be a generalized cluster category of type $A_n$. Let $X, Y$ be indecomposable objects in $\mathcal C$. Then:
 \begin{enumerate}
\item  $\dim_K\Hom_{\mathcal C}(X,Y)\in\{0,1\}$ 
\item  $\dim_K\Ext^1_{\mathcal C}(X,Y)\in\{0,1\}$.
\end{enumerate}
\end{prop}
\begin{proof} (1) Since cluster-tilted algebras of type $A_n$ are all of finite representation type, all the morphisms are linear combinations of compositions of irreducible maps and dimensions of the homomorphism spaces can be obtained  from the AR quiver. We may let $a_{1,t}$ be the point corresponding to the object $X$.  Using the fact that 
$$0\xrightarrow{}\Hom_{\mathcal C}(X,A)\xrightarrow{}\Hom_{\mathcal C}(X,B)\xrightarrow{}\Hom_{\mathcal C}(X,C)\xrightarrow{}0 \ \ \ \ \ \ \ \ (*)$$
is exact  for each  indecomposable object $X\ncong C$ and each AR-triangle 
$$\xrightarrow{}A\xrightarrow{}B\xrightarrow{}C\xrightarrow{}A[1]\xrightarrow{}\ \ \ \ \ \ \ (**)$$  
it is possible to compute the dimensions of $\Hom_{\mathcal C}(X,Y)$ for all indecomposable $Y$.
The only non-zero homomorphisms, from the object at any point, are morphisms to the objects in the maximal rectangle starting at that point. This follows from (*) and the fact that in the AR triangles (**) the object $B$ is indecomposable if $A$ is on the edge of the AR quiver, 
and otherwise $B\cong B_1\oplus B_2$ with $B_1,B_2$ both indecomposable.

Since all the points in the rectangle are distinct by Remark \ref{remark} (3), they correspond to distinct objects and the $\dim_K\Hom_{\mathcal C}(X,Y)$ is a rectangle of 1's. For example if the object $X$ is at the point $a_{1,3}$ then the $\Hom_{\mathcal C}(X,- )$  has dimensions:

$$\xymatrix@!C=0pt@!R=0pt{
&0\ar@{.}[rd]&&0\ar@{.}[rd]&&1\ar[rd]&&0\ar@{.}[rd]&&0\ar@{.}[rd]&&0\ar@{.}[rd]&&1\ar[rd]&&0&&\\
0\ar@{.}[ur]\ar@{.}[rd]&&0\ar@{.}[rd]\ar@{.}[ur]&&1\ar[rd]\ar[ru]&&1\ar@{.}[ru]\ar[rd]&&0\ar@{.}[ru]\ar@{.}[rd]&&0\ar@{.}[rd]\ar@{.}[ru]&&1\ar[ru]\ar[rd]&&1\ar[rd]\ar@{.}[ru]&&&&\\
&0\ar@{.}[ur]\ar@{.}[rd]&&1\ar[ur]\ar[dr]&&1\ar[dr]\ar[ur] &&1\ar@{.}[ru]\ar@{.}[rd]&&0\ar@{.}[rd]\ar@{.}[ru]&&\boxed{1_X}\ar[rd]\ar[ru]&&1\ar[rd]\ar[ur]&&1&&\\
0\ar@{.}[ur]\ar@{.}[rd]&&1\ar[ur]\ar@{.}[dr]&&1\ar[ur]\ar[dr]&&1\ar[ur]\ar@{.}[dr]&&0\ar@{.}[ur]\ar@{.}[dr]&&0\ar@{.}[dr]\ar@{.}[ur]\ar@{.}[rd]&&1\ar[rd]\ar[ru]&&1\ar[rd]\ar[ru]&&&\\
&1\ar[rd]\ar[ru]&&1\ar[ru]\ar[rd]&&1\ar[ru]\ar@{.}[rd]&&0\ar@{.}[ur]\ar@{.}[rd]&&0\ar@{.}[ru]\ar@{.}[rd]&&0\ar@{.}[ru]\ar@{.}[rd]&&\ar[ru]\ar[rd]1&&1&&&&&&&&\\
\boxed{1_X}\ar[ur]\ar[dr]&&1\ar[dr]\ar[ur] &&1\ar[ru]\ar@{.}[rd]&&0\ar@{.}[rd]\ar@{.}[ru]&&0\ar@{.}[rd]\ar@{.}[ru]&&0\ar@{.}[rd]\ar@{.}[ur]&&0\ar@{.}[ur]\ar@{.}[rd]&&1\ar[rd]\ar[ur]&&&\\
&1\ar[ur]\ar[dr]&&1\ar[ur]\ar@{.}[dr] &&0\ar@{.}[rd]\ar@{.}[ru]&&0\ar@{.}[rd]\ar@{.}[ru]&&0\ar@{.}[ru]\ar@{.}[rd]&&0\ar@{.}[rd]\ar@{.}[ur]&&0\ar@{.}[ur]\ar@{.}[rd]&&1\\
0\ar@{.}[ur]&&1\ar[ru]&&0\ar@{.}[ur]&&0\ar@{.}[ru]&&0\ar@{.}[ur]&&0\ar@{.}[ru]&&\ar@{.}[ur]0&&0\ar@{.}[ur]&&.
}$$

\centerline{$\dim_K \Hom_{\mathcal C}(X,- )$}

\sloppy The point corresponding to the object $X$ is indicated with $\boxed{1_X}$ and it stands 
for $\dim_K\Hom_{\mathcal C}(X,X)=1$. 
The other point labeled by $\boxed{1_X}$ 
corresponds to the same object $X$ but it appears again since a covering of the AR 
quiver is drawn.

(2) In order to prove the statement that $\dim_K\Ext^1_{\mathcal C}(X,Y)\in \{0,1\}$, 
it is enough to use the isomorphism 
$\Ext^1_{\mathcal C}(X,Y)\cong D\Hom_{\mathcal C}(\tau^{-1}Y,X)$.  
\end{proof}

\begin{cor} \label{supports} Let $\mathcal C$ be a generalized cluster category of type $A_n$. 
Let $X$ be an indecomposable object in $\mathcal C$. Then:
\begin{enumerate}
\item The support of $\Hom_{\mathcal C}(X,- )$ consist of the objects corresponding to the points 
in the maximal rectangle starting at the point corresponding to $X$. Furthermore, the last point in 
the support of   $\Hom_{\mathcal C}(X,-)$ corresponds to $\tau^2X$.
\item 
The support of $\Hom_{\mathcal C}(-,X)$ consist of the objects corresponding to the points 
in the maximal rectangle ending at the point corresponding to $X$.  Furthermore, the first point 
in the support of   $\Hom_{\mathcal C}(-,X)$ corresponds to $\tau^{-2}X$.
\item 
The support of $\Ext^1_{\mathcal C}(- ,X)$ consist of the objects corresponding to the points in 
the maximal rectangle starting at the point corresponding to $\tau^{-1}X$.  Furthermore, the last 
point in the support of $\Ext^1_{\mathcal C}(- ,X)$ corresponds to $\tau X$.
\item 
The support of $\Ext^1_{\mathcal C}(X,-)$ consist of the objects corresponding to the points in 
the maximal rectangle ending at the point corresponding to $\tau X$. Furthermore, the first point 
in the support of   $\Ext^1_{\mathcal C}(X,-)$ corresponds to $\tau^{-1} X$.
\item 
The support of $\Ext^1_{\mathcal C}(- ,X)$ is the same as the support of 
$\Ext^1_{\mathcal C}(X,-)$ .
\end{enumerate}
\end{cor}

\begin{proof} (1) and (2)  follow from the above diagram of $\dim \Hom_{\mathcal C}(X,- )$ 
on the cover of AR quiver.
(3), (4) and (5) follow from the following diagram of the $\dim_K\Ext^1_{\mathcal C}(X,Y)$.
$$\xymatrix@!C=0pt@!R=0pt{
&0\ar@{.}[rd]&&0\ar@{.}[rd]&&0\ar@{.}[rd]&&1\ar[rd]&&0\ar@{.}[rd]&&0\ar@{.}[rd]&&0\ar@{.}[rd]&&1\ar[rd]&&0\\
0\ar@{.}[rd]\ar@{.}[ur]&&0\ar@{.}[ur]\ar@{.}[rd]&&0\ar@{.}[rd]\ar@{.}[ur]&&1\ar[rd]\ar[ru]&&1\ar@{.}[ru]\ar[rd]&&0\ar@{.}[ru]\ar@{.}[rd]&&0\ar@{.}[rd]\ar@{.}[ru]&&1\ar[ru]\ar[rd]&&1\ar[rd]\ar@{.}[ru]&&\\
&0\ar@{.}[ur]\ar@{.}[rd]&&0\ar@{.}[ur]\ar@{.}[rd]&&1\ar[ur]\ar[dr]&&1\ar[dr]\ar[ur] &&1\ar@{.}[ru]\ar@{.}[rd]&&\boxed{0_X}\ar@{.}[rd]\ar@{.}[ru]&&1\ar[rd]\ar[ru]&&1\ar[rd]\ar[ur]&&1\\
0\ar@{.}[ur]\ar@{.}[rd]&&0\ar@{.}[rd]\ar@{.}[ru]&&1\ar[ur]\ar@{.}[dr]&&1\ar[ur]\ar[dr]&&1\ar[ur]\ar@{.}[dr]&&0\ar@{.}[ur]\ar@{.}[dr]&&0\ar@{.}[dr]\ar@{.}[ur]\ar@{.}[rd]&&1\ar[rd]\ar[ru]&&1\ar[rd]\ar[ru]&&&\\
&0\ar@{.}[rd]\ar@{.}[ru]&&1\ar[rd]\ar[ru]&&1\ar[ru]\ar[rd]&&1\ar[ru]\ar@{.}[rd]&&0\ar@{.}[ur]\ar@{.}[rd]&&0\ar@{.}[ru]\ar@{.}[rd]&&0\ar@{.}[ru]\ar@{.}[rd]&&\ar[ru]\ar[rd]1&&1&&&&&&&&\\
\boxed{0_X}\ar@{.}[rd]\ar@{.}[ur]&&1\ar[ur]\ar[dr]&&1\ar[dr]\ar[ur] &&1\ar[ru]\ar@{.}[rd]&&0\ar@{.}[rd]\ar@{.}[ru]&&0\ar@{.}[rd]\ar@{.}[ru]&&0\ar@{.}[rd]\ar@{.}[ur]&&0\ar@{.}[ur]\ar@{.}[rd]&&1\ar[rd]\ar[ur]\\
&0\ar@{.}[ur]\ar@{.}[rd]&&1\ar[ur]\ar[dr]&&1\ar[ur]\ar@{.}[dr] &&0\ar@{.}[rd]\ar@{.}[ru]&&0\ar@{.}[rd]\ar@{.}[ru]&&0\ar@{.}[ru]\ar@{.}[rd]&&0\ar@{.}[rd]\ar@{.}[ur]&&0\ar@{.}[ur]\ar@{.}[rd]&&1\\
0\ar@{.}[ur]&&0\ar@{.}[ur]&&1\ar[ru]&&0\ar@{.}[ur]&&0\ar@{.}[ru]&&0\ar@{.}[ur]&&0\ar@{.}[ru]&&\ar@{.}[ur]0&&0\ar@{.}[ur].
}$$
\centerline{$\dim_K \Ext^1_{\mathcal C}(X,-)$}
\sloppy
The point corresponding to the object $X$ is indicated with  $\boxed{0_X}$ and it stands for
$\dim_K\Ext^1_{\mathcal C}(X,X)=0$. The other point labeled by $\boxed{0_X}$ 
corresponds to the same object $X$ but it appears again since a covering of the AR quiver is drawn.
\end{proof}

\subsection{Configurations and structure of projectives} 

$\ $

Since the indecomposable projective modules have no extensions, the possible configurations 
of the projectives can be determined using the supports of the functors $\Ext^1_{\mathcal C}(P,- )$, 
i.e. avoiding maximal rectangles starting at $\tau^{-1}P$, which are the same as maximal 
rectangles ending at $\tau P$ by Corollary \ref{supports} (3),(4),(5).

In order to describe modules using the AR quiver and configurations of projectives, we first recall 
some general facts for finite dimensional algebras.

\begin{remark} \label{simples} Let $\Lambda$ be a finite dimensional $K$-algebra. 
The multiplicity of the simple $S_i$ as composition factor of a module $M$ is equal to 
the length of $\Hom_{\Lambda}(P_i, M)$ as an $\End(P_i)^{op}$-module, 
which in our case is equal to $\dim_K\Hom_{\Lambda}(P_i,M)$ 
since $\End(P_i)^{op}\cong K$ for all indecomposable projective modules $P_i$. 
\end{remark}

\begin{defi} Let $P_i, P_j$ be indecomposable projectives. A non-zero homomorphism 
$\rho:P_i\to P_j$ is called {\it projectively irreducible} if it is not an isomorphism and 
for any factorization $\rho = \beta\alpha$ with $\alpha:  P_i\to Q$ and $\beta: Q\to P_j$ where 
$Q$ is projective, one of the following holds: either $\alpha$ is split monomorphism or $\beta$ is 
split epimorphism.
\end{defi}

\begin{lemma} Let $\Lambda$ be a finite dimensional algebra and $\rho: P_i\to P_j$ a 
projectively irreducible map. Then 
$\img \rho\not\subset \rad^2P_j$.
\end{lemma}

\begin{proof}
If $\img \rho$ were contained in $\rad^2P_j$ there would be another projective in between the 
two, hence $\rho$ would not be projectively irreducible. 
\end{proof}

\begin{lemma} Let $\Lambda$ be a finite dimensional algebra and $\rho: P_i\to P_j$  
projectively irreducible. Then 
\begin{enumerate}
\item  There exists a non-split sequence
$0 \to S_i\to Z \to S_j \to 0$. 
\item There exists an epimorphism $\Psi: P_j\to Z$.
\end{enumerate}
\end{lemma}
\begin{proof} (1) Since $\rho$ is not an isomorphism, there is $\rho': P_i\to \rad P_j$ 
such that $\rho=a\rho'$ where $a: \rad P_j\to P_j$ is the inclusion. 
Since $\img\rho\not\subset \rad^2P_j$, the composition 
$P_i\xrightarrow{\rho'}\rad P_j\xrightarrow{\pi}\rad P_j/\rad^2P_j$ is non-zero with 
$\img(\pi \rho')\cong S_i$. 
Let 
$\varphi: \rad P_j\to S_i$ be the induced non-zero map. Now use the following exact 
sequence and push-out diagram to define $Z$ and the morphism $\Psi: P_j \to Z$:
$$
\xymatrix@!C=0pt@!R=5pt{
&0&\ar[r]&&\rad P_j\ar[d]^{\varphi}&\ar[r]^a&&P_j\ar[d]^{\Psi}&\ar[r]&&S_j\ar@{=}[d]&\ar[r]&&0\\
&0&\ar[r]&&S_i&\ar[r]&&Z&\ar[r]&&S_j&\ar[r]&&0.\\
}
$$

It follows from the diagram that $\Psi$ is epimorphism.
\end{proof}

From now on we concentrate on the generalized cluster categories of type $A_n$ and 
associated cluster-tilted algebras of type $A_n$.
\begin{lemma} Let $B$ be a cluster-tilted algebra of type $A_n$. Let $M$ be an indecomposable 
$B$-module. Then multiplicity of each simple composition factor of $M$ is $0$ or $1$.
\end{lemma}
\begin{proof} It follows from Remark \ref{simples} that the multiplicity of the simple $S_i$ in 
$M$ is equal to $\dim_k\Hom_B(P_i,M)$ and from Proposition \ref{dimHom1} that it is equal to 1.
\end{proof}

We recall, that a path $M_0\to M_1 \to \dots \to M_s$ in the AR quiver is called 
\emph{sectional} if $\tau M_{i+1} \not = M_{i-1}$ for all $i=1, \dots , s-1$. 
A maximal sectional path is a sectional path which is not a proper subpath of any other 
sectional path. With this definition consider the following results.  

\begin{lemma} \label{leq2} Let $B$ be a cluster-tilted algebra of type $A_n$. Let $P$ be an indecomposable projective $B$-module. Then:
\begin{enumerate}
\item All indecomposable projective $B$-modules which are in the support of $\Hom_B(-,P)$ are 
on the maximal sectional paths to $P$. 
\item There are at most two projectively irreducible maps $\rho: P_j\to P$.
 \end{enumerate}
\end{lemma}

\begin{proof} (1) By Corollary \ref{supports} the support of $\Hom_B(- ,P)$ is the maximal rectangle ending at $P$. Also the support of $\Ext_B^1(P,-)$ is the maximal rectangle ending at $\tau P$. 
Since $\Ext_B^1(P,P_i)=0$ it follows that all the indecomposable projective $B$-modules 
which map to $P$ must be on the maximal sectional paths to $P$.

(2) If $\rho: P_j\to P$ is projectively irreducible, then $P_j$ is on a sectional path 
ending at $P$ and it is the closest projective (on that path) to $P$. Since there are at most two different 
sectional paths ending at $P$, the result follows.
\end{proof}

\begin{lemma} \label{r/r2} Let $f: P_i\to P$ be a non-isomorphism with $P$ indecomposable 
such that the induced map 
$P_i\to \rad P/\rad^2P$ is non-zero. Then $f$ is projectively irreducible.
\end{lemma}

\begin{proof} Consider a factorization of $f$ as $f=\beta\alpha$ through an indecomposable projective $Q$. 
We need to show that either $\beta$ is an isomorphism or $\alpha$ is an isomorphism. 
If $\beta$ is not isomorphism then it factors through $\rad P$, 
hence there is $\gamma$ such that $\beta = a\gamma$. Then 
$f =\beta\alpha=a\gamma\alpha$ and $f = a\rho'$ imply $a\gamma\alpha=a\rho'$ (recall that 
$a$ is the inclusion $\rad P\to P$). 
Since $a$ is a monomorphism, this implies $\gamma\alpha=\rho'$. By assumption $\pi\rho'\neq0$ and hence $\pi\gamma\alpha\neq0$. This, together with the fact that $\rad P/\rad^2P$ is semisimple, imply  that $\img\alpha\not\subset \rad Q$. Therefore $\alpha$ is an isomorphism. So $f$ is projectively irreducible.
$$\xymatrix@!C=0pt@!R=0pt{
&&&&Q\ar[rrd]^{\beta}\ar[dd]^{\gamma}&&&&&&\\
&&P_i\ar[rrddd]_{\pi\rho'}\ar[rru]^{\alpha}\ar[rrrr]^{f \ \ \ \ \ \ \ \ \ \ \ \ \ \ \ }\ar[rrd]_{\ \ \ \ \ \rho'}&&&&P&\\
&&&&\rad P\ar[dd]^{\pi}\ar[rru]_a&&&&&\\
&&&&&&&&&&&&\\
&&&&\rad P/\rad^2P&&
}$$
\end{proof}

\begin{prop} Let $B$ be a cluster-tilted algebra of type $A_n$. Let $P$ be an indecomposable projective.  Then $\rad P$ is either indecomposable or a direct sum of two indecomposables.
\end{prop}
\begin{proof} Let $\rad P/\rad^2P=\oplus S_i$. Then the maps $\rho_i: P_i\to P$ are projectively irreducible by Lemma \ref{r/r2} and there are at most two such maps by Lemma \ref{leq2} (2). If there is only one such $\rho_i: P_i\to P$ then $\rad P=\img\rho_i$ and hence indecomposable. 

If there are two such maps $\rho_i: P_i\to P$ and $\rho_{j}: P_{j}\to P$ then the map 
$P_i\oplus P_j\xrightarrow{\rho_i',\,\rho_j'}\rad P$ is an epimorphism 
and to see that $\rad P=\img\rho_i'\cap\rho_j'$ we show that 
$\img\rho_i'\cap \img\rho_j'=0$. \\
Suppose there is a simple $S\subset \img\rho_i'\cap \img\rho_j'$. Then there are maps $P(S)\xrightarrow{\xi_i}P_i$ and $P(S)\xrightarrow{\xi_j}P_j$ such that 
$\rho_i\xi_i\neq0$ and $\rho_j\xi_j\neq0$ (where $P(S)$ is the projective cover of $S$). 
Therefore $P(S)$ is on both sectional paths, which is impossible if they are 
distinct sectional paths to $P$. Therefore $\img\rho_i'\cap \img\rho_j'=0$ and hence $\rad P= \img\rho_i'\oplus \img\rho_j'$.
\end{proof}

\begin{cor} \label{summand} Let $\rho_j: P_j\to P$ be a projectively irreducible map. 
Then $\img\rho_j$ is a direct summand of $\rad P$.  
\end{cor}

\begin{remark} \label{S1} Let $P_1\xrightarrow{\rho_1}P_2\xrightarrow{\rho_2}\dots \xrightarrow{}P_{j-1}\xrightarrow{\rho_{j-1}}P_j$ be a sectional sequence 
of projectively irreducible maps. Then $\rho_{t}\dots\rho_1\neq0$ and 
$\img\rho_{t}\dots\rho_1\subset \img\rho_t$ for all $t=1,\dots,j-1$. 
In particular $S_1\subset \img\rho_t$ for all $t=1, \dots, j-1$.
\end{remark}

\begin{prop} \label{Z} Let $B$ be a cluster-tilted algebra of type $A_n$.  Let $P_1\xrightarrow{\rho_1}P_2\xrightarrow{\rho_2}\dots \xrightarrow{}P_{t-1}\xrightarrow{\rho_{t-1}}P_t$ be a sectional sequence of projectively irreducible maps. Then:
\begin{enumerate}
\item There exists a uniserial module $Z^{(t)}$ with composition factors $S_1, S_2,\dots, S_t$, in this order, with $\soc Z^{(t)}=S_1$.
\item There is an epimorphism $\Psi_t: P_t \to Z^{(t)}$.
\end{enumerate}
\end{prop}

\begin{proof} This will be done by induction. \\
($t=1$) $Z^{(1)}=S_1$ and there is an epimorphism $\Psi_1: P_1 \to S_1$.\\
By construction and by induction hypothesis, there is an exact sequence and the following maps:
$$\xymatrix@!C=0pt@!R=0pt{
&0\ar[r]&&\ker\rho_t&\ar[r]&&P_t\ar[dd]^{\Psi_t}&\ar[r]^{\overline \rho_t}&&\img\rho_t\ar@{.>}[llldd]^{\exists \varphi_t}&\ar[r]&&0\\
&&&&&&&&&&&&&&&&&&&&&&&&&&&&&\\
&&&&&&Z^{(t)}&&&&&&\\
}$$
By induction hypothesis $Z^{(t)}$ is uniserial with $\soc Z^{(t)}=S_1$. Consider the induced exact sequence:
$$0\xrightarrow{} \Hom_B(P_1,\ker\rho_t)\xrightarrow {}\Hom_B(P_1,P_t)\xrightarrow{}\Hom_B(P_1,\img\rho_t)\xrightarrow{}0.$$
Since $S_1\subset \img\rho_t$ by Remark \ref{S1} and $\img\rho_t$ is indecomposable, it follows by Proposition \ref{dimHom1} that $\dim_K\Hom_B(P_1,\img\rho_t)=1$. Similarly we have $\dim_K\Hom_B(P_1,P_t)=1$.
Therefore $\dim_K\Hom_B(P_1,\ker\rho_t)=0$. Hence $\ker\rho_t$ does not have $S_1$ as composition factor and therefore $\Hom_B(\ker\rho_t, Z^{(t)})=0$. Therefore $\Psi_t$ factors through $\img\rho_t$, i.e. there exists $\varphi_t$ such that $\Psi_t= \varphi_{t}{\overline\rho_t}$. Consequently $\varphi_t$ is an epimorphism.

Since $\img\rho_t$ is a direct summand of $\rad P_{t+1}$, the composition $\varphi_t p$ as in
$\rad P_{t+1}\xrightarrow{p} \img\rho_t\xrightarrow{\varphi_t}Z^{(t)}$, is an epimorphism. Consider the following exact sequence and define $Z^{(t+1)}$ and $\Psi_{t+1}$ using the push-out diagram: 
$$\xymatrix@!C=0pt@!R=0pt{
&0&\ar[r]&&\rad P_{t+1}\ar[dd]^{\varphi_tp}&\ar[r]^a&&P_{t+1}\ar[dd]^{\Psi_{t+1}}&\ar[r]&&S_{t+1}\ar@{=}[dd]&\ar[r]&&0\\
&&&&&&&&&&&&&&&&&&&&&&&&&&&&&&&&&&&&&&&&&&&&&&\\
&0&\ar[r]&&Z^{(t)}&\ar[r]&&Z^{(t+1)}&\ar[r]&&S_{t+1}&\ar[r]&&0. &&(*_t)\\
}$$

Since $\varphi_t p$ is an epimorphism it follows that $\Psi_{t+1}$ is epimorphism. Therefore $Z^{(t+1)}$ is indecomposable and uniserial with composition factors $S_1,S_2,\dots,S_{t+1}$.
\end{proof}

\begin{cor}  Let $P$ be indecomposable projective and let 
\[
\begin{array}{ccccccccccc}
P_1 & \xrightarrow{\rho_1} & P_2 & \xrightarrow{\rho_2} & \dots & \xrightarrow{\rho_{i-2}} & P_{i-1}& \xrightarrow{\rho_{i-1}} & P_i  &\xrightarrow{\rho_{i}}& P\\
Q_1 & \xrightarrow{\xi_1} & Q_2 & \xrightarrow{\xi_2}& \dots & \xrightarrow{\xi_{j-2}}& Q_{j-1}& \xrightarrow{\xi_{j-1}} & Q_j &\xrightarrow{\xi_{j}} & P
\end{array}
\]
be two sectional sequences of projectively irreducible maps. Then there is a quotient $V$ of $P$ such that $\rad V\cong U_1\oplus U_2$ where $U_1$ and $U_2$ 
are uniserial modules with composition factors
$S_1,S_2,\dots,S_{i}$ and $R_1,R_2,\dots,R_{j}$ where $S_t=P_t/\rad P_t$ for $t=1,\dots, i$ 
and 
$R_t=Q_t/\rad Q_t$ for $t=1,\dots,j$ (resp.). 
\end{cor}

\begin{proof} Let $Z^{(i)}$ and $W^{(j)}$ be the uniserial modules and 
$\img\rho_i\xrightarrow{\varphi_i}Z^{(i)}$ and $\img\xi_i\xrightarrow{\eta_j}W^{(j)}$ epimorphisms as constructed in Proposition \ref{Z}. Since $\rad P\cong \img\rho_i\oplus \img\xi_j$ from Corollary \ref{summand}, using the push-out diagram, the module $V$ is defined:
$$\xymatrix@!C=0pt@!R=0pt{
&0&\ar[r]&&\rad P\ar[dd]^{[\varphi_{i-1},\eta_{j-1}]^t}&\ar[r]^a&&P\ar[dd]^{\Psi}&\ar[r]&&P/\rad P\ar@{=}[dd]&\ar[r]&&0\\
&&&&&&&&&&&&&&&&&&&&&&&&&&&&&&&&&&&&&&&&&&&&&&\\
&0&\ar[r]&&Z^{(i)}\oplus W^{(j)}&\ar[r]&&V&\ar[r]&&P/\rad P&\ar[r]&&0. &&
}$$
\end{proof}

\begin{cor} Let $M$ be a module and $f: P\to M$ a non-zero morphism. Let 
\[
\begin{array}{ccccccccccc}
P_1 & \xrightarrow{\rho_1} & P_2 & \xrightarrow{\rho_2} & \dots & \xrightarrow{\rho_{i-2}} & P_{i-1}& \xrightarrow{\rho_{i-1}} & P_i & \xrightarrow{\rho_{i}} & P\\
Q_1 & \xrightarrow{\xi_1} & Q_2 & \xrightarrow{\xi_2}& \dots & \xrightarrow{\xi_{j-2}}& Q_{j-1}& \xrightarrow{\xi_{j-1}} & Q_j & \xrightarrow{\xi_{j}} & P
\end{array}
\]
be two maximal sectional sequences of projectively irreducible maps such that $f\rho_i\dots\rho_1\neq0$
and $f\xi_j\dots\xi_1\neq0$. Then:
\begin{enumerate}
\item  $\img f\cong V$ where $V$ is a quotient of $P$ and  $\rad V\cong U_1\oplus U_2$ where $U_1$ and $U_2$ are uniserial modules with composition factors
$S_1,S_2,\dots,S_{i}$ with $S_t=P_t/\rad P_t$ for $t=1,\dots, i$ 
and $R_1,R_2,\dots,R_{j}$ where $R_t=Q_t/\rad Q_t$ for $t=1,\dots,j$. 
\item $\soc \img f\cong S_1\oplus R_1$.
\item If there is only one sectional sequence of projectively irreducible maps, then $\img f$ is uniserial and $\soc \img f\cong S_1$.
\end{enumerate}
\end{cor}

\begin{lemma} Let 
$P_1\xrightarrow{\rho_1}P_2\xrightarrow{\rho_2}\dots 
\xrightarrow{\rho_{i-1}}P_i\xrightarrow{\rho_{i}}P$ and  
$Q_1\xrightarrow{\xi_1}Q_2\xrightarrow{\xi_2}\dots 
\xrightarrow{\xi_{j-1}}Q_j\xrightarrow{\xi_{j}}Q$ 
be two  sectional sequences of projectively irreducible maps. If there is a common projectively irreducible map, then either $\Hom_B(P,Q)\neq0$ or $\Hom_B(Q,P)\neq0$.
\end{lemma}
\begin{proof} If there is a common projectively irreducible map (up to constant) then the sequences are on the same sectional path and the result follows.
\end{proof}

\begin{lemma}\label{lm:simple-image}
Let $M$ be indecomposable, let $f: P\to M$ and $g:Q\to M$ be summands of the projective cover of $M$. 
Assume $P\not\cong Q$. If $\img f\cap \img g\neq0$ then 
$\img f\cap \img g$ is a simple module.
\end{lemma}
\begin{proof} 
Consider maximal sectional sequences of projectively irreducible maps to $P$ and $Q$ composed with $f$ and $g$ inside the support of $\Hom_B(- ,M)$:\\
$P_1\xrightarrow{\rho_1}P_2\xrightarrow{\rho_2}\dots 
\xrightarrow{\rho_{i-1}}P_i\xrightarrow{\rho_{i}}P\xrightarrow{f}M\ \ \ \ \ $ and 
$\ \ \ \ \ \ P'_1\xrightarrow{\rho'_1}P'_2\xrightarrow{\rho'_2}\dots 
\xrightarrow{\rho'_{i'-1}}P'_{i'}\xrightarrow{\rho'_{i'}}P\xrightarrow{f}M$,  \\ 
$Q_1\xrightarrow{\xi_1}Q_2\xrightarrow{\xi_2}\dots 
\xrightarrow{\xi_{j-1}}Q_j\xrightarrow{\xi_{j}}Q\xrightarrow{g}M\ \ \ \ \ $ 
and
$\ \ \ \ \ Q'_1\xrightarrow{\xi'_1}Q'_2\xrightarrow{\xi'_2}\dots 
\xrightarrow{\xi'_{j'-1}}Q'_{j'}\xrightarrow{\xi'_{j'}}Q\xrightarrow{g}M$.  \\
Then $\soc\img f\cong S_1\oplus S'_1$ and $\soc \img g\cong R_1\oplus R'_1$.  If $\img f\cap \img g\neq0$ then there are two cases:\\
\underline{Case (1):} $\soc\img f =  \soc\img g$ in which case $S_1\cong R_1$ and $S'_1\cong R'_1$. Since $P\not\cong Q$ this would require four distinct sectional paths within the rectangle of support of $\Hom_B(- ,M)$ which is impossible.\\
\underline{Case (2):} $S_1\cong R_1$ and $S'_1\not\cong R'_1$. Since $P\not\cong Q$ and  they 
are both part of projective cover of $M$ 
it follows that $\Hom_B(P,Q)=0$ or $\Hom_B(Q,P)=0$ and therefore the two sectional 
paths from $S_1\cong R_1$ must be different. Hence $\img f\cap \img g\cong S_1\cong R_1$ which is simple.
\end{proof}

\begin{Construction}\label{Construction:mods}
{\bf Procedure for describing modules.} \label{procedure} Let $\mathcal C$ be generalized cluster category of type $A_n$ and let $B$ be the corresponding cluster tilted algebra. Let $M$ be an indecomposable $B$-module. Then the structure of $M$ can be described in the following way using the AR-quiver:
\begin{enumerate}
\item  Consider the rectangle in the AR-quiver ending at the point corresponding to $M$, i.e. this is the support of $\Hom_{\mathcal C}(\ ,M)\supset \Hom_{B}(\ ,M)$.
\item  Consider all indecomposable projective $B$-modules which appear in this rectangle 
\footnote{Note that all projectives in $\Hom_{\mathcal C}(-,M)$ are also in $\Hom_B(-,M)$.}. 
\item Chose one of the two directions of the sectional paths.
\item With the chosen direction, find the first sectional path within support of $\Hom_{\mathcal C}(\ ,M)\supset \Hom_{B}(\ ,M)$ which contains any projective $B$-modules (order on the sectional paths is given by  $t$ if the path passes through $\tau^t M$. 
\item  Let $\ \ P_{1,1}\xrightarrow{\rho_{1,1}}P_{1,2}\xrightarrow{\rho_{1,2}}\dots 
\xrightarrow{}P_{1,k_1-1}\xrightarrow{\rho_{1,k_1-1}}P_{1,k_1}\xrightarrow{\rho_{1,k_1}}P_{2,1}\ \ $
 be a maximal sequence of projectively irreducible maps on this sectional path.

\item Consider the second sectional path to $P_{2,1}$ (if it exists).\\
Let $\ \ P_{3,1}\xrightarrow{\rho'_{2,k_2}}P_{2,k_2}\xrightarrow{\rho'_{2,k_2-1}}\dots 
\xrightarrow{}P_{2,3}\xrightarrow{\rho'_{2,2}}P_{2,2}\xrightarrow{\rho'_{2,1}}P_{2,1}\ \ $
 be a maximal sequence of projectively irreducible maps on this sectional path.
\item Consider the second sectional path out of $P_{3,1}$ (if it exists).\\
 Let $\ \ P_{3,1}\xrightarrow{\rho_{3,1}}P_{3,2}\xrightarrow{\rho_{3,2}}\dots 
\xrightarrow{}P_{3,k_3-1}\xrightarrow{\rho_{3,k_3-1}}P_{3,k_3}\xrightarrow{\rho_{3,k_3}}P_{4,1}\ \ $
 be a maximal sequence of projectively irreducible maps on this sectional path.
\item Continue this way until there are no more projective $B$-modules in the support of 
$\Hom_{\mathcal C}(\ ,M)\supset \Hom_{B}(\ ,M)$.
This procedure must stop since there are only finitely many projectives and each projective can appear only once by Proposition \ref{dimHom1}.
\item \label{sequences} At the end, one obtains the following maximal sequences of projectively irreducible maps for even $s\in \{2,4,\dots r\}$,
in two directions along the AR-quiver: \\
$\ \ P_{s-1,1}\xrightarrow{\rho_{s-1,1}}P_{s-1,2}\xrightarrow{\rho_{s-1,2}}\dots 
\xrightarrow{}P_{s-1,k_{s-1}}\xrightarrow{\rho_{s-1,k_{s-1}}}P_{s,1}\  \  (*)_s$ and \\
$\ \ P_{s+1,1}\xrightarrow{\rho'_{s,k_s}}P_{s,k_s}\xrightarrow{\rho'_{s,k_s-1}}\dots 
\xrightarrow{}P_{s,3}\xrightarrow{\rho'_{s,2}}P_{s,2}\xrightarrow{\rho'_{s,1}}P_{s,1}\  \  (**)_s$.
\item It is possible for $k_1=0$ and/or $k_r=0$.
\end{enumerate}
\end{Construction}

\begin{theorem} Let $\mathcal C$ be generalized cluster category of type $A_n$ and let $B$ be the corresponding cluster tilted algebra. Let $M$ be an indecomposable $B$-module. Then:
\begin{enumerate}
\item  $M$ is a string module with the composition factors appearing exactly as the projective modules appear in the projectively irreducible sequences   in the Construction \ref{procedure}.(\ref{sequences}).
\item $M/\rad M\cong S_{2,1}\oplus S_{4,1}\oplus\dots\oplus S_{r,1}$.
\item $\soc M\cong S_{1,1}\oplus S_{3,1}\oplus\dots \oplus S_{r+1,1}$, where $S_{1,1}$ and $S_{r+1,1}$ may or may not be there.
\end{enumerate}
\end{theorem}
\begin{proof} Follows from the construction which defines projective presentation of $M$ as:
$$P_{0,1}\oplus P_{3,1}\oplus\dots\oplus P_{r+1,1}\xrightarrow{g} P_{2,1}\oplus P_{4,1}\oplus\dots\oplus P_{r,1} \ \ \ \text{where:}$$
$P_{0,1}\xrightarrow{g_1}P_{2,1}$ with $g_1= (\rho_{1,k_1}\dots\rho_{1,1})\rho_{1,0}$ is the composition of projectively irreducible maps as in $(*)_2$ with additional map $\rho_{1,0}$ on the same sectional path,\\
$P_{r+1,1}\xrightarrow{g'_r}P_{r,1}$ with $g'_r= (\rho'_{r,1}\dots\rho'_{r,k_r}) \rho'_{r,k_r+1}$ is the composition  of projectively irreducible maps as in $(**)_r$ with additional map $\rho'_{r,k_r+1}$ on the same sectional path,\\
$P_{t,1}\xrightarrow{(g'_{t-1},g_t)}P_{t-1,1}\oplus P_{t+1,1}$ with 
$g'_{t-1}= (\rho'_{t-1,1}\dots\rho'_{t-1,k_{t-1}})$ 
is the composition of projectively irreducible maps as in $(**)_{t-1}$, and 
$g_t= (\rho_{t,k_t}\dots\rho_{t,1})$ is the composition of projectively irreducible maps as in $(*)_{t+1}$, for  $t\in\{3,5,\dots, r-1\}$.\\
All other maps are zero.

\end{proof}

\begin{ex}
Let $B=\End_{\mc{C}_{A_{10}}}(T)$ be the cluster-tilted algebra from Example \ref{E:category_frieze}. 
Let us illustrate the method of Construction \ref{Construction:mods} for describing the decomposition factors of some modules $M$ over this algebra: take for example $M=\begin{smallmatrix}\;3\;\\8\;1\end{smallmatrix}$. 
(See Figure~\ref{fig:ARquiver} for the position of the modules in the AR quiver.)
From Cor.~\ref{supports} it follows that the projectives $P_1, P_3$ and $P_8$ 
are in the support of $\Hom_{\mc{C}}(-,M)$, which is given as the maximal rectangle 
ending in $M$. If we choose the sectional path ending at $M$ and coming south-east from 
$P_8$, we find that $P_8 \rightarrow P_3$ is a maximal sequence of projectively irreducible 
maps on this path. In step (6) of the above construction we get $P_1 \rightarrow P_3$. 
Since there are no more indecomposable projectives in $\supp(\Hom_{\mc{C}}(-,M))$, one sees 
that $M$ is of the form $\begin{smallmatrix}\;3\;\\8\;1\end{smallmatrix}$. If we had chosen 
the other sectional path from $P_4[1]$ to $M$, then there is no projective on this path. 
In step (5) we consider the parallel maximal sectional paths through the $\tau^t M$. Then 
the first projective is $P_1$ on the north-east path through 
$\tau^3 M=\begin{smallmatrix}\;1\;\\5\;2\;\\10\;\phantom{00}\end{smallmatrix}$. 
In step (6) nothing is added but in (7) we see that $P_3$ is on the north-east path out of $P_1$. 
Finally, $P_8$ lies on the other sectional path to $P_3$ and we have recovered the same $M$. \\
For another example, consider $M=P_7=\begin{smallmatrix}7\\2\\3\\8\end{smallmatrix}$ $\,$.
The support of $\Hom{\mc{C}}(-,M)$ consists of the modules on the line from $P_7$ to $P_8=8$. 
Since the projectives on this path are $P_8 \rightarrow P_3 \rightarrow P_2 \rightarrow P_7$, 
one immediately sees that $P_7$ is uniserial.
\end{ex}

\section{From cluster categories to friezes}\label{S:From cluster categories to friezes}

\subsection{The specialized Caldero Chapoton map}\label{ssec:spec-CC}
$\ $

Let $\CC = \CC_{(Q,W)}$ be a generalized cluster category and let $T = \bigoplus_{i=1}^n T_i$ be a cluster-tilting object of $\CC$ with pairwise non-isomorphic indecomposable summands $T_i$. As before, let $B_T = \End_\CC (T)$  denote the cluster-tilted algebra associated to $T$. As we have seen in 
Section~\ref{S:Description of the modules} every indecomposable object in $\CC$ is either $T_i[1]$ for some $1 \leq i \leq n$ or it can be viewed as 
an indecomposable $B_T$-module. 

The {\em specialized Caldero Chapoton map} is the map we get from postcomposing the 
Caldero Chapoton map associated to $T$ with the specialization of the initial cluster variables to one. 
The Caldero Chapoton map was introduced by Caldero and Chapoton \cite{CalderoChapoton} for 
the acyclic case 
and in a much more general setting for 2-Calabi-Yau triangulated categories with
cluster-tilting objects by Palu \cite{Palu}. 
It was extended to Frobenius categories 
by Fu and Keller in \cite{FuKeller}. See also Section~\ref{ssec:frobenius} below. 
The specialized Caldero Chapoton map is defined on indecomposable objects of $\CC$ by
\[
 \rho_T(M) = \begin{cases}
              1 & \text{ if } M = T_i[1] \\
              \sum_{\underline{e}}\chi(Gr_{\underline{e}}(M)) & \text{ if $M$ is a $B_T$-module}.
             \end{cases}
\]
Here, $Gr_{\underline{e}}(M)$ is the Grassmannian of submodules of the $B_T$-module $M$ with dimension vector $\underline{e}$ and $\chi$ is the Euler-Poincar\'e characteristic.
If $\CC$ is of type $A_n$ and $T$ is a cluster-tilting object of $\CC$, then for every indecomposable $B_T$-module $M$ the Grassmannian $Gr_{\underline{e}}(M)$ is either empty or a point, 
therefore the above formula simplifies to
\[
 \rho_T(M) = \sum_{N \subseteq M} 1 = s(M),
\]
where the sum goes over submodules of $M$ and we denote by $s(M)$ the number of submodules of $M$ up to isomorphism 
(cf. \cite[Example~3.2]{CalderoChapoton}).  

\subsection{Friezes via the specialized Caldero Chapoton map}\label{S:Friezes-via-CC}
$\ $

Let $T = \bigoplus_{i=1}^n T_i$ be a cluster-tilting object in the cluster category $\CC_{A_n}$ of type $A_n$ with pairwise non-isomorphic indecomposable objects 
$T_i$. The {\em frieze associated to $T$} is the frieze we obtain in the following way: We take the Auslander-Reiten quiver of $\CC_{A_n}$ and put in the 
position of the indecomposable object $M$ the positive integer $\rho_T(M)$. Then we add rows of $0$s and $1$s at the top and bottom such that the 
first and last row are rows of $0$s and the second and second-to-last row are rows of $1$s. This is indeed a frieze by \cite[Proposition~5.2]{CalderoChapoton}. 
Note that while \cite[Proposition~5.2]{CalderoChapoton} only shows the statement for cluster-tilting subcategories whose quivers are an orientation of $A_n$, 
the proof can be adapted to include the other quivers in the mutation class of an orientation of $A_n$. Alternatively, the statement follows from the much more 
general result \cite[Theorem~5.4]{HJ}. 
We will see in the next section how to extend the frieze patterns to include the rows of $1$s at 
top and bottom. We do so by using an exact category and the variant of Palu's cluster character 
defined in \cite{FuKeller}, this allows us  to extend $\rho_T$ to the Frobenius 
category. 

\begin{ex}\label{E:category_frieze}

Consider the cluster category $\CC_{A_{10}}$. Its Auslander-Reiten quiver is the quotient of the Auslander-Reiten quiver of $D^b(kA_{10})$ by the action of $\tau^{-1}[1]$, a fundamental domain for which is depicted in black below. We pick the cluster tilting object $T = \bigoplus_{i=1}^{10} T_i$ whose indecomposable summands are marked with circles:

\vskip-2ex
\begin{center}
\begin{tikzpicture}[scale=0.6, transform shape]

	\node[lightgray] (Z1) at (0,10) {$\bullet$};
	\node[lightgray, circle, draw] (Z2) at (2,10) {$\bullet$};
	\node[lightgray] (Z3) at (4,10) {$\bullet$};
	\node[lightgray, circle, draw] (Z4) at (6,10) {$\bullet$};
	\node[lightgray] (Z5) at (8,10) {$\bullet$};
	\node[circle, draw] (Z6) at (10,10) {$\bullet$};
	\node (Z7) at (12,10) {$\bullet$};
	\node[lightgray] (Z8) at (14,10) {$\bullet$};
	\node[lightgray, circle, draw] (Z9) at (16,10) {$\bullet$};
	\node[lightgray] (Z10) at (18,10) {$\bullet$};
	\node[lightgray, circle, draw] (Z11) at (20,10) {$\bullet$};
	\node[lightgray] (Z12) at (22,10) {$\bullet$};
	
	\node[lightgray] (Y1) at (1,9) {$\bullet$};
	\node[lightgray] (Y2) at (3,9) {$\bullet$};
	\node[lightgray] (Y3) at (5,9) {$\bullet$};
	\node[lightgray] (Y4) at (7,9) {$\bullet$};
	\node (Y5) at (9,9) {$\bullet$};
	\node[circle, draw] (Y6) at (11,9) {$\bullet$};
	\node (Y7) at (13,9) {$\bullet$};
	\node[lightgray] (Y8) at (15,9) {$\bullet$};
	\node[lightgray] (Y9) at (17,9) {$\bullet$}; 
	\node[lightgray] (Y10) at (19,9) {$\bullet$};
	\node[lightgray] (Y11) at (21,9) {$\bullet$};
	\node[lightgray] (Y12) at (23,9) {$\bullet$};
	
	\node[lightgray] (S1) at (0,8) {$\bullet$};
	\node[lightgray] (S2) at (2,8) {$\bullet$};
	\node[lightgray, circle, draw] (S3) at (4,8) {$\bullet$};
	\node[lightgray] (S4) at (6,8) {$\bullet$};
	\node (S5) at (8,8) {$\bullet$};
	\node (S6) at (10,8) {$\bullet$};
	\node (S7) at (12,8) {$\bullet$};
	\node[] (S8) at (14,8) {$\bullet$};
	\node[lightgray] (S9) at (16,8) {$\bullet$};
	\node[lightgray, circle, draw] (S10) at (18,8) {$\bullet$};
	\node[lightgray] (S11) at (20,8) {$\bullet$};
	\node[lightgray] (S12) at (22,8) {$\bullet$};
	
	\node[lightgray] (B1) at (1,7) {$\bullet$};
	\node[lightgray] (B2) at (3,7) {$\bullet$};
	\node[lightgray] (B3) at (5,7) {$\bullet$};
	\node (B4) at (7,7) {$\bullet$};
	\node (B5) at (9,7) {$\bullet$};
	\node (B6) at (11,7) {$\bullet$};
	\node (B7) at (13,7) {$\bullet$};
	\node (B8) at (15,7) {$\bullet$};
	\node[lightgray] (B9) at (17,7) {$\bullet$}; 
	\node[lightgray] (B10) at (19,7) {$\bullet$};
	\node[lightgray] (B11) at (21,7) {$\bullet$};
	\node[lightgray] (B12) at (23,7) {$\bullet$};
	\node[lightgray] (C1) at (0,6) {$\bullet$};
	\node[lightgray] (C2) at (2,6) {$\bullet$};
	\node[lightgray] (C3) at (4,6) {$\bullet$};
	\node (C4) at (6,6) {$\bullet$};
	\node (C5) at (8,6) {$\bullet$};
	\node (C6) at (10,6) {$\bullet$};
	\node (C7) at (12,6) {$\bullet$};
	\node (C8) at (14,6) {$\bullet$};
	\node[] (C9) at (16,6) {$\bullet$};
	\node[lightgray] (C10) at (18,6) {$\bullet$};
	\node[lightgray, circle, draw] (C11) at (20,6) {$\bullet$};
	\node[lightgray] (C12) at (22,6) {$\bullet$};
	\node[lightgray] (D1) at (1,5) {$\bullet$};
	\node[lightgray] (D2) at (3,5) {$\bullet$};
	\node (D3) at (5,5) {$\bullet$};
	\node[circle, draw] (D4) at (7,5) {$\bullet$};
	\node (D5) at (9,5) {$\bullet$};
	\node (D6) at (11,5) {$\bullet$};
	\node (D7) at (13,5) {$\bullet$};
	\node (D8) at (15,5) {$\bullet$};
	\node (D9) at (17,5) {$\bullet$}; 
	\node[lightgray] (D10) at (19,5) {$\bullet$};
	\node[lightgray] (D11) at (21,5) {$\bullet$};
	\node[lightgray] (D12) at (23,5) {$\bullet$};
	\node[lightgray] (E1) at (0,4) {$\bullet$};
	\node[lightgray] (E2) at (2,4) {$\bullet$};
	\node (E3) at (4,4) {$\bullet$};
	\node (E4) at (6,4) {$\bullet$};
	\node (E5) at (8,4) {$\bullet$};
	\node (E6) at (10,4) {$\bullet$};
	\node (E7) at (12,4) {$\bullet$};
	\node (E8) at (14,4) {$\bullet$};
	\node (E9) at (16,4) {$\bullet$};
	\node (E10) at (18,4) {$\bullet$};
	\node[lightgray] (E11) at (20,4) {$\bullet$};
	\node[lightgray] (E12) at (22,4) {$\bullet$};
	\node[lightgray] (F1) at (1,3) {$\bullet$};
	\node (F2) at (3,3) {$\bullet$};
	\node[circle, draw] (F3) at (5,3) {$\bullet$};
	\node (F4) at (7,3) {$\bullet$};
	\node (F5) at (9,3) {$\bullet$};
	\node (F6) at (11,3) {$\bullet$};
	\node (F7) at (13,3) {$\bullet$};
	\node (F8) at (15,3) {$\bullet$};
	\node[circle, draw] (F9) at (17,3) {$\bullet$};
	\node (F10) at (19,3) {$\bullet$};
	\node[lightgray] (F11) at (21,3) {$\bullet$};
	\node[lightgray] (F12) at (23,3) {$\bullet$}; 
	\node[lightgray] (G1) at (0,2) {$\bullet$};
	\node (G2) at (2,2) {$\bullet$};
	\node (G3) at (4,2) {$\bullet$};
	\node (G4) at (6,2) {$\bullet$};
	\node (G5) at (8,2) {$\bullet$};
	\node (G6) at (10,2) {$\bullet$};
	\node (G7) at (12,2) {$\bullet$};
	\node (G8) at (14,2) {$\bullet$};
	\node (G9) at (16,2) {$\bullet$};
	\node (G10) at (18,2) {$\bullet$};
	\node (G11) at (20,2) {$\bullet$};
	\node[lightgray] (G12) at (22,2) {$\bullet$};
	\node (H1) at (1,1) {$\bullet$};
	\node[circle, draw] (H2) at (3,1) {$\bullet$};
	\node (H3) at (5,1) {$\bullet$};
	\node[circle, draw] (H4) at (7,1) {$\bullet$};
	\node (H5) at (9,1) {$\bullet$};
	\node[circle, draw] (H6) at (11,1) {$\bullet$};
	\node (H7) at (13,1) {$\bullet$};
	\node[circle, draw] (H8) at (15,1) {$\bullet$};
	\node (H9) at (17,1) {$\bullet$}; 
	\node[circle, draw] (H10) at (19,1) {$\bullet$};
	\node (H11) at (21,1) {$\bullet$};
	\node[lightgray, circle, draw] (H12) at (23,1) {$\bullet$};
	
	\draw[lightgray, ->] (Z1) -- (Y1);
	\draw[lightgray,->] (Z2) -- (Y2);
	\draw[lightgray,->] (Z3) -- (Y3);
	\draw[lightgray,->] (Z4) -- (Y4);
	\draw[lightgray,->] (Z5) -- (Y5);
	\draw[->] (Z6) -- (Y6);
	\draw[->] (Z7) -- (Y7);
	\draw[lightgray,->] (Z8) -- (Y8);
	\draw[lightgray,->] (Z9) -- (Y9);
	\draw[lightgray,->] (Z10) -- (Y10);
	\draw[lightgray,->] (Z11) -- (Y11);
	\draw[lightgray,->] (Z12) -- (Y12);

	\draw[lightgray,->] (Y1) -- (Z2);
	\draw[lightgray,->] (Y2) -- (Z3);
	\draw[lightgray,->] (Y3) -- (Z4);
	\draw[lightgray,->] (Y4) -- (Z5);
	\draw[->] (Y5) -- (Z6);
	\draw[->] (Y6) -- (Z7);
	\draw[lightgray,->] (Y7) -- (Z8);
	\draw[lightgray,->] (Y8) -- (Z9);
	\draw[lightgray,->] (Y9) -- (Z10);
	\draw[lightgray,->] (Y10) -- (Z11);
	\draw[lightgray,->] (Y11) -- (Z12);
	
	\draw[lightgray,<-] (Y1) -- (S1);
	\draw[lightgray,<-] (Y2) -- (S2);
	\draw[lightgray,<-] (Y3) -- (S3);
	\draw[lightgray,<-] (Y4) -- (S4);
	\draw[<-] (Y5) -- (S5);
	\draw[<-] (Y6) -- (S6);
	\draw[<-] (Y7) -- (S7);
	\draw[lightgray,<-] (Y8) -- (S8);
	\draw[lightgray,<-] (Y9) -- (S9);
	\draw[lightgray,<-] (Y10) -- (S10);
	\draw[lightgray,<-] (Y11) -- (S11);
	\draw[lightgray,<-] (Y12) -- (S12);
	
	\draw[lightgray,->] (Y1) -- (S2);
	\draw[lightgray,->] (Y2) -- (S3);
	\draw[lightgray,->] (Y3) -- (S4);
	\draw[lightgray,->] (Y4) -- (S5);
	\draw[->] (Y5) -- (S6);
	\draw[->] (Y6) -- (S7);
	\draw[->] (Y7) -- (S8);
	\draw[lightgray,->] (Y8) -- (S9);
	\draw[lightgray,->] (Y9) -- (S10);
	\draw[lightgray,->] (Y10) -- (S11);
	\draw[lightgray,->] (Y11) -- (S12);	
	
	\draw[lightgray,->] (S1) -- (B1);
	\draw[lightgray,->] (S2) -- (B2);
	\draw[lightgray,->] (S3) -- (B3);
	\draw[lightgray,->] (S4) -- (B4);
	\draw[->] (S5) -- (B5);
	\draw[->] (S6) -- (B6);
	\draw[->] (S7) -- (B7);
	\draw[->] (S8) -- (B8);
	\draw[lightgray,->] (S9) -- (B9);
	\draw[lightgray,->] (S10) -- (B10);
	\draw[lightgray,->] (S11) -- (B11);
	\draw[lightgray,->] (S12) -- (B12);

	\draw[lightgray,->] (B1) -- (S2);
	\draw[lightgray,->] (B2) -- (S3);
	\draw[lightgray,->] (B3) -- (S4);
	\draw[->] (B4) -- (S5);
	\draw[->] (B5) -- (S6);
	\draw[->] (B6) -- (S7);
	\draw[->] (B7) -- (S8);
	\draw[lightgray,->] (B8) -- (S9);
	\draw[lightgray,->] (B9) -- (S10);
	\draw[lightgray,->] (B10) -- (S11);
	\draw[lightgray,->] (B11) -- (S12);
	
	\draw[lightgray,<-] (B1) -- (C1);
	\draw[lightgray,<-] (B2) -- (C2);
	\draw[lightgray,<-] (B3) -- (C3);
	\draw[<-] (B4) -- (C4);
	\draw[<-] (B5) -- (C5);
	\draw[<-] (B6) -- (C6);
	\draw[<-] (B7) -- (C7);
	\draw[<-] (B8) -- (C8);
	\draw[lightgray,<-] (B9) -- (C9);
	\draw[lightgray,<-] (B10) -- (C10);
	\draw[lightgray,<-] (B11) -- (C11);
	\draw[lightgray,<-] (B12) -- (C12);
	
	\draw[lightgray,->] (B1) -- (C2);
	\draw[lightgray,->] (B2) -- (C3);
	\draw[lightgray,->] (B3) -- (C4);
	\draw[->] (B4) -- (C5);
	\draw[->] (B5) -- (C6);
	\draw[->] (B6) -- (C7);
	\draw[->] (B7) -- (C8);
	\draw[->] (B8) -- (C9);
	\draw[lightgray,->] (B9) -- (C10);
	\draw[lightgray,->] (B10) -- (C11);
	\draw[lightgray,->] (B11) -- (C12);	
	
	\draw[lightgray,->] (C1) -- (D1);
	\draw[lightgray,->] (C2) -- (D2);
	\draw[lightgray,->] (C3) -- (D3);
	\draw[->] (C4) -- (D4);
	\draw[->] (C5) -- (D5);
	\draw[->] (C6) -- (D6);
	\draw[->] (C7) -- (D7);
	\draw[->] (C8) -- (D8);
	\draw[->] (C9) -- (D9);
	\draw[lightgray,->] (C10) -- (D10);
	\draw[lightgray,->] (C11) -- (D11);
	\draw[lightgray,->] (C12) -- (D12);
	
	\draw[lightgray,->] (D1) -- (C2);
	\draw[lightgray,->] (D2) -- (C3);
	\draw[->] (D3) -- (C4);
	\draw[->] (D4) -- (C5);
	\draw[->] (D5) -- (C6);
	\draw[->] (D6) -- (C7);
	\draw[->] (D7) -- (C8);
	\draw[->] (D8) -- (C9);
	\draw[lightgray,->] (D9) -- (C10);
	\draw[lightgray,->] (D10) -- (C11);
	\draw[lightgray,->] (D11) -- (C12);
	
	\draw[lightgray,<-] (D1) -- (E1);
	\draw[lightgray,<-] (D2) -- (E2);
	\draw[<-] (D3) -- (E3);
	\draw[<-] (D4) -- (E4);
	\draw[<-] (D5) -- (E5);
	\draw[<-] (D6) -- (E6);
	\draw[<-] (D7) -- (E7);
	\draw[<-] (D8) -- (E8);
	\draw[<-] (D9) -- (E9);
	\draw[lightgray,<-] (D10) -- (E10);
	\draw[lightgray,<-] (D11) -- (E11);
	\draw[lightgray,<-] (D12) -- (E12);
	
	\draw[lightgray,->] (D1) -- (E2);
	\draw[lightgray,->] (D2) -- (E3);
	\draw[->] (D3) -- (E4);
	\draw[->] (D4) -- (E5);
	\draw[->] (D5) -- (E6);
	\draw[->] (D6) -- (E7);
	\draw[->] (D7) -- (E8);
	\draw[->] (D8) -- (E9);
	\draw[->] (D9) -- (E10);
	\draw[lightgray,->] (D10) -- (E11);
	\draw[lightgray,->] (D11) -- (E12);
	
	\draw[lightgray,->] (E1) -- (F1);
	\draw[lightgray,->] (E2) -- (F2);
	\draw[->] (E3) -- (F3);
	\draw[->] (E4) -- (F4);
	\draw[->] (E5) -- (F5);
	\draw[->] (E6) -- (F6);
	\draw[->] (E7) -- (F7);
	\draw[->] (E8) -- (F8);
	\draw[->] (E9) -- (F9);
	\draw[->] (E10) -- (F10);
	\draw[lightgray,->] (E11) -- (F11);
	\draw[lightgray,->] (E12) -- (F12);
	
	\draw[lightgray,->] (F1) -- (E2);
	\draw[->] (F2) -- (E3);
	\draw[->] (F3) -- (E4);
	\draw[->] (F4) -- (E5);
	\draw[->] (F5) -- (E6);
	\draw[->] (F6) -- (E7);
	\draw[->] (F7) -- (E8);
	\draw[->] (F8) -- (E9);
	\draw[->] (F9) -- (E10);
	\draw[lightgray,->] (F10) -- (E11);
	\draw[lightgray,->] (F11) -- (E12);

	\draw[lightgray,<-] (F1) -- (G1);
	\draw[<-] (F2) -- (G2);
	\draw[<-] (F3) -- (G3);
	\draw[<-] (F4) -- (G4);
	\draw[<-] (F5) -- (G5);
	\draw[<-] (F6) -- (G6);
	\draw[<-] (F7) -- (G7);
	\draw[<-] (F8) -- (G8);
	\draw[<-] (F9) -- (G9);
	\draw[<-] (F10) -- (G10);
	\draw[lightgray,<-] (F11) -- (G11);
	\draw[lightgray,<-] (F12) -- (G12);
	
	\draw[lightgray,->] (F1) -- (G2);
	\draw[->] (F2) -- (G3);
	\draw[->] (F3) -- (G4);
	\draw[->] (F4) -- (G5);
	\draw[->] (F5) -- (G6);
	\draw[->] (F6) -- (G7);
	\draw[->] (F7) -- (G8);
	\draw[->] (F8) -- (G9);
	\draw[->] (F9) -- (G10);
	\draw[->] (F10) -- (G11);
	\draw[lightgray,->] (F11) -- (G12);
	
	\draw[lightgray,->] (G1) -- (H1);
	\draw[->] (G2) -- (H2);
	\draw[->] (G3) -- (H3);
	\draw[->] (G4) -- (H4);
	\draw[->] (G5) -- (H5);
	\draw[->] (G6) -- (H6);
	\draw[->] (G7) -- (H7);
	\draw[->] (G8) -- (H8);
	\draw[->] (G9) -- (H9);
	\draw[->] (G10) -- (H10);
	\draw[->] (G11) -- (H11);
	\draw[lightgray,->] (G12) -- (H12);
	
	\draw[->] (H1) -- (G2);
	\draw[->] (H2) -- (G3);
	\draw[->] (H3) -- (G4);
	\draw[->] (H4) -- (G5);
	\draw[->] (H5) -- (G6);
	\draw[->] (H6) -- (G7);
	\draw[->] (H7) -- (G8);
	\draw[->] (H8) -- (G9);
	\draw[->] (H9) -- (G10);
	\draw[->] (H10) -- (G11);
	\draw[lightgray,->] (H11) -- (G12);

\end{tikzpicture}
\end{center}

Consider the cluster tilted algebra $B_T = \End_{\CC_{A_{10}}}(T)$. 
We have $B_T = kQ \big/I$, where $Q$ is the quiver 
\vskip-2ex
\[\xymatrix@!C=0pt@!R=0pt{
&&&&&&7\ar[dl] &\\
Q: &8&&3\ar[ll]\ar[dr]&&2\ar[ll]\ar[rr] &&6\ar[lu]   \\
&&10\ar[ld]&&1\ar[ur]\ar[dl] \\
&9\ar[rr]&&5\ar[rr]\ar[lu]&&4\ar[ul]}\]
and $I$ is the ideal generated by the directed paths of length 2 which are part of the same 3-cycle. 
We refer the reader to \cite{BuanMarshReiten} for a detailed description of cluster-tilted algebras of Dynkin type $A$.

In light of Theorem \ref{T:KR} we can view $\modcat(B_T)$ as a subcategory of $\CC_{A_{10}}$ and label the indecomposable objects in 
$\CC_{A_{10}}$ by modules and shifts of projective modules respectively:

\begin{center}
\begin{tikzpicture}[scale=0.6, transform shape]

	\node[lightgray] (Z1) at (0,10) {$P_6[1]$};
	\node[lightgray] (Z2) at (2,10) {$\begin{smallmatrix}6\\7\end{smallmatrix}$};
	\node[lightgray] (Z3) at (4,10) {$P_7[1]$};
	\node[lightgray] (Z4) at (6,10) {$\begin{smallmatrix}7\\2\\3\\8\end{smallmatrix}$};
	\node[lightgray] (Z5) at (8,10) {$P_8[1]$};
	\node (Z6) at (10,10) {$8$};
	\node (Z7) at (12,10) {$\begin{smallmatrix}3\\1\\5\\10\end{smallmatrix}$};
	\node[lightgray] (Z8) at (14,10) {$P_{10}[1]$};
	\node[lightgray] (Z9) at (16,10) {$\begin{smallmatrix}10\\9\end{smallmatrix}$};
	\node[lightgray] (Z10) at (18,10) {$P_9[1]$};
	\node[lightgray] (Z11) at (20,10) {$\begin{smallmatrix}9\\5\\4\end{smallmatrix}$};
	\node[lightgray] (Z12) at (22,10) {$P_4[1]$};
	
	\node[lightgray] (Y1) at (1,9) {$7$};
	\node[lightgray] (Y2) at (3,9) {$6$};
	\node[lightgray] (Y3) at (5,9) {$\begin{smallmatrix}2\\3\\8\end{smallmatrix}$};
	\node[lightgray] (Y4) at (7,9) {$\begin{smallmatrix}7\\2\\3\end{smallmatrix}$};
	\node (Y5) at (9,9) {$P_3[1]$};
	\node (Y6) at (11,9) {$\begin{smallmatrix}\;3\;\\8\;1\\\;\;5\\\;\;10\end{smallmatrix}$};
	\node (Y7) at (13,9) {$\begin{smallmatrix}3\\1\\5\end{smallmatrix}$};
	\node[lightgray] (Y8) at (15,9) {$9$};
	\node[lightgray] (Y9) at (17,9) {$10$}; 
	\node[lightgray] (Y10) at (19,9) {$\begin{smallmatrix}5\\4\end{smallmatrix}$};
	\node[lightgray] (Y11) at (21,9) {$\begin{smallmatrix}9\\5\end{smallmatrix}$};
	\node[lightgray] (Y12) at (23,9) {$\begin{smallmatrix}1\\2\\6\end{smallmatrix}$};
	
	\node[lightgray] (S1) at (0,8) {$\begin{smallmatrix}\;\;4\\7\;1\\\;2\end{smallmatrix}$};
	\node[lightgray] (S2) at (2,8) {$P_2[1]$};
	\node[lightgray] (S3) at (4,8) {$\begin{smallmatrix}\;2\;\\3\;6\\8\;\;\end{smallmatrix}$};
	\node[lightgray] (S4) at (6,8) {$\begin{smallmatrix}2\\3\end{smallmatrix}$};
	\node (S5) at (8,8) {$\begin{smallmatrix}7\\2\end{smallmatrix}$};
	\node (S6) at (10,8) {$\begin{smallmatrix}1\\5\\10\end{smallmatrix}$};
	\node (S7) at (12,8) {$\begin{smallmatrix}\;3\;\\8\;1\\\;\;5\end{smallmatrix}$};
	\node (S8) at (14,8) {$\begin{smallmatrix}3\;\;\;\\1\;9\\5\end{smallmatrix}$};
	\node[lightgray] (S9) at (16,8) {$P_5[1]$};
	\node[lightgray] (S10) at (18,8) {$\begin{smallmatrix}\;5\;\\4\;10\end{smallmatrix}$};
	\node[lightgray] (S11) at (20,8) {$5$};
	\node[lightgray] (S12) at (22,8) {$\begin{smallmatrix}9\;1\;\\\;\;5\;2\\\;\;\;\;6\end{smallmatrix}$};
	
	\node[lightgray] (B1) at (1,7) {$\begin{smallmatrix}4\\1\end{smallmatrix}$};
	\node[lightgray] (B2) at (3,7) {$\begin{smallmatrix}3\\8\end{smallmatrix}$};
	\node[lightgray] (B3) at (5,7) {$\begin{smallmatrix}\;2\;\\3\;6\end{smallmatrix}$};
	\node (B4) at (7,7) {$2$};
	\node (B5) at (9,7) {$\begin{smallmatrix}\;\;1\;7\\\;5\;2\;\\10\;\;\;\;\end{smallmatrix}$};
	\node (B6) at (11,7) {$\begin{smallmatrix}1\\5\end{smallmatrix}$};
	\node (B7) at (13,7) {$\begin{smallmatrix}\;3\;\;\;\\8\;1\;9\\\;\;\;5\;\end{smallmatrix}$};
	\node (B8) at (15,7) {$\begin{smallmatrix}3\\1\end{smallmatrix}$};
	\node[lightgray] (B9) at (17,7) {$4$}; 
	\node[lightgray] (B10) at (19,7) {$\begin{smallmatrix}5\\10\end{smallmatrix}$};
	\node[lightgray] (B11) at (21,7) {$\begin{smallmatrix}\;1\;\\5\;2\\\;\;6\end{smallmatrix}$};
	\node[lightgray] (B12) at (23,7) {$\begin{smallmatrix}9\;1\\\;5\;2\end{smallmatrix}$};
	\node[lightgray] (C1) at (0,6) {$1$};
	\node[lightgray] (C2) at (2,6) {$\begin{smallmatrix}\;\;\;8\\4\;3\;\\\;1\;\;\end{smallmatrix}$};
	\node[lightgray] (C3) at (4,6) {$3$};
	\node (C4) at (6,6) {$\begin{smallmatrix}2\\6\end{smallmatrix}$};
	\node (C5) at (8,6) {$\begin{smallmatrix}\;\;1\;\\\;5\;2\\10\;\;\;\end{smallmatrix}$};
	\node (C6) at (10,6) {$\begin{smallmatrix}\;1\;7\\5\;2\;\end{smallmatrix}$};
	\node (C7) at (12,6) {$\begin{smallmatrix}9\;1\\\;5\;\end{smallmatrix}$};
	\node (C8) at (14,6) {$\begin{smallmatrix}\;3\;\\8\;1\end{smallmatrix}$};
	\node (C9) at (16,6) {$\begin{smallmatrix}4\;3\\\;1\;\end{smallmatrix}$};
	\node[lightgray] (C10) at (18,6) {$P_1[1]$};
	\node[lightgray] (C11) at (20,6) {$5$};
	\node[lightgray] (C12) at (22,6) {$\begin{smallmatrix}9\;1\;\\\;\;5\;2\\\;\;\;\;6\end{smallmatrix}$};
	\node[lightgray] (D1) at (1,5) {$\begin{smallmatrix}\;3\;\\8\;1\end{smallmatrix}$};
	\node[lightgray] (D2) at (3,5) {$\begin{smallmatrix}4\;3\\\;1\;\end{smallmatrix}$};
	\node (D3) at (5,5) {$P_1[1]$};
	\node (D4) at (7,5) {$\begin{smallmatrix}\;\;1\;\;\\\;5\;2\;\\10\;\;\;6\end{smallmatrix}$};
	\node (D5) at (9,5) {$\begin{smallmatrix}\;1\;\\5\;2\end{smallmatrix}$};
	\node (D6) at (11,5) {$\begin{smallmatrix}9\;1\;7\\\;5\;2\;\end{smallmatrix}$};
	\node (D7) at (13,5) {$1$};
	\node (D8) at (15,5) {$\begin{smallmatrix}\;3\;4\\8\;1\;\end{smallmatrix}$};
	\node (D9) at (17,5) {$3$}; 
	\node[lightgray] (D10) at (19,5) {$\begin{smallmatrix}2\\6\end{smallmatrix}$};
	\node[lightgray] (D11) at (21,5) {$\begin{smallmatrix}\;\;1\;\\\;5\;2\;\\10\;\;\;\;\end{smallmatrix}$};
	\node[lightgray] (D12) at (23,5) {$\begin{smallmatrix}1\\5\end{smallmatrix}$};
	\node[lightgray] (E1) at (0,4) {$\begin{smallmatrix}\;3\;\;\;\\8\;1\;9\\\;\;\;5\end{smallmatrix}$};
	\node[lightgray] (E2) at (2,4) {$\begin{smallmatrix}3\\1\end{smallmatrix}$};
	\node (E3) at (4,4) {$4$};
	\node (E4) at (6,4) {$\begin{smallmatrix}5\\10\end{smallmatrix}$};
	\node (E5) at (8,4) {$\begin{smallmatrix}\;1\;\\5\;2\\\;\;6\end{smallmatrix}$};
	\node (E6) at (10,4) {$\begin{smallmatrix}9\;1\\\;5\;2\end{smallmatrix}$};
	\node (E7) at (12,4) {$\begin{smallmatrix}7\;1\\\;2\end{smallmatrix}$};
	\node (E8) at (14,4) {$\begin{smallmatrix}4\\1\end{smallmatrix}$};
	\node (E9) at (16,4) {$\begin{smallmatrix}3\\8\end{smallmatrix}$};
	\node (E10) at (18,4) {$\begin{smallmatrix}\;2\;\\3\;6\end{smallmatrix}$};
	\node[lightgray] (E11) at (20,4) {$2$};
	\node[lightgray] (E12) at (22,4) {$\begin{smallmatrix}\;\;1\;7\\\;5\;2\;\\10\;\;\;\end{smallmatrix}$};
	\node[lightgray] (F1) at (1,3) {$\begin{smallmatrix}3\;\;\;\\1\;9\\5\end{smallmatrix}$};
	\node (F2) at (3,3) {$P_5[1]$};
	\node (F3) at (5,3) {$\begin{smallmatrix}\;5\;\\4\;10\end{smallmatrix}$};
	\node (F4) at (7,3) {$5$};
	\node (F5) at (9,3) {$\begin{smallmatrix}9\;1\;\\\;\;5\;2\\\;\;\;\;6\end{smallmatrix}$};
	\node (F6) at (11,3) {$\begin{smallmatrix}1\\2\end{smallmatrix}$};
	\node (F7) at (13,3) {$\begin{smallmatrix}\;\;4\\7\;1\\\;2\end{smallmatrix}$};
	\node (F8) at (15,3) {$P_2[1]$};
	\node (F9) at (17,3) {$\begin{smallmatrix}\;2\;\\3\;6\\8\;\;\end{smallmatrix}$};
	\node (F10) at (19,3) {$\begin{smallmatrix}2\\3\end{smallmatrix}$};
	\node[lightgray] (F11) at (21,3) {$\begin{smallmatrix}7\\2\end{smallmatrix}$};
	\node[lightgray] (F12) at (23,3) {$\begin{smallmatrix}1\\5\\10\end{smallmatrix}$}; 
	\node[lightgray] (G1) at (0,2) {$\begin{smallmatrix}3\\1\\5\end{smallmatrix}$};
	\node (G2) at (2,2) {$9$};
	\node (G3) at (4,2) {$10$};
	\node (G4) at (6,2) {$\begin{smallmatrix}5\\4\end{smallmatrix}$};
	\node (G5) at (8,2) {$\begin{smallmatrix}9\\5\end{smallmatrix}$};
	\node (G6) at (10,2) {$\begin{smallmatrix}1\\2\\6\end{smallmatrix}$};
	\node (G7) at (12,2) {$\begin{smallmatrix}4\\1\\2\end{smallmatrix}$};
	\node (G8) at (14,2) {$7$};
	\node (G9) at (16,2) {$6$};
	\node (G10) at (18,2) {$\begin{smallmatrix}2\\3\\8\end{smallmatrix}$};
	\node (G11) at (20,2) {$\begin{smallmatrix}7\\2\\3\end{smallmatrix}$};
	\node[lightgray] (G12) at (22,2) {$P_3[1]$};
	\node (H1) at (1,1) {$P_{10}[1]$};
	\node (H2) at (3,1) {$\begin{smallmatrix}10\\9\end{smallmatrix}$};
	\node (H3) at (5,1) {$P_9[1]$};
	\node (H4) at (7,1) {$\begin{smallmatrix}9\\5\\4\end{smallmatrix}$};
	\node (H5) at (9,1) {$P_4[1]$};
	\node (H6) at (11,1) {$\begin{smallmatrix}4\\1\\2\\6\end{smallmatrix}$};
	\node (H7) at (13,1) {$P_6[1]$};
	\node(H8) at (15,1) {$\begin{smallmatrix}6\\7\end{smallmatrix}$};
	\node (H9) at (17,1) {$P_7[1]$}; 
	\node(H10) at (19,1) {$\begin{smallmatrix}7\\2\\3\\8\end{smallmatrix}$};
	\node (H11) at (21,1) {$P_8[1]$};
	\node[lightgray] (H12) at (23,1) {$8$};
	
	\draw[lightgray, ->] (Z1) -- (Y1);
	\draw[lightgray,->] (Z2) -- (Y2);
	\draw[lightgray,->] (Z3) -- (Y3);
	\draw[lightgray,->] (Z4) -- (Y4);
	\draw[lightgray,->] (Z5) -- (Y5);
	\draw[->] (Z6) -- (Y6);
	\draw[->] (Z7) -- (Y7);
	\draw[lightgray,->] (Z8) -- (Y8);
	\draw[lightgray,->] (Z9) -- (Y9);
	\draw[lightgray,->] (Z10) -- (Y10);
	\draw[lightgray,->] (Z11) -- (Y11);
	\draw[lightgray,->] (Z12) -- (Y12);

	\draw[lightgray,->] (Y1) -- (Z2);
	\draw[lightgray,->] (Y2) -- (Z3);
	\draw[lightgray,->] (Y3) -- (Z4);
	\draw[lightgray,->] (Y4) -- (Z5);
	\draw[->] (Y5) -- (Z6);
	\draw[->] (Y6) -- (Z7);
	\draw[lightgray,->] (Y7) -- (Z8);
	\draw[lightgray,->] (Y8) -- (Z9);
	\draw[lightgray,->] (Y9) -- (Z10);
	\draw[lightgray,->] (Y10) -- (Z11);
	\draw[lightgray,->] (Y11) -- (Z12);
	
	\draw[lightgray,<-] (Y1) -- (S1);
	\draw[lightgray,<-] (Y2) -- (S2);
	\draw[lightgray,<-] (Y3) -- (S3);
	\draw[lightgray,<-] (Y4) -- (S4);
	\draw[<-] (Y5) -- (S5);
	\draw[<-] (Y6) -- (S6);
	\draw[<-] (Y7) -- (S7);
	\draw[lightgray,<-] (Y8) -- (S8);
	\draw[lightgray,<-] (Y9) -- (S9);
	\draw[lightgray,<-] (Y10) -- (S10);
	\draw[lightgray,<-] (Y11) -- (S11);
	\draw[lightgray,<-] (Y12) -- (S12);
	
	\draw[lightgray,->] (Y1) -- (S2);
	\draw[lightgray,->] (Y2) -- (S3);
	\draw[lightgray,->] (Y3) -- (S4);
	\draw[lightgray,->] (Y4) -- (S5);
	\draw[->] (Y5) -- (S6);
	\draw[->] (Y6) -- (S7);
	\draw[->] (Y7) -- (S8);
	\draw[lightgray,->] (Y8) -- (S9);
	\draw[lightgray,->] (Y9) -- (S10);
	\draw[lightgray,->] (Y10) -- (S11);
	\draw[lightgray,->] (Y11) -- (S12);	
	
	\draw[lightgray,->] (S1) -- (B1);
	\draw[lightgray,->] (S2) -- (B2);
	\draw[lightgray,->] (S3) -- (B3);
	\draw[lightgray,->] (S4) -- (B4);
	\draw[->] (S5) -- (B5);
	\draw[->] (S6) -- (B6);
	\draw[->] (S7) -- (B7);
	\draw[->] (S8) -- (B8);
	\draw[lightgray,->] (S9) -- (B9);
	\draw[lightgray,->] (S10) -- (B10);
	\draw[lightgray,->] (S11) -- (B11);
	\draw[lightgray,->] (S12) -- (B12);

	\draw[lightgray,->] (B1) -- (S2);
	\draw[lightgray,->] (B2) -- (S3);
	\draw[lightgray,->] (B3) -- (S4);
	\draw[->] (B4) -- (S5);
	\draw[->] (B5) -- (S6);
	\draw[->] (B6) -- (S7);
	\draw[->] (B7) -- (S8);
	\draw[lightgray,->] (B8) -- (S9);
	\draw[lightgray,->] (B9) -- (S10);
	\draw[lightgray,->] (B10) -- (S11);
	\draw[lightgray,->] (B11) -- (S12);
	
	\draw[lightgray,<-] (B1) -- (C1);
	\draw[lightgray,<-] (B2) -- (C2);
	\draw[lightgray,<-] (B3) -- (C3);
	\draw[<-] (B4) -- (C4);
	\draw[<-] (B5) -- (C5);
	\draw[<-] (B6) -- (C6);
	\draw[<-] (B7) -- (C7);
	\draw[<-] (B8) -- (C8);
	\draw[lightgray,<-] (B9) -- (C9);
	\draw[lightgray,<-] (B10) -- (C10);
	\draw[lightgray,<-] (B11) -- (C11);
	\draw[lightgray,<-] (B12) -- (C12);
	
	\draw[lightgray,->] (B1) -- (C2);
	\draw[lightgray,->] (B2) -- (C3);
	\draw[lightgray,->] (B3) -- (C4);
	\draw[->] (B4) -- (C5);
	\draw[->] (B5) -- (C6);
	\draw[->] (B6) -- (C7);
	\draw[->] (B7) -- (C8);
	\draw[->] (B8) -- (C9);
	\draw[lightgray,->] (B9) -- (C10);
	\draw[lightgray,->] (B10) -- (C11);
	\draw[lightgray,->] (B11) -- (C12);	
	
	\draw[lightgray,->] (C1) -- (D1);
	\draw[lightgray,->] (C2) -- (D2);
	\draw[lightgray,->] (C3) -- (D3);
	\draw[->] (C4) -- (D4);
	\draw[->] (C5) -- (D5);
	\draw[->] (C6) -- (D6);
	\draw[->] (C7) -- (D7);
	\draw[->] (C8) -- (D8);
	\draw[->] (C9) -- (D9);
	\draw[lightgray,->] (C10) -- (D10);
	\draw[lightgray,->] (C11) -- (D11);
	\draw[lightgray,->] (C12) -- (D12);
	
	\draw[lightgray,->] (D1) -- (C2);
	\draw[lightgray,->] (D2) -- (C3);
	\draw[->] (D3) -- (C4);
	\draw[->] (D4) -- (C5);
	\draw[->] (D5) -- (C6);
	\draw[->] (D6) -- (C7);
	\draw[->] (D7) -- (C8);
	\draw[->] (D8) -- (C9);
	\draw[lightgray,->] (D9) -- (C10);
	\draw[lightgray,->] (D10) -- (C11);
	\draw[lightgray,->] (D11) -- (C12);
	
	\draw[lightgray,<-] (D1) -- (E1);
	\draw[lightgray,<-] (D2) -- (E2);
	\draw[<-] (D3) -- (E3);
	\draw[<-] (D4) -- (E4);
	\draw[<-] (D5) -- (E5);
	\draw[<-] (D6) -- (E6);
	\draw[<-] (D7) -- (E7);
	\draw[<-] (D8) -- (E8);
	\draw[<-] (D9) -- (E9);
	\draw[lightgray,<-] (D10) -- (E10);
	\draw[lightgray,<-] (D11) -- (E11);
	\draw[lightgray,<-] (D12) -- (E12);
	
	\draw[lightgray,->] (D1) -- (E2);
	\draw[lightgray,->] (D2) -- (E3);
	\draw[->] (D3) -- (E4);
	\draw[->] (D4) -- (E5);
	\draw[->] (D5) -- (E6);
	\draw[->] (D6) -- (E7);
	\draw[->] (D7) -- (E8);
	\draw[->] (D8) -- (E9);
	\draw[->] (D9) -- (E10);
	\draw[lightgray,->] (D10) -- (E11);
	\draw[lightgray,->] (D11) -- (E12);
	
	\draw[lightgray,->] (E1) -- (F1);
	\draw[lightgray,->] (E2) -- (F2);
	\draw[->] (E3) -- (F3);
	\draw[->] (E4) -- (F4);
	\draw[->] (E5) -- (F5);
	\draw[->] (E6) -- (F6);
	\draw[->] (E7) -- (F7);
	\draw[->] (E8) -- (F8);
	\draw[->] (E9) -- (F9);
	\draw[->] (E10) -- (F10);
	\draw[lightgray,->] (E11) -- (F11);
	\draw[lightgray,->] (E12) -- (F12);
	
	\draw[lightgray,->] (F1) -- (E2);
	\draw[->] (F2) -- (E3);
	\draw[->] (F3) -- (E4);
	\draw[->] (F4) -- (E5);
	\draw[->] (F5) -- (E6);
	\draw[->] (F6) -- (E7);
	\draw[->] (F7) -- (E8);
	\draw[->] (F8) -- (E9);
	\draw[->] (F9) -- (E10);
	\draw[lightgray,->] (F10) -- (E11);
	\draw[lightgray,->] (F11) -- (E12);

	\draw[lightgray,<-] (F1) -- (G1);
	\draw[<-] (F2) -- (G2);
	\draw[<-] (F3) -- (G3);
	\draw[<-] (F4) -- (G4);
	\draw[<-] (F5) -- (G5);
	\draw[<-] (F6) -- (G6);
	\draw[<-] (F7) -- (G7);
	\draw[<-] (F8) -- (G8);
	\draw[<-] (F9) -- (G9);
	\draw[<-] (F10) -- (G10);
	\draw[lightgray,<-] (F11) -- (G11);
	\draw[lightgray,<-] (F12) -- (G12);
	
	\draw[lightgray,->] (F1) -- (G2);
	\draw[->] (F2) -- (G3);
	\draw[->] (F3) -- (G4);
	\draw[->] (F4) -- (G5);
	\draw[->] (F5) -- (G6);
	\draw[->] (F6) -- (G7);
	\draw[->] (F7) -- (G8);
	\draw[->] (F8) -- (G9);
	\draw[->] (F9) -- (G10);
	\draw[->] (F10) -- (G11);
	\draw[lightgray,->] (F11) -- (G12);
	
	\draw[lightgray,->] (G1) -- (H1);
	\draw[->] (G2) -- (H2);
	\draw[->] (G3) -- (H3);
	\draw[->] (G4) -- (H4);
	\draw[->] (G5) -- (H5);
	\draw[->] (G6) -- (H6);
	\draw[->] (G7) -- (H7);
	\draw[->] (G8) -- (H8);
	\draw[->] (G9) -- (H9);
	\draw[->] (G10) -- (H10);
	\draw[->] (G11) -- (H11);
	\draw[lightgray,->] (G12) -- (H12);
	
	\draw[->] (H1) -- (G2);
	\draw[->] (H2) -- (G3);
	\draw[->] (H3) -- (G4);
	\draw[->] (H4) -- (G5);
	\draw[->] (H5) -- (G6);
	\draw[->] (H6) -- (G7);
	\draw[->] (H7) -- (G8);
	\draw[->] (H8) -- (G9);
	\draw[->] (H9) -- (G10);
	\draw[->] (H10) -- (G11);
	\draw[lightgray,->] (H11) -- (G12);

\end{tikzpicture}
\end{center}

Replacing each vertex labelled by a module by the number of its submodules, the shifts of projectives by 1s and adding in the first two and last two rows of 0s and 1s gives rise to the associated frieze:
\vskip-2ex
\begin{center}
\begin{tikzpicture}[scale=0.6, transform shape]

	\node[lightgray] (1) at (0,12) {$0$};
	\node[lightgray] (2) at (2,12) {$0$};
	\node[lightgray] (3) at (4,12) {$0$};
	\node[lightgray] (4) at (6,12) {$0$};
	\node[lightgray] (5) at (8,12) {$0$};
	\node[lightgray] (6) at (10,12) {$0$};
	\node[lightgray] (7) at (12,12) {$0$};
	\node[lightgray] (8) at (14,12) {$0$};
	\node[lightgray] (9) at (16,12) {$0$};
	\node[lightgray] (10) at (18,12) {$0$};
	\node[lightgray] (11) at (20,12) {$0$};
	\node[lightgray] (12) at (22,12) {$0$};
	
	\node[lightgray] (1) at (1,11) {$1$};
	\node[lightgray] (2) at (3,11) {$1$};
	\node[lightgray] (3) at (5,11) {$1$};
	\node[lightgray] (4) at (7,11) {$1$};
	\node[lightgray] (5) at (9,11) {$1$};
	\node[lightgray] (6) at (11,11) {$1$};
	\node[lightgray] (7) at (13,11) {$1$};
	\node[lightgray] (8) at (15,11) {$1$};
	\node[lightgray] (9) at (17,11) {$1$}; 
	\node[lightgray] (10) at (19,11) {$1$};
	\node[lightgray] (11) at (21,11) {$1$};
	\node[lightgray] (12) at (23,11) {$1$};

	\node[lightgray] (Z1) at (0,10) {$1$};
	\node[lightgray] (Z2) at (2,10) {$3$};
	\node[lightgray] (Z3) at (4,10) {$1$};
	\node[lightgray] (Z4) at (6,10) {$5$};
	\node[lightgray] (Z5) at (8,10) {$1$};
	\node (Z6) at (10,10) {$2$};
	\node (Z7) at (12,10) {$5$};
	\node[lightgray] (Z8) at (14,10) {$1$};
	\node[lightgray] (Z9) at (16,10) {$3$};
	\node[lightgray] (Z10) at (18,10) {$1$};
	\node[lightgray] (Z11) at (20,10) {$4$};
	\node[lightgray] (Z12) at (22,10) {$1$};
	
	\node[lightgray] (Y1) at (1,9) {$2$};
	\node[lightgray] (Y2) at (3,9) {$2$};
	\node[lightgray] (Y3) at (5,9) {$4$};
	\node[lightgray] (Y4) at (7,9) {$4$};
	\node (Y5) at (9,9) {$1$};
	\node (Y6) at (11,9) {$9$};
	\node (Y7) at (13,9) {$4$};
	\node[lightgray] (Y8) at (15,9) {$2$};
	\node[lightgray] (Y9) at (17,9) {$2$}; 
	\node[lightgray] (Y10) at (19,9) {$3$};
	\node[lightgray] (Y11) at (21,9) {$3$};
	\node[lightgray] (Y12) at (23,9) {$4$};
	
	\node[lightgray] (S1) at (0,8) {$7$};
	\node[lightgray] (S2) at (2,8) {$1$};
	\node[lightgray] (S3) at (4,8) {$7$};
	\node[lightgray] (S4) at (6,8) {$3$};
	\node (S5) at (8,8) {$3$};
	\node (S6) at (10,8) {$4$};
	\node (S7) at (12,8) {$7$};
	\node (S8) at (14,8) {$7$};
	\node[lightgray] (S9) at (16,8) {$1$};
	\node[lightgray] (S10) at (18,8) {$5$};
	\node[lightgray] (S11) at (20,8) {$2$};
	\node[lightgray] (S12) at (22,8) {$11$};
	
	\node[lightgray] (B1) at (1,7) {$3$};
	\node[lightgray] (B2) at (3,7) {$3$};
	\node[lightgray] (B3) at (5,7) {$5$};
	\node (B4) at (7,7) {$2$};
	\node (B5) at (9,7) {$11$};
	\node (B6) at (11,7) {$3$};
	\node (B7) at (13,7) {$12$};
	\node (B8) at (15,7) {$3$};
	\node[lightgray] (B9) at (17,7) {$2$}; 
	\node[lightgray] (B10) at (19,7) {$3$};
	\node[lightgray] (B11) at (21,7) {$7$};
	\node[lightgray] (B12) at (23,7) {$8$};
	\node[lightgray] (C1) at (0,6) {$2$};
	\node[lightgray] (C2) at (2,6) {$8$};
	\node[lightgray] (C3) at (4,6) {$2$};
	\node (C4) at (6,6) {$3$};
	\node (C5) at (8,6) {$7$};
	\node (C6) at (10,6) {$8$};
	\node (C7) at (12,6) {$5$};
	\node (C8) at (14,6) {$5$};
	\node (C9) at (16,6) {$5$};
	\node[lightgray] (C10) at (18,6) {$1$};
	\node[lightgray] (C11) at (20,6) {$10$};
	\node[lightgray] (C12) at (22,6) {$5$};
	\node[lightgray] (D1) at (1,5) {$5$};
	\node[lightgray] (D2) at (3,5) {$5$};
	\node (D3) at (5,5) {$1$};
	\node (D4) at (7,5) {$10$};
	\node (D5) at (9,5) {$5$};
	\node (D6) at (11,5) {$13$};
	\node (D7) at (13,5) {$2$};
	\node (D8) at (15,5) {$8$};
	\node (D9) at (17,5) {$2$}; 
	\node[lightgray] (D10) at (19,5) {$3$};
	\node[lightgray] (D11) at (21,5) {$7$};
	\node[lightgray] (D12) at (23,5) {$8$};
	\node[lightgray] (E1) at (0,4) {$12$};
	\node[lightgray] (E2) at (2,4) {$3$};
	\node (E3) at (4,4) {$2$};
	\node (E4) at (6,4) {$3$};
	\node (E5) at (8,4) {$7$};
	\node (E6) at (10,4) {$8$};
	\node (E7) at (12,4) {$5$};
	\node (E8) at (14,4) {$3$};
	\node (E9) at (16,4) {$3$};
	\node (E10) at (18,4) {$5$};
	\node[lightgray] (E11) at (20,4) {$2$};
	\node[lightgray] (E12) at (22,4) {$11$};
	\node[lightgray] (F1) at (1,3) {$7$};
	\node (F2) at (3,3) {$1$};
	\node (F3) at (5,3) {$5$};
	\node (F4) at (7,3) {$2$};
	\node (F5) at (9,3) {$11$};
	\node (F6) at (11,3) {$3$};
	\node (F7) at (13,3) {$7$};
	\node (F8) at (15,3) {$1$};
	\node (F9) at (17,3) {$7$};
	\node (F10) at (19,3) {$3$};
	\node[lightgray] (F11) at (21,3) {$3$};
	\node[lightgray] (F12) at (23,3) {$4$}; 
	\node[lightgray] (G1) at (0,2) {$4$};
	\node (G2) at (2,2) {$2$};
	\node (G3) at (4,2) {$2$};
	\node (G4) at (6,2) {$3$};
	\node (G5) at (8,2) {$3$};
	\node (G6) at (10,2) {$4$};
	\node (G7) at (12,2) {$4$};
	\node (G8) at (14,2) {$2$};
	\node (G9) at (16,2) {$2$};
	\node (G10) at (18,2) {$4$};
	\node (G11) at (20,2) {$4$};
	\node[lightgray] (G12) at (22,2) {$1$};
	\node (H1) at (1,1) {$1$};
	\node (H2) at (3,1) {$3$};
	\node (H3) at (5,1) {$1$};
	\node (H4) at (7,1) {$4$};
	\node (H5) at (9,1) {$1$};
	\node (H6) at (11,1) {$5$};
	\node (H7) at (13,1) {$1$};
	\node (H8) at (15,1) {$3$};
	\node (H9) at (17,1) {$1$}; 
	\node (H10) at (19,1) {$5$};
	\node (H11) at (21,1) {$1$};
	\node[lightgray] (H12) at (23,1) {$2$};
	
	\node[lightgray] (1) at (0,0) {$1$};
	\node[lightgray] (2) at (2,0) {$1$};
	\node[lightgray] (3) at (4,0) {$1$};
	\node[lightgray] (4) at (6,0) {$1$};
	\node[lightgray] (5) at (8,0) {$1$};
	\node[lightgray] (6) at (10,0) {$1$};
	\node[lightgray] (7) at (12,0) {$1$};
	\node[lightgray] (8) at (14,0) {$1$};
	\node[lightgray] (9) at (16,0) {$1$};
	\node[lightgray] (10) at (18,0) {$1$};
	\node[lightgray] (11) at (20,0) {$1$};
	\node[lightgray] (12) at (22,0) {$1$};
	
	\node[lightgray] (1) at (1,-1) {$0$};
	\node[lightgray] (2) at (3,-1) {$0$};
	\node[lightgray] (3) at (5,-1) {$0$};
	\node[lightgray] (4) at (7,-1) {$0$};
	\node[lightgray] (5) at (9,-1) {$0$};
	\node[lightgray] (6) at (11,-1) {$0$};
	\node[lightgray] (7) at (13,-1) {$0$};
	\node[lightgray] (8) at (15,-1) {$0$};
	\node[lightgray] (9) at (17,-1) {$0$}; 
	\node[lightgray] (10) at (19,-1) {$0$};
	\node[lightgray] (11) at (21,-1) {$0$};
	\node[lightgray] (12) at (23,-1) {$0$};
\end{tikzpicture}
\end{center}
\vskip-2ex
 \end{ex}

Knowing the position of the 1s in the frieze, i.e.\ the position of the (shifts of the) indecomposable projectives in the Auslander-Reiten quiver, 
is enough to determine the whole frieze using the frieze rule. However, we can determine each entry independently: It is 
the number of submodules of the indecomposable $B_T$-module sitting in the same position in the Auslander-Reiten quiver. In Section \ref{S:Description of the modules} we have seen how to determine the composition series of each module according to its position in the Auslander-Reiten quiver, and in Section \ref{S:Number of submodules} we will derive the number of isomorphism classes of submodules of any given indecomposable $B_T$-module from its composition series. This will allow us to directly determine the entry of the frieze solely based on its relative position to the $1$s in the frieze.

\subsection{Triangulations and frieze patterns}\label{ssec:tri-frieze}

$\ $

Given a regular $n+3$-gon we label its vertices clockwise $1, \dots, n+3$, and we consider them modulo $n+3$.  
A diagonal with endpoints $i$ and $j$ is denoted by $[ij]$.
By \cite{CoCo1,CoCo2} and \cite{BCI}, matching numbers for 
diagonals in a triangulated polygon together 
with the boundary segments correspond to the non-zero entries in a frieze: 
The boundary segments correspond to the entries in the first and the last row of the frieze 
filled only with 1's. 
An entry $m_{i-1,i+1}$, cf. Section~\ref{sec:intro}, 
in the first non-trivial row corresponds to a diagonal of the 
form $[i-1,i+1]$, it is the number of triangles incident with vertex $i$. 
In the next row are the $m_{i-1,i+2}$, they are the 
matching numbers for the diagonals $[i-1,i+2]$, i.e. the number 
of matchings between triangles of the triangulation and vertices $\{i,i+1\}$. 
In particular, apart from the two rows of $1$s, 
the diagonals $[ij]$ of the triangulation are exactly the entries that equal $1$ in the frieze pattern. 

\subsection{Triangulations and cluster categories}\label{ssec:tri-category}
$\ $

On the other hand, diagonals in an $n+3$-gon are in bijection with indecomposable objects 
in the corresponding cluster category $\mathcal{C}$ (\cite{CalderoChapotonSchiffler}): 
we can define a stable translation quiver 
on the diagonals $[i,j]$ of the $n+3$-gon with arrows $[i,j]\to [i,j-1]$ and $[i,j]\to [i-1,j]$ (provided the 
endpoint is a diagonal) and translation $\tau([i,j])=[i+1,j+1]$. This quiver is isomorphic to the AR quiver of 
$\mathcal C$. 
In terms of the AR quiver in Section~\ref{sec:AR quiver}, the row of objects 
$a_{j,1}$, $1\le j\le n+1$ together with $a_{1,n}$, $a_{2,n}$ corresponds to the diagonals $[i-1,i+1]$ 
($i=1,\dots, n+3$), reducing endpoints modulo $n+3$. 
Take $\mathcal T$ a triangulation of an $n+3$-gon with diagonals $d_i=[i_1,i_2]$ (for $i=1,\dots, n$). 
To $\mathcal T$, we associate the cluster-tilting object $T=\oplus T_i$ where $T_i$ corresponds to the diagonal 
$[i_1-1,i_2-1]$. 
Recall that the specialized Caldero Chapoton map $\rho_T$ from Section~\ref{S:Friezes-via-CC} 
associates an entry $1$ with every indecomposable $T_i[1]$. 
In terms of the diagonals of the triangulation $\mathcal T$, this amounts to an entry $1$ 
for every $[ij]\in\mathcal T$ as in Section~\ref{ssec:tri-frieze}.

\subsection{Triangulations and a Frobenius category}\label{ssec:frobenius}

$\ $

We extend $\ind\mathcal C$ by adding an indecomposable for each boundary segment of 
the polygon and denote the resulting category 
by $\mathcal{C}_f$. Then $\mathcal C_f$ is the Frobenius category of maximal CM-modules 
categorifying 
the cluster algebra structure of the coordinate 
ring of the (affine cone of the) Grassmannian Gr(2,n) as studied in \cite{DeLuo} 
and for general Grassmannians in \cite{JKS,BKM}. 
The stable category of $\mathcal{C}_f$ is equivalent to $\mathcal{C}$. 
We then extend the definition of $\rho_T$ to $\mathcal{C}_f$ by setting 
\[
\rho_T(M)=1 \quad\mbox{if $M$ corresponds to a boundary segment.}
\]
This agrees with the extension of the cluster character to Frobenius category 
given by Fu and Keller, cf. Theorem~\cite[Theorem 3.3]{FuKeller}.

\section{Number of submodules } \label{S:Number of submodules}

As in the previous section, let $\CC = \CC_{(Q,W)}$ be a generalized cluster category and $T = \bigoplus_{i=1}^n T_i$ be a cluster-tilting object of $\CC$ with pairwise non-isomorphic indecomposable summands $T_i$, and $B_T = \End_\CC (T)$ the cluster-tilted algebra associated to $T$. Then $B_T=kQ/I$ and $Q$ is mutation equivalent to type $A_n$ \cite{CalderoChapotonSchiffler,SchifflerBook}.
All indecomposable modules of $kQ/I$ are string modules. 
Let $M$ be an indecomposable module over $B_T$, or equivalently, an irreducible representation of the quiver $Q$.  

Let $s(M)$ denote the  number of submodules of $M$. For representations of quivers, we use notation following \cite{SchifflerBook}. 
Let $S$ be a simple in the support of $M$. Then $S$ is a \emph{valley} if  $S$ belongs to $\mathrm{soc}(M)$ and $S$ is a \emph{peak} if $S$ belongs to $\mathrm{top}(M) = M/\rad(M)$. We number the simples appearing as peaks and valleys as $S(i)$, $i=1, \ldots, m$. Define further modules $N_i$ for $i=1, \ldots, m$ as follows:
\[ N_i= \begin{cases} \!\begin{aligned}
       & \bullet \text{   max. uniserial submodules of $M$ containing $S(i)$ and extending $S(i+1)$ } \\
       & \phantom{\bullet} \text{  if $S(i)$ is a  valley,}
    \end{aligned}            \\
        \!\begin{aligned}
       & \bullet \text{ max. uniserial submodules of $M$ containing $S(i+1)$ and extending $S(i)$} \\
       & \phantom{\bullet} \text{ if $S(i)$ is a  peak.}
    \end{aligned} 
\end{cases}
\]

We call the $N_i$ the \emph{legs} of $M$ (see Fig.~\ref{F:string}). Then we set $k_i:=l(N_i)$ for $i=1, \ldots, m$, where $l(N_i)$ denotes the composition length of the uniserial module $N_i$. Note that $s(M)$ is determined by the sequence $ k_1, \ldots,  k_m$. We say that  $M$ is of \emph{shape} $(k_1, k_2, \ldots, k_m)$ and by abuse of notation write $s(k_1, \ldots, k_m)$ for the number of submodules of  a module $M$ of shape $(k_1, \ldots, k_m)$. \\
Note here that  $s(M)$ only depends on the shape of $M$: for any string module $M$, the number of submodules is equal to the number of quotients, since submodules $M'$ are in bijection to quotients $M''$ by the short exact sequence $0 \xrightarrow{} M' \xrightarrow{} M \xrightarrow{} M'' \xrightarrow{} 0$.
Suppose that $M$ is of shape $(k_1, \ldots, k_m)$ and starts with a peak. Denote by $M^\vee$ a module of the same shape but starting with a valley. Then $s(M)$ equals the number of quotients of $M$. The quotients of $M$ are clearly in bijection to the submodules of $M^\vee$, thus $s(M)=s(M^\vee)$.

\begin{figure}[h]
\includegraphics[width=\textwidth]{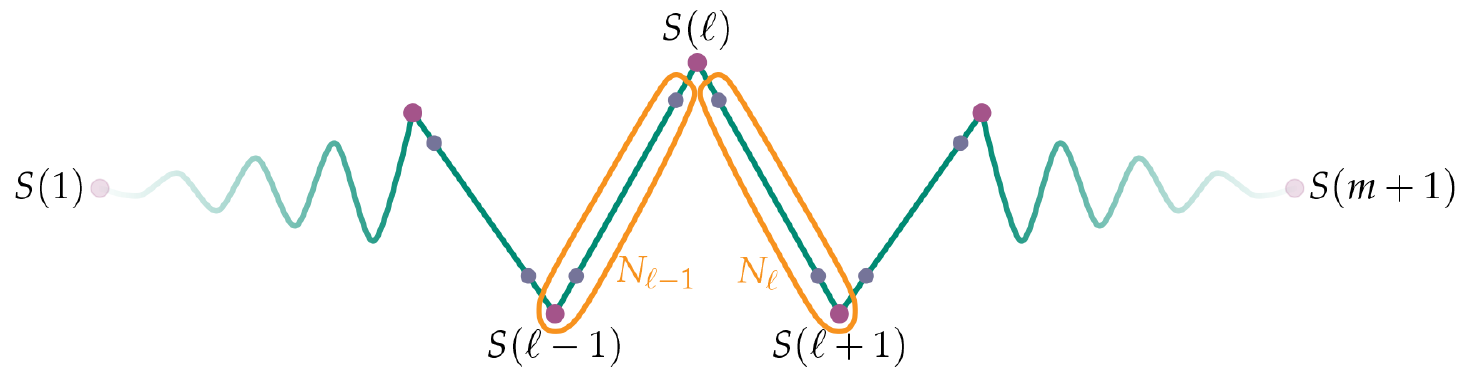}
\caption{A string module $M$ of shape $(k_1, \ldots, k_m)$ with legs $N_l$.}
\label{F:string}
\end{figure}

 Note that the number of arguments of $s(M)$ depends on  the number of legs of $M$. Sometimes we need more information about $M$, namely the orientations and particular simples in the support of the legs, then we write
$$M=(1_1 \leftarrow 1_2 \leftarrow \cdots \leftarrow 1_{k_1} \leftarrow 2_1 \rightarrow 2_2 \rightarrow \cdots \rightarrow 2_{k_2} \rightarrow 3_1 \leftarrow 3_2 \leftarrow \cdots),$$
if $M$ starts with a peak, i.e.\ $S(2)$ is a peak. 
The $i_j$ are contained in $Q_0$. Note that here $S(1)=S_{1_1}$, $S(2)=S_{2_1}$, etc. 
Written differently, see Fig.~\ref{F:stringdetail}:

\begin{figure}[h]
\includegraphics[scale=0.85]{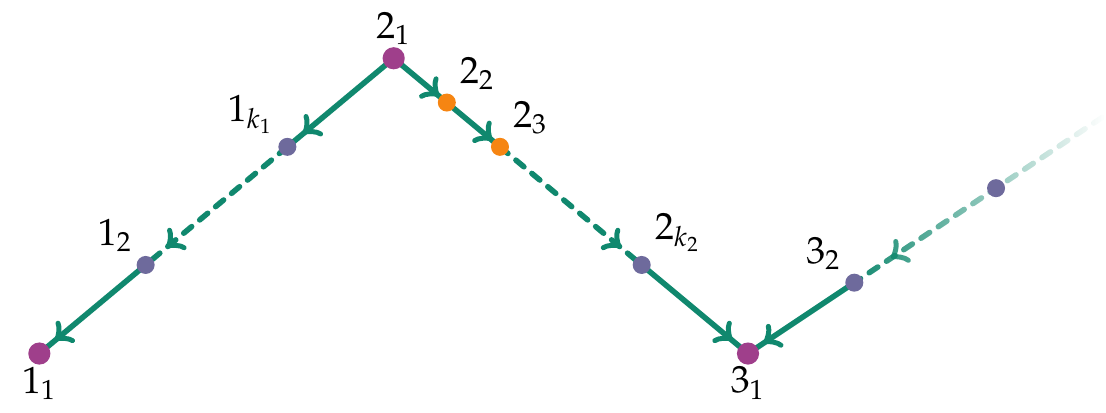}
\caption{Starting at $S(1)=S_{1_1}$. Here $i_j$ denotes the $j$-th simple in the $i$-th leg of $M$.}
\label{F:stringdetail}
\end{figure}

\vspace{0.2cm}
For the number of submodules of $M$ of shape $(k_1, \ldots, k_m)$ we will first derive a recursive formula in terms of certain subrepresentations of $M$ in Lemma \ref{Lem:recursion} and then an explicit formula, only depending on the numbers $k_i$ in Thm.~\ref{Prop:explicitsubmodules}. The proof of the following lemma is straightforward:

\begin{lemma} \label{Lem:uniserial}
If $M$ is uniserial, that is, it is of the shape $(k_1)$ for some $k_1 \in \mathbb{N}$ then $s(M)=k_1+2=l(M)+1$. 
In particular, if $M$ is simple, then $s(M)=2$. 
\end{lemma}

For modules $M$ with $m \geq 2$ one can express the number of submodules of $M$ in terms of submodules $M_{S(2)}$ and $\overline{M}_{S(2)}$ in $M$. Therefore define the following: if $S(1)$ is a valley, let $M_{S(2)}$ be the maximal (proper) indecomposable submodule of $M$ containing $S(m+1)$ 
and $\overline{M}_{S(2)}$ the maximal (proper) indecomposable 
submodule of $M_{S(2)}$ containing $S(m+1)$. If $S(1)$ is a peak, define $M_{S(2)}$, 
$\widetilde{M}_{S(2)}$ dually as quotients: let $M_{S(2)}$ be the maximal (proper) quotient of $M$  containing $S(m+1)$ and $\widetilde{M}_{S(2)}$ the maximal (proper) quotient of $M_{S(2)}$ 
containing $S(m+1)$.   

\begin{ex}
Let $M=(1 \leftarrow 2 \rightarrow 3 \rightarrow 4 \rightarrow 5 \leftarrow 6)$. Here $M$ is of shape $(1, 3, 1)$, that is, $k_1=1, k_2=3, k_3=1$.  The submodules of $M$ (without indicated arrows) are: 
\begin{align*} 
&0, & &(3456), &  &(345), & &(456), & &(45) & &(56), & &(5), \\ 
&(1),& &(1)\oplus(3456), &  &(1)\oplus(345), & &(1) \oplus(456),& &(1)\oplus(45), & &(1) \oplus(56), & &(1) \oplus (5), \\
&   & &(12345), & &(123456). & & & & & & &
\end{align*}
These are $16$ submodules. We have 
\[ M_{S(2)}=(3 \rightarrow 4 \rightarrow 5 \leftarrow 6)  \text{ and } \overline{M}_{S(2)}=(4 \rightarrow 5 \leftarrow 6).\]
Then the submodules of $M$ can be partitioned into submodules of $M_{S(2)}$, $(1) \oplus$ submodules of $M_{S(2)}$ and submodules of $M_{S(2)}$ that  contain the simple $3$ glued to the first leg $1 \leftarrow 2$ (these are the two modules $(1 \leftarrow 2 \rightarrow 3 \rightarrow 4 \rightarrow 5)$ and $(1 \leftarrow 2 \rightarrow 3 \rightarrow 4 \rightarrow 5 \leftarrow 6)$). The cardinality of the first set is $2s(M_{S(2)})$. The others can be seen as submodules of $M_{S(2)}$ minus the ones not containing $3$, so there are  $s(M_{S(2)}) - s(\overline{M}_{S(2)})$ many of them. The total number of submodules of $M$ is  $s(M)=3s(M_{S(2)})-s(\overline{M}_{S(2)})=3 \cdot 7 -  5=16$.
\end{ex}

\begin{lemma} \label{Lem:recursion}
The number of submodules of $M$ of shape $(k_1, \ldots, k_m)$ is given by
$$s(M)=(k_1 +1) s(M_{S(2)}) + s(M_{S(2)}) - s(\overline{M}_{S(2)})=(k_1+2)s(M_{S(2)})-s(\overline{M}_{S(2)}),$$
where $M_{S(2)}$ and $\overline{M}_{S(2)}$ are as defined above. 
In particular, if $k_2 >1$, then $M_{S(2)}$ is of shape $(k_2-1, k_3, \ldots, k_m)$,  and $\overline{M}_{S(2)}$ is of shape $(k_2-2,k_3, \ldots, k_m)$. If $k_2=1$, then $M_{S(2)}=(S(3) - \cdots - S(m))$ is of shape $(k_3, \ldots, k_m)$ and $\overline{M}_{S(2)}$ is of shape $(k_4 -1, \ldots, k_m)$.
\end{lemma}

\begin{proof}  Without loss of generality we may assume that ${\begin{smallmatrix} & S(2) & \\N_1 & & N_2\end{smallmatrix}}$ is a peak. \\
Now the number of submodules of $M$ can be split into two parts: the first ones are direct sums of submodules of $M_{S(2)}$ and submodules of $N_1$, that is 
\begin{align*}0 & \oplus \{ \text{submodules of }M_{S(2)} \}, S(1) \oplus \{ \text{submodules of }M_{S(2)} \},   \\
(1_1\leftarrow 1_2) & \oplus \{ \text{submodules of }M_{S(2)} \},  \ldots, N_1 \oplus \{\text{ submodules of }M_{S(2)} \} \ ,
\end{align*}
 which makes a total of $(l(N_1)+1) \cdot s(M_{S(2)})=(k_1+1)s(M_{S(2)})$ submodules. Note here that we counted the zero-module in this part, as $0 \oplus 0$. The remaining submodules are of the form ${\begin{smallmatrix} & S(2) & \\N_1 & & N_2\end{smallmatrix}}$. These are the submodules of $M$ that contain the first simple of $N_2$ below of $S(2)$. One can easily see that if $N \subseteq M$ contains some simple $2_i$ in $N_2$, it has to contain the whole leg $N_2$. Moreover if $N$ is a submodule of $M$ containing $S(2)$, then $N$ also has to contain ${\begin{smallmatrix}S(2) \\N_1\end{smallmatrix}}$ [see this with writing down the quiver representations].  
Submodules of $M$ are then ${\begin{smallmatrix}S(2) \\N_1\end{smallmatrix}}$ ``glued to'' submodules of $M_{S(2)}$ which also contain the simple just below $S(2)$. The number of those submodules is given by $s(M_{S(2)})-s(\overline{M}_{S(2)})$. Thus in total we get 
$$s(M)=(k_1 + 1)s(M_{S(2)}) + \left(s(M_{S(2)})-s(\overline{M}_{S(2)})\right)=(k_1+2)s(M_{S(2)})-s(\overline{M}_{S(2)}).$$
\end{proof}

Before giving an explicit formula we introduce some notation. 

\begin{defi}
Let $\underline{m}=\{1, \ldots, m\}$. Let $I$ be a subset in $\underline{m}$. We may assume that the elements in $I$ are ordered, i.e., $i_1 < i_2 < \cdots$. The \emph{interior} of $I=\{ i_1, \ldots, i_l\}$ are the integers $i_2, \ldots, i_{l-1}$. Then $I \subseteq \underline{m}$ is called \emph{$\underline{m}$-admissible} if ``gaps'' in $I$ come in pairs $i, (i+1)$ in the interior of $I$.  This means in particular that two consecutive numbers $i_j, i_{j+1}$ in $I$,  are either even--odd or odd-even. Note that $\emptyset$ and all $\{ i \} \subseteq \underline{m}$ are admissible. 
\end{defi}

\begin{ex}
For $m=5$ the admissible sets are $12345$, $1234$, $2345$, $123$, $125$, $145$, $345$, $234$, $12$, $14$, $34$, $23$, $25$, $45$, $1$, $2$, $3$, $4$, $5$, $\emptyset$.
\end{ex}

\begin{theorem} \label{Prop:explicitsubmodules} 
Let $M$ be of shape $(k_1, \ldots, k_m)$.  Then 
$$s(M)=\prod_{i=1}^mk_i + \sum_{\stackrel{ I \subset \underline{m}, |I|=m-1,}{I \text{ admissible}}}\left( \prod_{i \in I} k_i \right) +  \sum_{\stackrel{I \subset \underline{m}, |I|=m-2,}{ I \text{ admissible}}}\left( \prod_{i \in I} k_i\right) + \cdots + \sum_{i=1}^m k_i + 2.$$
In a more compact form:
\begin{equation} \label{Eq:submodFormula}
s(M)=\sum_{j=0}^m \sum_{\stackrel{I \subseteq \underline{m}, |I|=m-j,}{ I \text{ admissible}}}\left( \prod_{i \in I}k_i\right) +1 .
\end{equation}
\end{theorem}

\begin{proof}
The proof is by induction on $m$. If $m=1$, then $M$ is uniserial and $s(M)=k_1+2$ by Lemma \ref{Lem:uniserial}. The formula \eqref{Eq:submodFormula} reads  for $\underline{1}=\{1\}$, and all $I \subseteq \underline{1}$ are admissible:
$$k_1 + \prod_{i \in \emptyset}k_i +1 = k_1 +2.$$  
For $m=2$, that is, $M$ is a peak, one can use the recursion formula from lemma \ref{Lem:recursion}: here $M_{S(2)}=N_2=(2_2\to 2_3- \cdots \to 2_{k_2-1} \to S(3))$ and $\overline{M}_{S(2)}=(2_3 \to \cdots \to 2_{k_2-1} \to S(3))$. Both $M_{S(2)}$ and $\overline{M}_{S(2)}$ are uniserial. \  Then
 \begin{align*}
s(M) & =(k_1+2)s(M_{S(2)})-s(\overline{M}_{S(2)}) = (k_1+2)(k_2+1)-k_2  \\
& = k_1 k_2 + k_1 + k_2 +2 \ ,
\end{align*}
since $s(M_{S(2)})=l(M_{S(2)})+1=k_2+1$ and $s(\overline{M}_{S(2)})=l(\overline{M}_{S(2)})+1=(k_2-1)+1$. Note that if $k_2=1$, then $\overline{M}_{S(2)}$ is the zero-module  and we have $s(\overline{M}_{S(2)})=1$ in this case. \\
For $m=2$, that is, $M$ of shape $(k_1,k_2)$, the right hand side of \eqref{Eq:submodFormula} gives
\[ k_1 k_2 + k_1 +k_2 +2, \]
since all $I \subseteq \underline{2}$ are admissible. So the formula holds for $m=2$. Suppose now that the formula holds for $m-1$ legs and let  $M$ be of shape $(k_1, \ldots, k_m)$, i.e., have $m$ legs. By lemma \ref{Lem:recursion}
\begin{equation} \label{Eq:submodRec}
s(M)=(k_1 +2) s(M_{S(2)}) - s(\overline{M}_{S(2)}).
\end{equation}
We have to consider $2$ cases:
 (i) $k_2 >1$ and (ii) $k_2=1$, since the shape of $M_{S(2)}$ and $\overline{M}_{S(2)}$ are different in the second case. \\
For case (i) we group the right hand side of \eqref{Eq:submodRec} as 
\begin{equation} \label{Eq:RecursionInduction}
\underbrace{(k_1 +1) s(k_2-1, k_3, \ldots, k_m)}_{(*)} + \underbrace{(s(k_2-1, k_3, \ldots, k_m) - s(k_2 -2, k_3, \ldots, k_m))}_{(**)}.
\end{equation}
We may assume that the $k_i$ appearing in the products are ordered with respect to $i$, i.e., $k_{i_1} k_{i_2} \cdots k_{i_s}$ with $i_1 < \ldots< i_s$. From the definition of $\underline{m}$-admissible sets it follows that in each $I$, $i_l$ has to be followed by an odd $i_{l+1}$ and the same for odd $i_l$  has to be followed by an even $i_{l+1}$. For brevity, throughout the rest of this proof, we will simply write admissible, whenever we mean $\underline{m}$-admissible. 
We call an admissible set $I \subseteq \underline{m}$ {\it $2$-admissible} (short: $2-a.$), if it does not contain $1$ and starts with an odd $i>1$ (that is, some $k_i$, where $i$ is odd) or is empty. Thus admissible sets of $\underline{m}$ (or of $(k_1, \ldots, k_m)$) for $(k_2, \ldots, k_m)$ are partitioned into $2$-admissible sets and non-$2$-admissible sets not containing $1$. 
Then for any integer $0 \leq l \leq k_2$ we have using \eqref{Eq:submodFormula}
{\footnotesize
\[  
s(k_2-l,k_3, \ldots, k_m)  =\sum_{j=0}^{m-1}\left( (k_2-l+1) \sum_{\stackrel{I \subseteq \underline{m} \  2-\text{a.},}{|I|=m-j-1}}\prod_{i \in I}k_i + \sum_{\stackrel{I \subseteq \underline{m} \ \text{adm. but not $2$-a.},}{ \{1,2\} \not \subseteq I, |I|=m-j-1}}\prod_{i \in I}k_i\right) +1.
\]
}
Thus $(*)$ in \eqref{Eq:RecursionInduction} is
{\footnotesize 
\begin{align} \label{Eq:induction1}
(k_1 +1) s(k_2-1, k_3, \ldots, k_m)= (k_1 +1)\sum_{j=0}^{m-1} \left( k_2 \sum_{\stackrel{I \subseteq \underline{m} \  2-\text{a.},}{|I|=m-j-1}}\prod_{i \in I}k_i + \sum_{\stackrel{I \subseteq \underline{m} \ \text{adm. but not $2$-a.},}{ \{1,2\} \not \subseteq I, |I|=m-j-1}}\prod_{i \in I}k_i\right) +1.
\end{align}
}
and $(**)$ of \eqref{Eq:RecursionInduction} becomes
\begin{align} \label{Eq:induction2}
s(k_2-1, k_3, \ldots, k_m) - s(k_2 -2, k_3, \ldots, k_m) = \sum_{j=0}^{m-1} \sum_{\stackrel{I \subseteq \underline{m} \  2-\text{a.},}{|I|=m-j-1}}\prod_{i \in I}k_i.
\end{align}
Adding \eqref{Eq:induction1} and \eqref{Eq:induction2} yields
{\footnotesize 
\begin{align}
\sum_{j=0}^{m-1}\left( (k_1k_2 + k_2  + 1) \left( \sum_{\stackrel{I \subseteq \underline{m} \  2-\text{a.},}{|I|=m-j-1}}\prod_{i \in I}k_i\right) + (k_1 +1)\left( \sum_{\stackrel{I \subseteq \underline{m} \ \text{adm. but not $2$-a.},}{ \{1,2\} \not \subseteq I, |I|=m-j-1}}\prod_{i \in I}k_i \right)\right) +k_1 +1,
\end{align}
}
which is precisely the expression on the right hand side of \eqref{Eq:submodFormula}. \\
Case (ii), that is, $k_2=1$, is similar: First note that if $m=3$, then $\overline{M}_{S(2)}$ is the zero-module with $s(\overline{M}_{S(2)})=1$. The right hand side of \eqref{Eq:submodRec} is then 
$$(k_1+2)s(k_3)-s(\overline{M}_{S(2)})=(k_1 +2)(k_3+2)-1=k_1k_3 +2k_3 + 2k_1 +3 \ ,$$
which can be grouped (using $k_2=1$) to
$$k_1k_2k_3 + (k_2k_3 +k_3) +(k_1 k_2 + k_1) + k_2 +2.$$
This is the expression of \eqref{Eq:submodFormula} for $m=3$. \\
Now for $m \geq 4$, the argument is similar to case (i): set $A:=(\sum_{j=0}^{m-2} \sum_{\stackrel{I \subseteq \underline{m} \  2-\text{a.},}{|I|=m-j-2}}\prod_{i \in I}k_i +1)$ and $B:=\sum_{i=0}^{m-2}\sum_{\stackrel{I \subseteq \underline{m} \ \text{adm. but not $2$-a.},}{ \{1,2\} \not \subseteq I, |I|=m-j-2}}\prod_{i \in I}k_i$. Formula \eqref{Eq:submodFormula} can be written as
\begin{equation} k_1 k_2 A + k_1 B +k_2 A + A + B. \label{Eq:FormulaSimple}
\end{equation}
Now write the right hand side of \eqref{Eq:submodRec} as
$$ \underbrace{(k_1 +1) s(k_3, \ldots, k_m)}_{(*)} + \underbrace{s(k_3, \ldots, k_m) - s(k_4 -1, \ldots, k_m)}_{(**)}.$$
Note that $s(k_3, \ldots, k_m)=A+B$. Compute $(*)$, which is equal to $(k_1 +1)(A+B)$. Since $A$ is the sum over all $2$-admissible sets and $k_2=1$, this is also equal to
$$(k_1 +1)(k_2 A +B)=k_1k_2 A + k_2 A + k_1 B + B.$$
Further we get that $(**)$ is equal to 
$$(A+B-B) = A.$$
Adding $(*)$ and $(**)$ yields \eqref{Eq:FormulaSimple}.

\end{proof}

\section{Description of the regions in the frieze}\label{S:Description of the regions in the frieze}

\noindent
{\bf The quiver of a triangulation.} 
We recall here how to get the quiver $Q_{\mathcal T}$ of a triangulation. If $\mathcal T$ is a triangulation of an 
$n+3$-gon, we label the diagonals by $1,2,\dots, n$ and draw an arrow $i\to j$ in case 
the diagonals share an endpoint and the diagonal $i$ can be rotated clockwise to diagonal 
$j$ (without passing through another diagonal incident with the common vertex). 
This is illustrated in Example~\ref{ex:quiver-triangul-flip} and Figure~\ref{fig:triangulations-1-2} 
below. 

Let $B$ be the cluster-tilted algebra associated to $Q_{\mathcal{T}}$.  
For $x$ a vertex of this quiver, we write $P_x$ for the projective $B_T$-module 
of $x$ and $S_x$ for its simple top. 

We then have $B = \End_{\mathcal{C}} T$, where the cluster-tilting object 
\[
T=\bigoplus_{x\in \mathcal{T}} P_x 
\]
in $\mathcal C$. We can extend this to an object in the Frobenius category $\mathcal C_f$ by 
adding the $n+3$ projective-injective summands associated to the boundary segments $[12],[23]$ 
$,\dots,[n+3,1]$ 
of the polygon, 
with irreducible maps between the objects corresponding to diagonals/edges as follows: 
$[i-1,i+1]\to [i,i+1]$, $[i,i+1]\to [i,i+2]$ (\cite{JKS, BKM,DeLuo}). 
We denote the projective-injective associated to $[i,i+1]$ by $Q_{x_i}$. 
Let 
\[
{T}_f=\left( \oplus _{x\in\mathcal T} P_x\right) \oplus \left(Q_{x_1}\oplus \dots\oplus Q_{x_{n+3}}\right)
\]
This is a cluster-tilting object of $\mathcal C_f$ in the sense of~\cite[Section 3]{FuKeller}. 
Given a $B$-module $M$, by abuse of notation, 
we denote the corresponding objects in $\mathcal{C}$ and $\mathcal{C}_f$ by $M$, 
that is $\text{Hom}_{\mathcal{C}}({T}, M)=M$.  In other words, an indecomposable object of 
$\mathcal{C}_f$ is either an indecomposable $B$-module or $Q_{x_i}$ for some 
$i\in\{1,\dots, n+3\}$ 
or of the form $P_x[1]$ for some $x \in \mathcal T$ 

\subsection{Diagonal defines quadrilateral}
$\ $

Let $a$ be a diagonal in the triangulation, $a\in\{1,2,\dots, n\}$. 
This diagonal uniquely defines a quadrilateral formed by diagonals or 
boundary segments. Label them $b,c,d,e$ as in Figure~\ref{fig:local-triang}.  

\begin{figure}
\begin{center}
\begin{minipage}{65mm} \includegraphics[width=45mm]{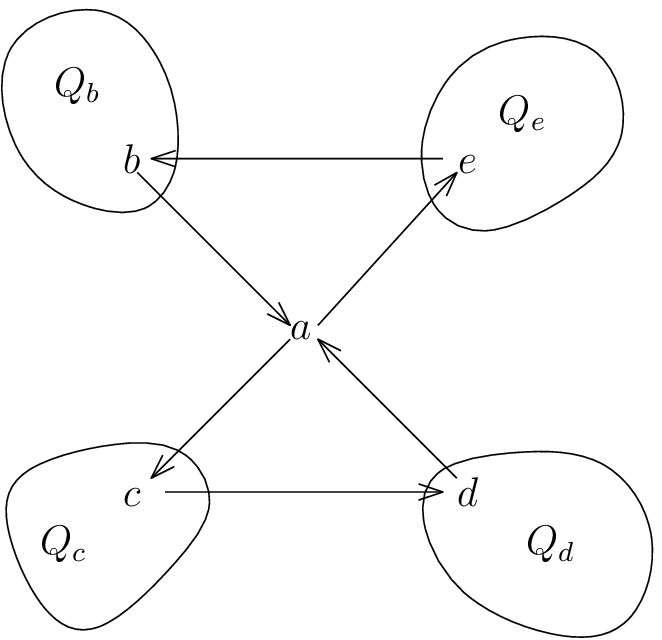}
 \caption{Regions in quiver.}\label{fig:quiver-divided}
\end{minipage}
\hfil
\begin{minipage}{65mm} \includegraphics[width=45mm]{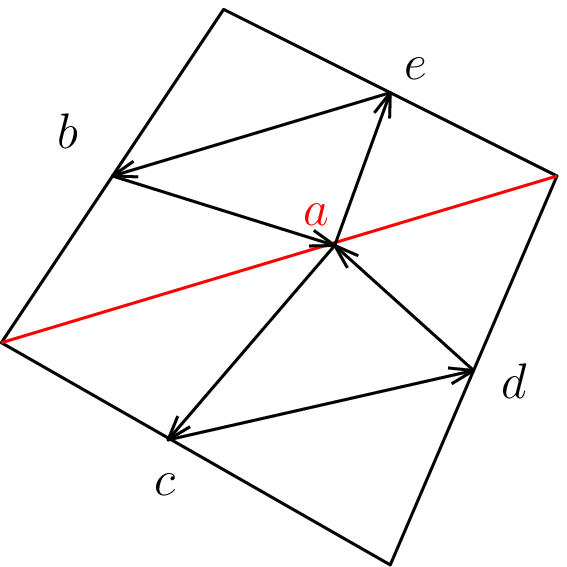} 
\caption{Triangulation around $a$.}
\label{fig:local-triang}
\end{minipage}
\end{center}
\end{figure}

\subsection{Diagonal defines two rays} \label{ssec:diagonal-rays}
$\ $

Consider the entry $1$ of the frieze corresponding to $a$. There are two rays passing through it. We go along 
these rays forwards and backwards until we reach the first entry 1. As the frieze has two rows of 
ones bounding it, we will always reach an entry 1 in each of these four directions. 
Going forwards and upwards: the first occurrence of $1$ corresponds to the diagonal $b$. Down and 
forwards: diagonal $d$. Backwards down from the entry corresponding to $1$: diagonal $c$ and backwards up: 
diagonal $e$. 
If we compare with the coordinate system for friezes of Section~\ref{sec:intro}, the two rays 
through the object corresponding 
to diagonal $a=[kl]$ are the entries $m_{i,l}$ (with $i$ varying) and $m_{k,j}$ (with $j$ varying). 

\begin{ex}
For illustration, we consider the triangulation $\mathcal T=\{[25],[35],[15]\}$ of a hexagon. 
In the frieze pattern associated to $\mathcal T$, we have entries $1$ at 
$a=[25]$, at the diagonals $b=[15]$ and $c=[35]$ as well as at the entries corresponding to edges $d=[23]$ 
and $e=[12]$ (Section~\ref{ssec:tri-frieze}). 
In the figure, the vertex for $a$ is in a circle, the vertices for $b,c,d,e$ are in boxes. 
\[
\tiny{
 \xymatrix@=0.5em{ 
    \sqbox{$e=[21]$}\ar[rd] && [16] && [65] && [54] && [43] && \sqbox{$d=[32]$}    \\
    & [26]\ar[rd] && \sqbox{$b=[15]$}  && [64] && \sqbox{$c=[53]$}\ar[rd] && [42]\ar[ru] \\
    [36] && \circled{[25]}\ar[rd]\ar[ru] && [14] && [63] &&  \circled{[52]}\ar[rd]\ar[ru] && [41]   \\
    & \sqbox{$c=[35]$}\ar[ru] && [24]\ar[rd] && [13] && [62]\ar[ru] && \sqbox{$b=[51]$} \\
     [45] && [34] && \sqbox{$d=[23]$} && \sqbox{$e=[12]$}\ar[ru] && [61] && [56] 
  }}
\]

In the AR quiver of $\mathcal{C}_f$, the entries for $a,b$ and $c$ are the shifted projectives $P_a[1]$, $P_b[1]$ and $P_c[1]$. 
The entries for $d,e$ are the projective-injectives $Q_{x_2}$ and $Q_{x_1}$. 
\end{ex}

In the frieze or in the AR quiver, we give the four segments between the entry $1$ corresponding to 
$a$ and the entries corresponding to $b,c,d$ and $e$ names (see Figure~\ref{fig:regions} for a larger example containing these paths). 
Whereas $a$ is always a diagonal, $b,c,d,e$ may be boundary segments. If $b$ is a diagonal, the ray through $P_a[1]$ goes through 
$P_b[1]$, and if $b$ is a boundary segment, say $b=[i,i+1]$ (with $a=[ij]$) this ray goes through $Q_{x_i}$. By abuse of notation, it will be more convenient to 
write this projective-injective as $P_b[1]$ or as $P_{x_i}[1]$ (if we want to emphasize that it is an object 
of the Frobenius category $\mathcal C_f$ that does not live in $\mathcal C$).

Let $\mathfrak{e}$ and $\mathfrak{c}$ denote the unique sectional paths in $\mathcal{C}_f$ starting at $P_a[1]$ and ending at $P_b[1]$ and 
$P_d[1]$ respectively, but not containing $P_b[1]$ or $P_d[1]$. 
Similarly, let $\mathfrak{b}$ and $\mathfrak{d}$ denote the sectional paths in $\mathcal{C}_f$ starting at $P_e[1]$ and $P_c[1]$ respectively and 
ending at $P_a[1]$, not containing 
$P_e[1]$, $P_c[1]$, see Figure~\ref{fig:regions}.

Note that $b$ and $d$ are opposite sides of the quadrilateral determined by $a$. In particular, the corresponding diagonals do not share endpoints. 
In other words, $P_b[1]$ and $P_d[1]$ do not lie on a common ray in the AR quiver. 
So by the combinatorics of $\mathcal{C}_f$ there exist two distinct sectional paths
starting at $P_b[1], P_d[1]$. These sectional paths both go through 
$S_a$.  Let $\mathfrak{c}^a, \mathfrak{e}^a$  denote these paths starting at $P_b[1]$ and at $P_d[1]$, up to $S_a$, 
but not including $P_b [1], P_d[1]$ respectively. 
Observe that the composition of $\mathfrak{e}$ with $\mathfrak{c}^a$ and the composition of $\mathfrak{c}$ with $\mathfrak{e}^a$ are not sectional, 
see Figure~\ref{fig:regions}. 
Similarly, let $\mathfrak{d}_a,\mathfrak{b}_a$ denote the two distinct sectional paths starting at $S_a$ 
and ending at $P_e[1], P_c[1]$ respectively but not including $P_e[1], P_c[1]$.  Note that the composition of $\mathfrak{c}^a$ with 
$\mathfrak{b}_a$ and the composition of $\mathfrak{e}^a$ with $\mathfrak{d}_a$ are not sectional.

\subsection{Diagonal defines subsets of indecomposables}\label{ssec:regions}
$\ $

For $x$ a diagonal in the triangulation $\mathcal{T}$ and $P_x$ the corresponding projective indecomposable, 
we write $\mathcal X$ for the set of indecomposable $B$-modules having 
a non-zero homomorphism from $P_x$ into them, 
$\mathcal X=\{M \in \mbox{ind}\,B\mid \Hom(P_x,M)\ne 0\}$.  Given a $B$-module $M$, its \emph{support} is the full subquiver $\text{supp}(M)$ 
of $Q_{\mathcal{T}}$ 
generated by all vertices $x$ of $Q_{\mathcal{T}}$ such that $M\in \mathcal{X}$.  
It is well known that the support of an indecomposable module is connected.

If $x$ is a boundary segment, we set $\mathcal X$ to be the empty set (there is no projective 
indecomposable associated to $x$, so there are no indecomposables reached). 

We use the notation above to describe the regions in the frieze. 
Thus, if $x,y$ are diagonals or boundary segments, we write 
$\mathcal X\cap \mathcal Y$ for the indecomposable objects in $\mathcal{C}$ that 
have $x$ and $y$ in their support. 


\begin{figure}
\hspace*{-3cm}\scalebox{1}{\begingroup%
  \makeatletter%
  \providecommand\color[2][]{%
    \errmessage{(Inkscape) Color is used for the text in Inkscape, but the package 'color.sty' is not loaded}%
    \renewcommand\color[2][]{}%
  }%
  \providecommand\transparent[1]{%
    \errmessage{(Inkscape) Transparency is used (non-zero) for the text in Inkscape, but the package 'transparent.sty' is not loaded}%
    \renewcommand\transparent[1]{}%
  }%
  \providecommand\rotatebox[2]{#2}%
  \ifx\svgwidth\undefined%
    \setlength{\unitlength}{561.59197587bp}%
    \ifx\svgscale\undefined%
      \relax%
    \else%
      \setlength{\unitlength}{\unitlength * \real{\svgscale}}%
    \fi%
  \else%
    \setlength{\unitlength}{\svgwidth}%
  \fi%
  \global\let\svgwidth\undefined%
  \global\let\svgscale\undefined%
  \makeatother%
  \begin{picture}(1,0.51235295)%
    \put(0,0){\includegraphics[width=\unitlength]{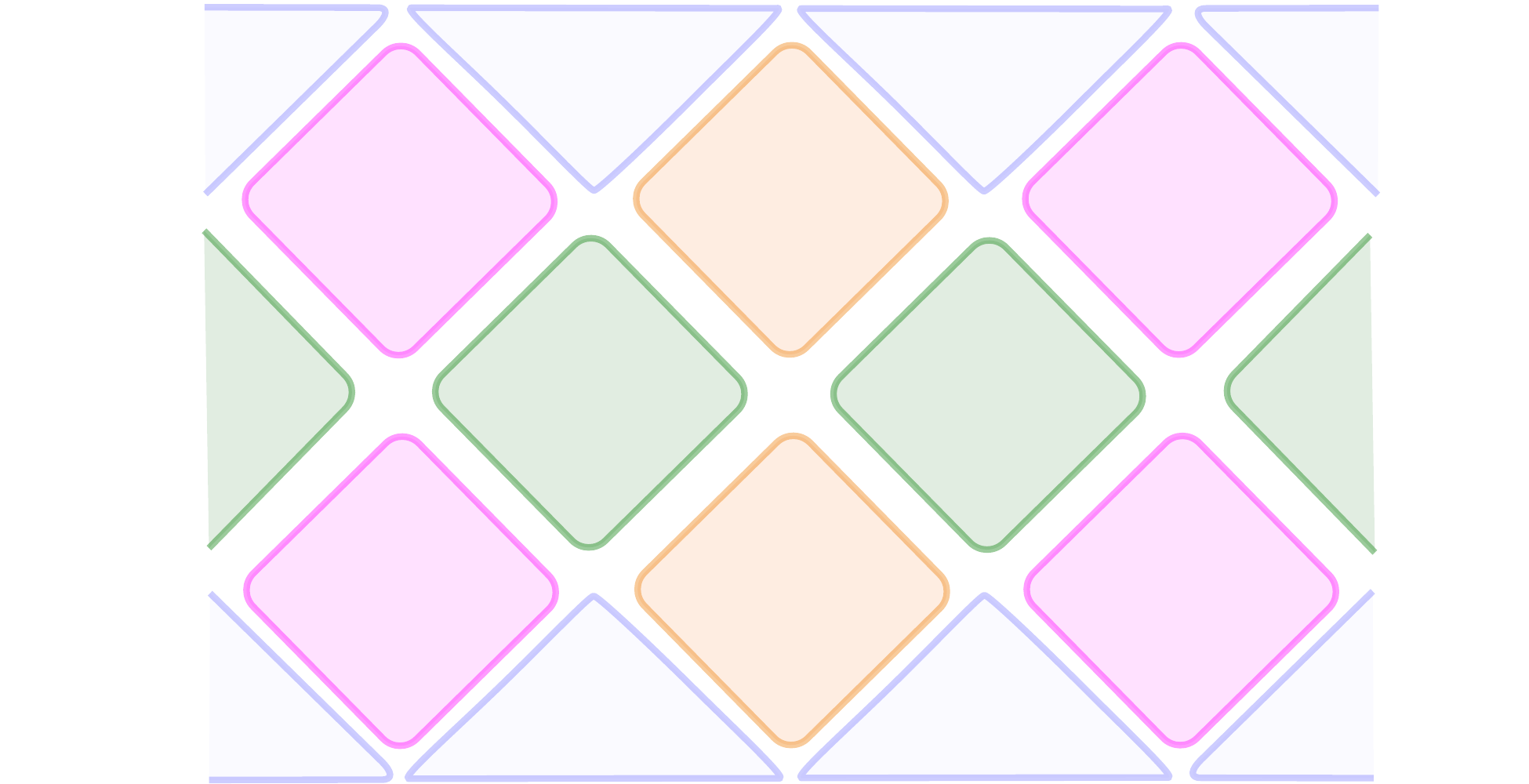}}%
    \put(0.00000507,0.51499728){\color[rgb]{0,0,0}\makebox(0,0)[lt]{\begin{minipage}{0.61508012\unitlength}\raggedright $\xymatrix@!C=0pt@!R=0pt{&& && && \ar@{..}[dddrrr]&& && && \ar@{..}[dddrrr]&& && && \ar@{..}[dddrrr]&&&&&&\\&&&&&&&&&&&&&&&&&&&&&&&&&&&&&\\&&  &&    &&    &&  &&  &&  &&  &&  &&  &&  &&&\\&&&P_d[1] \ar@{..}[uuurrr]\ar[dr] &&  &&  && P_e[1]\ar@{..}[uuurrr] \ar[dr] && && && P_b[1]\ar@{..}[uuurrr]\ar[dr] && && && P_c[1]\\&&&&\ar@{..}[dr]&&&&\ar[ur]&&  \ar@{..}[dr]&&&&\ar[ur] &&\ar@{..}[dr]&&&&\ar[ur]&&\\&&&&&\ar[dr]&&\ar@{..}[ur]&&  &&  \ar[dr]&&  \ar@{..}[ur]&&&& \ar[dr] &&\ar@{..}[ur] &&&&\\&& && && S_a\ar[ur]\ar[dr]&&  &&  &&  P_a[1]\ar[ur]\ar[dr] &&&& && S_a \ar[ur]\ar[dr]\\&&&&&\ar[ur]&&\ar@{..}[dr]&&  &&  \ar[ur]&&  \ar@{..}[dr] && && \ar[ur] &&\ar@{..}[dr] \\&&&&\ar@{..}[ur]&&  &&  \ar[dr]&& \ar@{..}[ur]&&&& \ar[dr]&& \ar@{..}[ur]&& && \ar[dr]\\&&&P_b[1]\ar@{..}[dddrrr]\ar[ur]&&  &&  && P_c[1]\ar@{..}[dddrrr] \ar[ur]&& && && P_d[1] \ar@{..}[dddrrr]\ar[ur]&& && &&P_e[1]\\&& &&  &&  &&  &&  &&   &&   &&  &&  &&  &&\\&&&  &&  && &&  &&   && &&  &&  &&  &&\\&& && && \ar@{..}[uuurrr]&&  && && \ar@{..}[uuurrr]  &&  &&  &&  \ar@{..}[uuurrr]&&&&&&}$\\ \end{minipage}}}%
    \put(-0.09258756,0.16705617){\color[rgb]{0,0,0}\makebox(0,0)[lb]{\smash{
}}}%
    \put(0.35459248,0.26771578){\color[rgb]{0,0,0}\makebox(0,0)[lt]{\begin{minipage}{0.12792571\unitlength}\raggedright $\mathcal{B}\cap\mathcal{D}$\end{minipage}}}%
    \put(0.61766378,0.26645273){\color[rgb]{0,0,0}\makebox(0,0)[lt]{\begin{minipage}{0.12792571\unitlength}\raggedright $\mathcal{C}\cap\mathcal{E}$\end{minipage}}}%
    \put(0.4892344,0.14054161){\color[rgb]{0,0,0}\makebox(0,0)[lt]{\begin{minipage}{0.12792571\unitlength}\raggedright $\mathcal{C}\cap\mathcal{D}$\end{minipage}}}%
    \put(0.49175265,0.389342){\color[rgb]{0,0,0}\makebox(0,0)[lt]{\begin{minipage}{0.12792571\unitlength}\raggedright $\mathcal{B}\cap\mathcal{E}$\end{minipage}}}%
    \put(0.74407853,0.39740031){\color[rgb]{0,0,0}\makebox(0,0)[lt]{\begin{minipage}{0.12792571\unitlength}\raggedright $\mathcal{B}\cap\mathcal{C}$\end{minipage}}}%
    \put(0.74760404,0.13852703){\color[rgb]{0,0,0}\makebox(0,0)[lt]{\begin{minipage}{0.12792571\unitlength}\raggedright $\mathcal{D}\cap\mathcal{E}$\end{minipage}}}%
    \put(0.23942673,0.13651245){\color[rgb]{0,0,0}\makebox(0,0)[lt]{\begin{minipage}{0.12792571\unitlength}\raggedright $\mathcal{B}\cap\mathcal{C}$\end{minipage}}}%
    \put(0.23489393,0.38732741){\color[rgb]{0,0,0}\makebox(0,0)[lt]{\begin{minipage}{0.12792571\unitlength}\raggedright $\mathcal{D}\cap\mathcal{E}$\end{minipage}}}%
    \put(0.83020174,0.26695638){\color[rgb]{0,0,0}\makebox(0,0)[lt]{\begin{minipage}{0.12792571\unitlength}\raggedright $\mathcal{B}\cap\mathcal{D}$\end{minipage}}}%
    \put(0.14323064,0.26544545){\color[rgb]{0,0,0}\makebox(0,0)[lt]{\begin{minipage}{0.12792571\unitlength}\raggedright $\mathcal{C}\cap\mathcal{E}$\end{minipage}}}%
    \put(0.4442412,0.32664216){\color[rgb]{0,0,0}\makebox(0,0)[lt]{\begin{minipage}{0.17829014\unitlength}\raggedright $\mathfrak{b}$\end{minipage}}}%
    \put(0.58013239,0.32975184){\color[rgb]{0,0,0}\makebox(0,0)[lt]{\begin{minipage}{0.17829014\unitlength}\raggedright $\mathfrak{e}$\end{minipage}}}%
    \put(0.44213378,0.19552361){\color[rgb]{0,0,0}\makebox(0,0)[lt]{\begin{minipage}{0.17829014\unitlength}\raggedright $\mathfrak{d}$\end{minipage}}}%
    \put(0.56703763,0.20458922){\color[rgb]{0,0,0}\makebox(0,0)[lt]{\begin{minipage}{0.17829014\unitlength}\raggedright $\mathfrak{c}$\end{minipage}}}%
    \put(0.83094734,0.33050034){\color[rgb]{0,0,0}\makebox(0,0)[lt]{\begin{minipage}{0.17829014\unitlength}\raggedright $\mathfrak{b}_a$\end{minipage}}}%
    \put(0.30917164,0.20761108){\color[rgb]{0,0,0}\makebox(0,0)[lt]{\begin{minipage}{0.17829014\unitlength}\raggedright $\mathfrak{b}_a$\end{minipage}}}%
    \put(0.31924453,0.33150763){\color[rgb]{0,0,0}\makebox(0,0)[lt]{\begin{minipage}{0.17829014\unitlength}\raggedright $\mathfrak{d}_a$\end{minipage}}}%
    \put(0.82526584,0.20883479){\color[rgb]{0,0,0}\makebox(0,0)[lt]{\begin{minipage}{0.17829014\unitlength}\raggedright $\mathfrak{d}_a$\end{minipage}}}%
    \put(0.19067378,0.19976919){\color[rgb]{0,0,0}\makebox(0,0)[lt]{\begin{minipage}{0.17829014\unitlength}\raggedright $\mathfrak{c}^a$\end{minipage}}}%
    \put(0.7018593,0.32668761){\color[rgb]{0,0,0}\makebox(0,0)[lt]{\begin{minipage}{0.17829014\unitlength}\raggedright $\mathfrak{c}^a$\end{minipage}}}%
    \put(0.18612733,0.33071676){\color[rgb]{0,0,0}\makebox(0,0)[lt]{\begin{minipage}{0.17829014\unitlength}\raggedright $\mathfrak{e}^a$\end{minipage}}}%
    \put(0.70222389,0.20279106){\color[rgb]{0,0,0}\makebox(0,0)[lt]{\begin{minipage}{0.17829014\unitlength}\raggedright $\mathfrak{e}^a$\end{minipage}}}%
  \end{picture}%
\endgroup%
}

  \caption{Regions in the AR quiver determined by $P_a[1]$.}
   \label{fig:regions}
\end{figure}
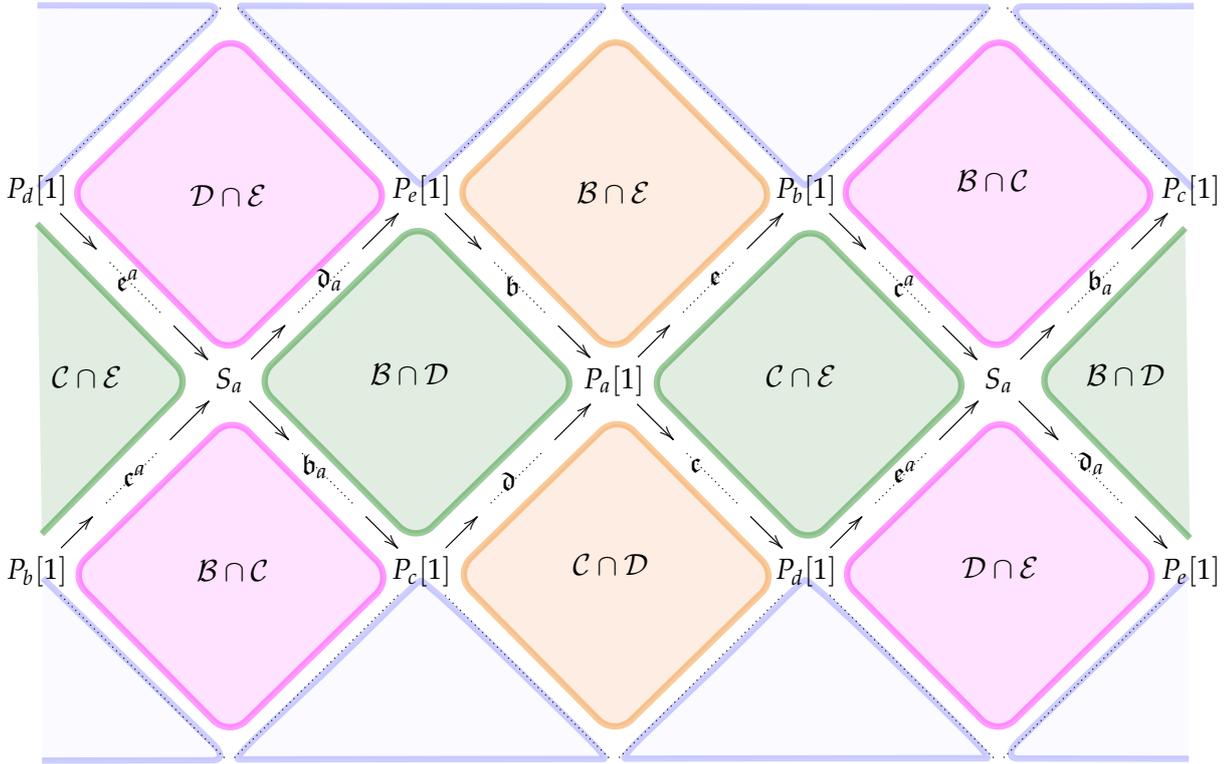

\begin{remark}
Let $M$ be an indecomposable $B$-module in $\mathcal{X}\cap \mathcal{Y}$ such that there exists a (unique) arrow $\alpha:\,x \to y$ in the quiver.  It follows that the right action of the element $\alpha\in B$ on $M$ is nonzero, that is $M\alpha \not=0$. 
\end{remark}

By the remark above we have the following equalities.  Note that none of the modules below are supported at $a$,  
because the same remark would imply that such modules are supported on the entire 3-cycle in $Q_{\mathcal T}$ containing $a$.  
However, this is impossible as the composition of any two arrows in a 3-cycles is zero in $B$.  
We have 

$$\mathcal{B}\cap \mathcal{E} = \{ M \in \text{ind}\,B \mid M \text{ is supported on } e\to b \}$$

$$\mathcal{C}\cap \mathcal{D} = \{ M \in \text{ind}\,B \mid M \text{ is supported on } c\to d \}$$

Moreover, since the support of an indecomposable $B$-module forms a connected subquiver of $Q$, we also have the following equalities.  

$$\mathcal{B}\cap \mathcal{C} = \{ M \in \text{ind}\,B \mid M \text{ is supported on } b\to a \to c \}$$

$$\mathcal{D}\cap \mathcal{E} = \{ M \in \text{ind}\,B \mid M \text{ is supported on } d\to a \to e \}$$

$$\mathcal{B}\cap \mathcal{D} = \{ M \in \text{ind}\,B \mid M \text{ is supported on } b\to a  \leftarrow d \}$$

$$\mathcal{C}\cap \mathcal{E} = \{ M \in \text{ind}\,B \mid M \text{ is supported on } c \leftarrow a  \to e \}$$

Finally, using similar reasoning it is easy to see that the sets described above are disjoint.  Next we describe modules lying on sectional paths defined in section \ref{ssec:diagonal-rays}. 
First, consider sectional paths starting or ending in $P_a[1]$, then we claim that  
 \[
 \mathfrak{i} = \{ M \in \text{ind}\,B \mid i \in \text{supp} (M) \subset Q_i \}\cup \{P_a[1]\}
 \]
for all $i \in \{ b, c, d, e\}$, for $Q_i$ the subquiver of $Q$ containing $i$, as in Figure~\ref{fig:quiver-divided}. 
%
We show that the claim holds for $i=b$, but similar arguments can be used to justify the remaining cases.  Note, that it suffices to show that a module $M \in \mathfrak{b}$ is supported on $b$ but it is not supported on $e$ or $a$.   By construction the sectional path $\mathfrak{b}$ starts at $P_e[1]$, so $0=\text{Hom}  (\tau^{-1} P_e[1], M)=\text{Hom} (P_e, M)$.  On the other hand, $\mathfrak{b}$ ends at $P_a[1]$, so $0=\text{Hom} (M, \tau P_a[1])=\text{Hom}(M, I_a)$, where $I_a$ is the injective $B$-module at $a$.  This shows that $M$ is not supported at $e$ or $a$.   
Finally, we can see from Figure~\ref{fig:regions} that $M$ has a nonzero morphism into 
$\tau P_b[1] = I_b$, provided that $b$ is not a boundary segment.  However, if $b$ is a 
boundary segment, then 
$\mathfrak{b} \cap \mathrm{Ob}(\text{mod}\,B) = \emptyset$ 
and we have $\mathfrak{b} = \{P_a[1]\}$. 
Conversely, it also follows from Figure~\ref{fig:regions} that every module $M$ supported on $b$ and some other vertices of $Q_b$ lies on $\mathfrak{b}$.  This shows the claim.   

Now consider sectional paths starting or ending in $S_a$.  Using similar arguments as above we see that

$$\mathfrak{i}^a = \{ M \in \text{ind}\,B \mid a    \in \text{supp} (M) \subset Q_i^a \}$$
for  $i \in \{c, e\}$ and 

$$\mathfrak{i}_a = \{ M \in \text{ind}\,B \mid a \in \text{supp} (M) \subset Q_i^a \} $$ 
for $i \in \{b, d\}$,  where $Q^a_i$ is the full subquiver of $Q$ on vertices of $Q_i$ and the vertex $a$.  

Finally, we define $\mathcal{F}$ to be the set of indecomposable objects of $\mathcal{C}_f$ that do not belong to 
\[
\mathcal{A}\cup \mathcal{B}\cup \mathcal{C}\cup \mathcal{D}\cup \mathcal{E} \cup \{P_a[1]\}.
\]
The region $\mathcal{F}$ is a succession of wings 
in the AR quiver of $\mathcal C_f$, with 
peaks at the $P_x[1]$ for $x\in \{b,c,d,e\}$.  That is, in the AR quiver of $\mathcal C_f$ consider two neighboured copies 
of $P_a[1]$ with the four vertices $P_b[1]$, $P_c[1]$, $P_d[1]$, $P_e[1]$. Then the indecomposables of $\mathcal{F}$ 
are the vertices in the triangular regions below these four vertices, including them (as their peaks). 
By the glide symmetry, we also have these regions at the top of the frieze. In Figure~\ref{fig:regions}, 
the wings 
are the shaded unlabelled regions at the boundary. 
It corresponds to the diagonals inside and bounding the shaded regions in Figure~\ref{fig:triangulations-1-2}.
We will see in the next section that objects in $\mathcal{F}$ do not change under mutation 
of ${T}_f$ at $P_a[1]$.  

\begin{ex}\label{ex:quiver-triangul-flip}
We consider the triangulation $\mathcal{T}$ of a 13-gon, see left hand of Figure~\ref{fig:triangulations-1-2} 
and the triangulation $\mathcal{T}'=\mu_1(\mathcal{T})$ obtained by flipping diagonal $1$.  

The quivers of $\mathcal{T}$ and of $\mathcal{T}'$ are given below. Note that the quiver $Q$ is the same as in Example~\ref{E:category_frieze}.

$$
\xymatrix@!C=0pt@!R=0pt{
&&&&&&7\ar[dl] &&&&&&& &&7\ar[dl]\\
Q: &8&&3\ar[ll]\ar[dr]&&2\ar[ll]\ar[rr] &&6\ar[lu]   &&Q':&8&&3\ar[ll]\ar[dd] &&2 \ar[dl] \ar[rr] &&6\ar[lu]  \\
&&10\ar[ld]&&1\ar[ur]\ar[dl]&&  &&&&&10\ar[ld]&&1\ar[rd] \ar[ul]&&\\
&9\ar[rr]&&5\ar[rr]\ar[lu]&&4\ar[ul] &&    &&&9\ar[rr]&&5\ar[lu]\ar[ur]&&4\ar[uu] && }
$$

Figure~\ref{fig:ARquiver} shows the Auslander-Reiten quiver of the cluster category 
$\mathcal C_f$ for $Q$. 

In Figure~\ref{fig:frieze-entries} (Section~\ref{sec:mutating}), the frieze patterns of $T$ and of $T'$ are given. 

\begin{figure}
\[
\includegraphics[width=6cm]{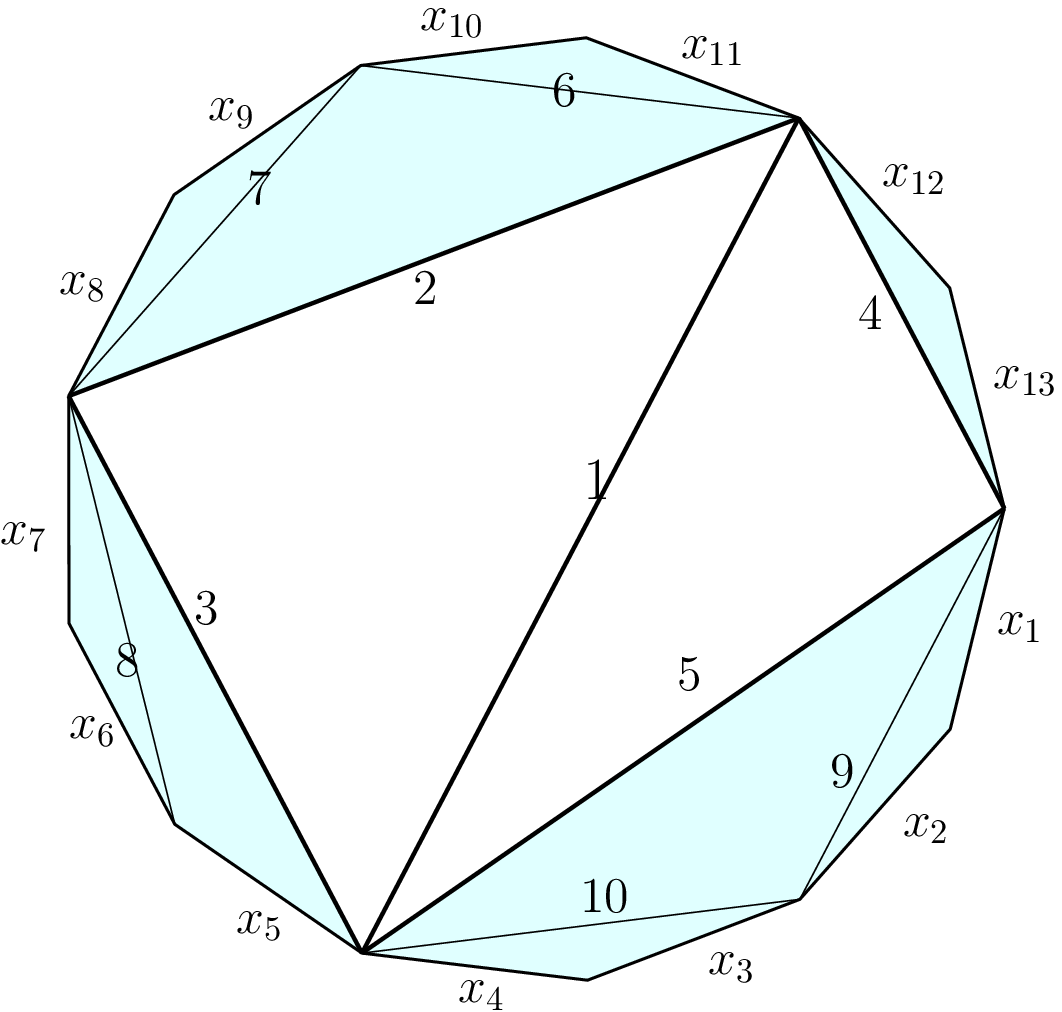}
\hskip 1cm
\includegraphics[width=6cm]{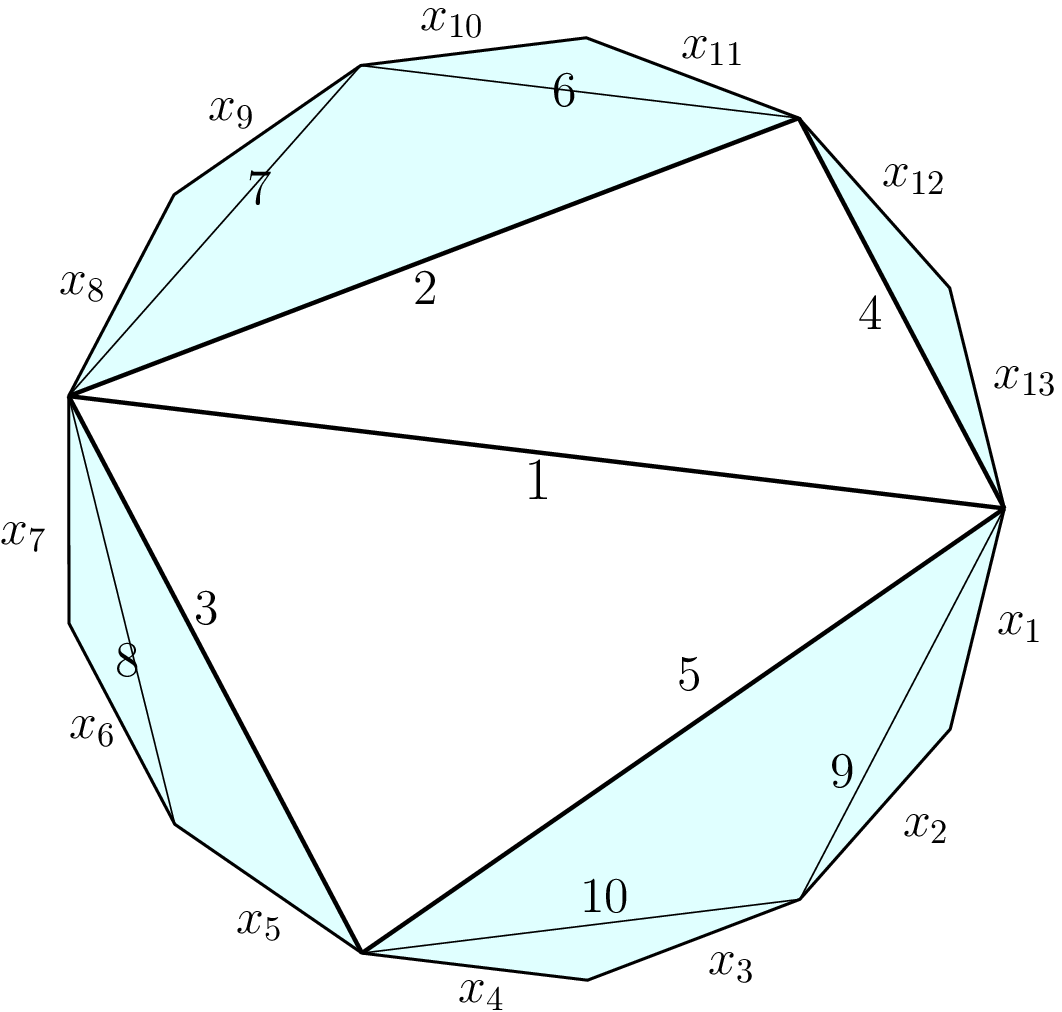}
\]
\caption{Triangulations $\mathcal{T}$ and $\mathcal{T}'=\mu_1(\mathcal{T})$}
\label{fig:triangulations-1-2}
\end{figure}

\begin{figure}
\hspace*{-.5cm}\scalebox{.9}{\begingroup%
  \makeatletter%
  \providecommand\color[2][]{%
    \errmessage{(Inkscape) Color is used for the text in Inkscape, but the package 'color.sty' is not loaded}%
    \renewcommand\color[2][]{}%
  }%
  \providecommand\transparent[1]{%
    \errmessage{(Inkscape) Transparency is used (non-zero) for the text in Inkscape, but the package 'transparent.sty' is not loaded}%
    \renewcommand\transparent[1]{}%
  }%
  \providecommand\rotatebox[2]{#2}%
  \ifx\svgwidth\undefined%
    \setlength{\unitlength}{498.54893649bp}%
    \ifx\svgscale\undefined%
      \relax%
    \else%
      \setlength{\unitlength}{\unitlength * \real{\svgscale}}%
    \fi%
  \else%
    \setlength{\unitlength}{\svgwidth}%
  \fi%
  \global\let\svgwidth\undefined%
  \global\let\svgscale\undefined%
  \makeatother%
  \begin{picture}(1,0.68930642)%
    \put(0,0){\includegraphics[width=\unitlength]{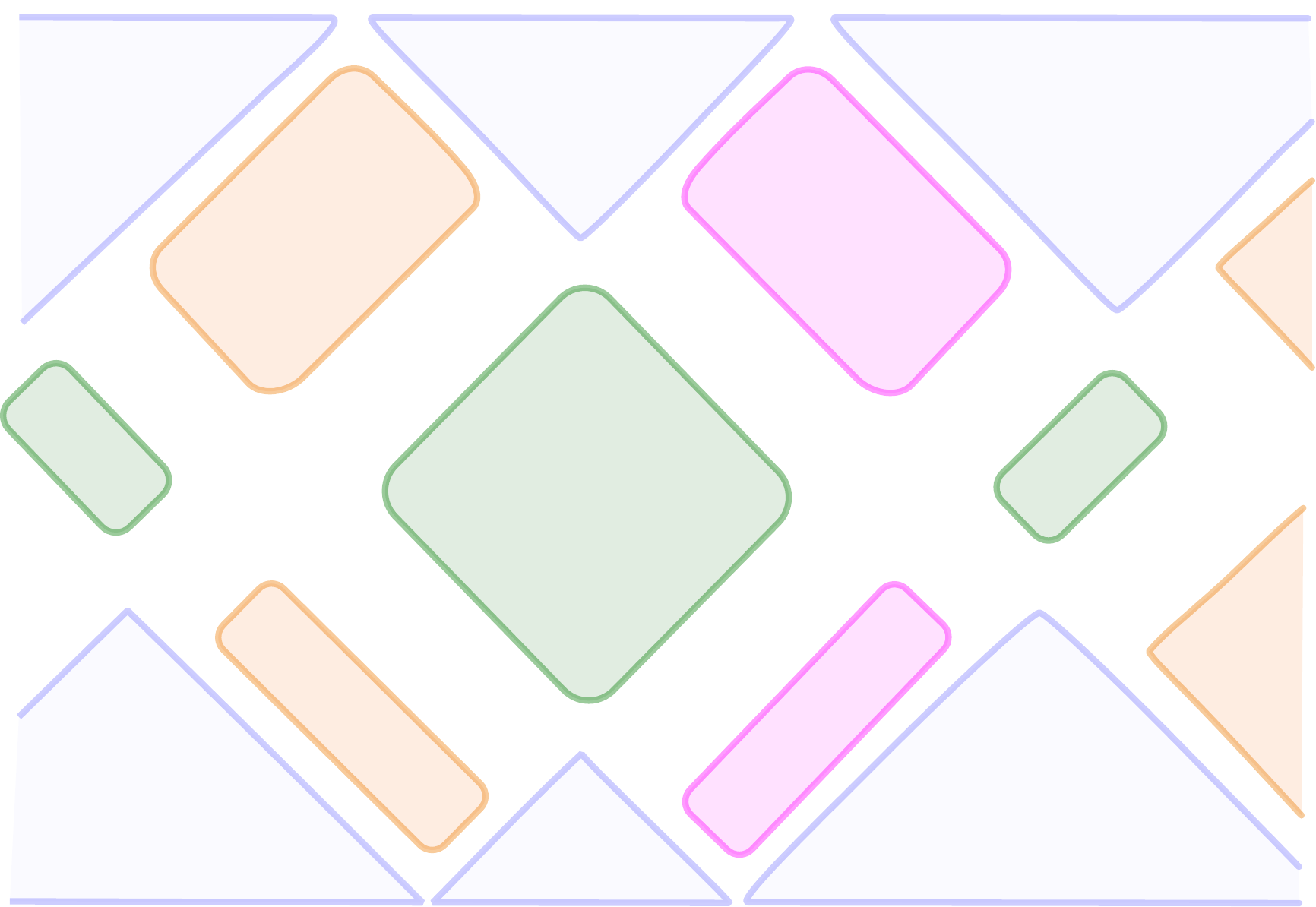}}%
    \put(0.00188615,0.69083277){\color[rgb]{0,0,0}\makebox(0,0)[lt]{\begin{minipage}{0.36649641\unitlength}\raggedright $$\xymatrix@!C=5pt@!R=5pt{&{\begin{smallmatrix}P_{x_9}[1] \end{smallmatrix}}\ar[dr]&&{\begin{smallmatrix}P_{x_{8}}[1]\end{smallmatrix}}\ar[dr]&&{\begin{smallmatrix}P_{x_{7}}[1]\end{smallmatrix}}\ar[dr]&&{\begin{smallmatrix}P_{x_{6}}[1]\end{smallmatrix}}\ar[dr]&&{\begin{smallmatrix}P_{x_{5}}[1]\end{smallmatrix}}\ar[dr]&&{\begin{smallmatrix}P_{x_{4}}[1]\end{smallmatrix}}\ar[dr]&&{\begin{smallmatrix}P_{x_{3}}[1]\end{smallmatrix}}\ar[dr]&&{\begin{smallmatrix}P_{x_{2}}[1]\end{smallmatrix}}\ar[dr]\\{\begin{smallmatrix}6\\7\end{smallmatrix}}\ar[dr]\ar[ur]&&{\begin{smallmatrix}P_7[1]\end{smallmatrix}}\ar[dr]\ar[ur]&&{\begin{smallmatrix}7\\2\\3\\8\end{smallmatrix}}\ar[dr]\ar[ur]&&{\begin{smallmatrix}P_8[1]\end{smallmatrix}}\ar[dr]\ar[ur]&&{\begin{smallmatrix}8\end{smallmatrix}}\ar[dr]\ar[ur]&&{\begin{smallmatrix}3\\1\\5\\10\end{smallmatrix}}\ar[dr]\ar[ur]&&{\begin{smallmatrix}P_{10}[1]\end{smallmatrix}}\ar[dr]\ar[ur]&&{\begin{smallmatrix}10\\9\end{smallmatrix}}\ar[dr]\ar[ur]&&{\begin{smallmatrix}P_9[1]\end{smallmatrix}}\\&{\begin{smallmatrix}6\end{smallmatrix}}\ar[ur]\ar[dr]&&{\begin{smallmatrix}2\\3\\8\end{smallmatrix}}\ar[ur]\ar[dr]&&{\begin{smallmatrix}7\\2\\3\end{smallmatrix}}\ar[ur]\ar[dr]&&{\begin{smallmatrix}P_3[1]\end{smallmatrix}}\ar[ur]\ar[dr]&&{\begin{smallmatrix}3\\8\;1\\\;\;\;5\\\;\;\;10\end{smallmatrix}}\ar[ur]\ar[dr]&&{\begin{smallmatrix}3\\1\\5\end{smallmatrix}}\ar[ur]\ar[dr]&&{\begin{smallmatrix}9\end{smallmatrix}}\ar[ur]\ar[dr]&&{\begin{smallmatrix}10\end{smallmatrix}}\ar[ur]\ar[dr]\\{\begin{smallmatrix}P_2[1]\end{smallmatrix}}\ar[ur]\ar[dr]&&{\begin{smallmatrix}2\\3\;6\\8\;\;\;\end{smallmatrix}}\ar[ur]\ar[dr]&&{\begin{smallmatrix}2\\3\end{smallmatrix}}\ar[ur]\ar[dr]&&{\begin{smallmatrix}7\\2\end{smallmatrix}}\ar[ur]\ar[dr]&&{\begin{smallmatrix}1\\5\\10\end{smallmatrix}}\ar[ur]\ar[dr]&&{\begin{smallmatrix}3\\8\;1\\\;\;\;5\end{smallmatrix}}\ar[ur]\ar[dr]&&{\begin{smallmatrix}3\;\;\;\\1\;9\\5\end{smallmatrix}}\ar[ur]\ar[dr]&&{\begin{smallmatrix}P_5[1]\end{smallmatrix}}\ar[ur]\ar[dr]&&{\begin{smallmatrix}5\\4\;10\end{smallmatrix}}\\&{\begin{smallmatrix}3\\8\end{smallmatrix}}\ar[ur]\ar[dr]&&{\begin{smallmatrix}2\\3\;6\end{smallmatrix}}\ar[ur]\ar[dr]&&{\begin{smallmatrix}2\end{smallmatrix}}\ar[ur]\ar[dr]&&{\begin{smallmatrix}\;\;1\;7\\5\;2\\10\;\;\;\;\;\end{smallmatrix}}\ar[ur]\ar[dr]&&{\begin{smallmatrix}1\\5\end{smallmatrix}}\ar[ur]\ar[dr]&&{\begin{smallmatrix}3\;\;\\8\;1\;9\\\;\;\;5\end{smallmatrix}}\ar[ur]\ar[dr]&&{\begin{smallmatrix}3\\1\end{smallmatrix}}\ar[ur]\ar[dr]&&{\begin{smallmatrix}4\end{smallmatrix}}\ar[ur]\ar[dr]\\{\begin{smallmatrix}\;\;\;8\\4\;3\\1\end{smallmatrix}}\ar[ur]\ar[dr]&&{\begin{smallmatrix}3\end{smallmatrix}}\ar[ur]\ar[dr]&&{\begin{smallmatrix}2\\6\end{smallmatrix}}\ar[ur]\ar[dr]&&{\begin{smallmatrix}1\\5\;2\\10\;\;\;\end{smallmatrix}}\ar[ur]\ar[dr]&&{\begin{smallmatrix}\;\;1\;7\\5\;2\end{smallmatrix}}\ar[ur]\ar[dr]&&{\begin{smallmatrix}9\;1\\5\end{smallmatrix}}\ar[ur]\ar[dr]&&{\begin{smallmatrix}3\\8\;1\end{smallmatrix}}\ar[ur]\ar[dr]&&{\begin{smallmatrix}4\;3\\1\end{smallmatrix}}\ar[ur]\ar[dr]&&{\begin{smallmatrix}P_1[1]\end{smallmatrix}}\\&{\begin{smallmatrix}4\;3\\1\end{smallmatrix}}\ar[ur]\ar[dr]&&{\begin{smallmatrix}P_1[1]\end{smallmatrix}}\ar[ur]\ar[dr]&&{\begin{smallmatrix}1\\5\;2\\10\;\;6\end{smallmatrix}}\ar[ur]\ar[dr]&&{\begin{smallmatrix}1\\5\;2\end{smallmatrix}}\ar[ur]\ar[dr]&&{\begin{smallmatrix}9\;1\;7\\\;5\;2\end{smallmatrix}}\ar[ur]\ar[dr]&&{\begin{smallmatrix}1\end{smallmatrix}}\ar[ur]\ar[dr]&&{\begin{smallmatrix}\;\;3\;4\\8\;1\end{smallmatrix}}\ar[ur]\ar[dr]&&{\begin{smallmatrix}3\end{smallmatrix}}\ar[ur]\ar[dr]\\{\begin{smallmatrix}3\\1\end{smallmatrix}}\ar[ur]\ar[dr]&&{\begin{smallmatrix}4\end{smallmatrix}}\ar[ur]\ar[dr]&&{\begin{smallmatrix}5\\10\end{smallmatrix}}\ar[ur]\ar[dr]&&{\begin{smallmatrix}1\\5\;2\\\;\;\;6\end{smallmatrix}}\ar[ur]\ar[dr]&&{\begin{smallmatrix}9\;1\\\;\;5\;2\end{smallmatrix}}\ar[ur]\ar[dr]&&{\begin{smallmatrix}7\;1\\2\end{smallmatrix}}\ar[ur]\ar[dr]&&{\begin{smallmatrix}4\\1\end{smallmatrix}}\ar[ur]\ar[dr]&&{\begin{smallmatrix}3\\8\end{smallmatrix}}\ar[ur]\ar[dr]&&{\begin{smallmatrix}2\\3\;6\end{smallmatrix}}\\&{\begin{smallmatrix}P_5[1]\end{smallmatrix}}\ar[ur]\ar[dr]&&{\begin{smallmatrix}5\\4\;10\end{smallmatrix}}\ar[ur]\ar[dr]&&{\begin{smallmatrix}5\end{smallmatrix}}\ar[ur]\ar[dr]&&{\begin{smallmatrix}9\;1\\\;\;\;5\;2\\\;\;\;\;\;\;6\end{smallmatrix}}\ar[ur]\ar[dr]&&{\begin{smallmatrix}1\\2\end{smallmatrix}}\ar[ur]\ar[dr]&&{\begin{smallmatrix}\;\;\;4\\7\;1\\2\end{smallmatrix}}\ar[ur]\ar[dr]&&{\begin{smallmatrix}P_2[1]\end{smallmatrix}}\ar[ur]\ar[dr]&&{\begin{smallmatrix}2\\3\;6\\8\;\;\;\end{smallmatrix}}\ar[ur]\ar[dr]\\{\begin{smallmatrix}9\end{smallmatrix}}\ar[ur]\ar[dr]&&{\begin{smallmatrix}10\end{smallmatrix}}\ar[ur]\ar[dr]&&{\begin{smallmatrix}5\\4\end{smallmatrix}}\ar[ur]\ar[dr]&&{\begin{smallmatrix}9\\5\end{smallmatrix}}\ar[ur]\ar[dr]&&{\begin{smallmatrix}1\\2\\6\end{smallmatrix}}\ar[ur]\ar[dr]&&{\begin{smallmatrix}4\\1\\2\end{smallmatrix}}\ar[ur]\ar[dr]&&{\begin{smallmatrix}7\end{smallmatrix}}\ar[ur]\ar[dr]&&{\begin{smallmatrix}6\end{smallmatrix}}\ar[ur]\ar[dr]&&{\begin{smallmatrix}2\\3\\8\end{smallmatrix}}\\&{\begin{smallmatrix}10\\9\end{smallmatrix}}\ar[ur]\ar[dr]&&{\begin{smallmatrix}P_9[1]\end{smallmatrix}}\ar[ur]\ar[dr]&&{\begin{smallmatrix}9\\5\\4\end{smallmatrix}}\ar[ur]\ar[dr]&&{\begin{smallmatrix}P_4[1]\end{smallmatrix}}\ar[ur]\ar[dr]&&{\begin{smallmatrix}4\\1\\2\\6\end{smallmatrix}}\ar[ur]\ar[dr]&&{\begin{smallmatrix}P_6[1]\end{smallmatrix}}\ar[ur]\ar[dr]&&{\begin{smallmatrix}6\\7\end{smallmatrix}}\ar[ur]\ar[dr]&&{\begin{smallmatrix}P_7[1]\end{smallmatrix}}\ar[ur]\ar[dr]\\{\begin{smallmatrix}P_{x_{3}}[1]\end{smallmatrix}}\ar[ur]&&{\begin{smallmatrix}P_{x_{2}}[1]\end{smallmatrix}}\ar[ur]&&{\begin{smallmatrix}P_{x_{1}}[1]\end{smallmatrix}}\ar[ur]&&{\begin{smallmatrix}P_{x_{13}}[1]\end{smallmatrix}}\ar[ur]&&{\begin{smallmatrix}P_{x_{12}}[1]\end{smallmatrix}}\ar[ur]&&{\begin{smallmatrix}P_{x_{11}}[1]\end{smallmatrix}}\ar[ur]&&{\begin{smallmatrix}P_{x_{10}}[1]\end{smallmatrix}}\ar[ur]&&{\begin{smallmatrix}P_{x_{9}}[1]\end{smallmatrix}}\ar[ur]&&{\begin{smallmatrix}P_{x_{8}}[1]\end{smallmatrix}}}$$\end{minipage}}}%
  \end{picture}%
\endgroup%
}

  \caption{AR quiver of the category $\mathcal C_f$ arising from $Q$}
   \label{fig:ARquiver}
\end{figure}
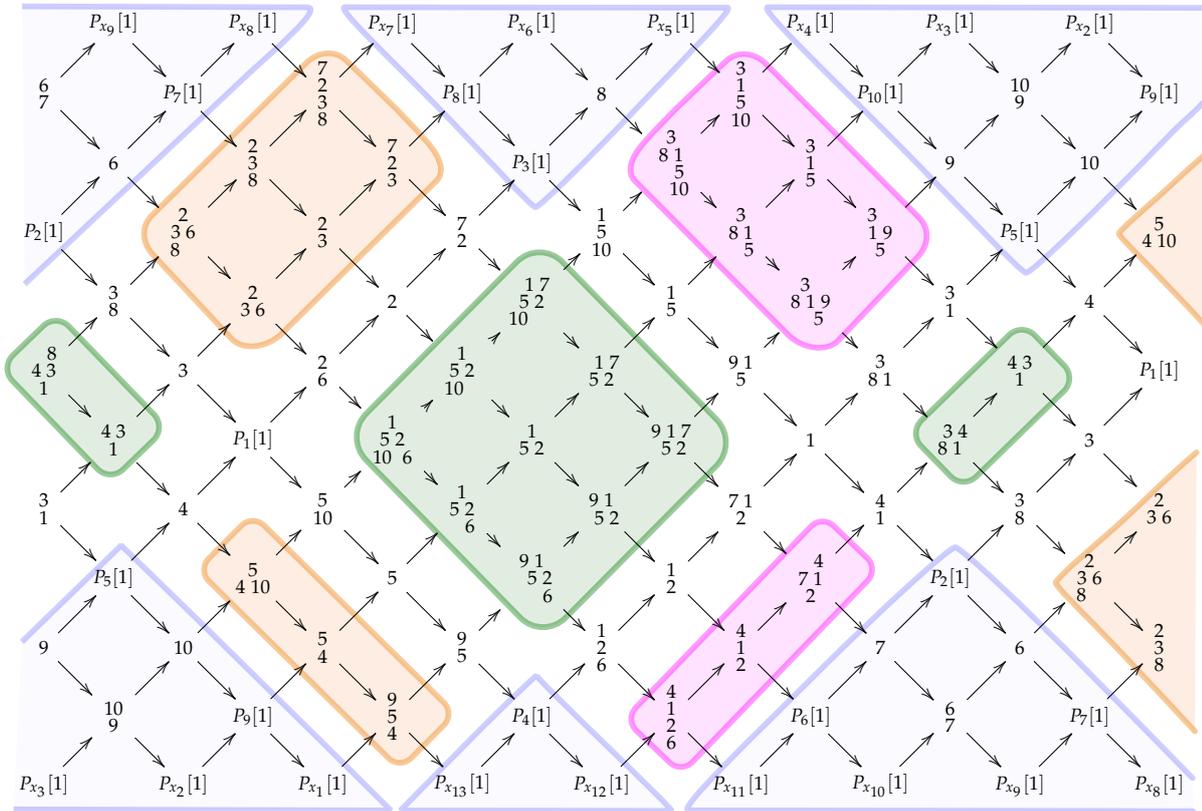

\end{ex}

\begin{figure}
\hspace*{-.5cm}\scalebox{.9}{\begingroup%
  \makeatletter%
  \providecommand\color[2][]{%
    \errmessage{(Inkscape) Color is used for the text in Inkscape, but the package 'color.sty' is not loaded}%
    \renewcommand\color[2][]{}%
  }%
  \providecommand\transparent[1]{%
    \errmessage{(Inkscape) Transparency is used (non-zero) for the text in Inkscape, but the package 'transparent.sty' is not loaded}%
    \renewcommand\transparent[1]{}%
  }%
  \providecommand\rotatebox[2]{#2}%
  \ifx\svgwidth\undefined%
    \setlength{\unitlength}{498.54893649bp}%
    \ifx\svgscale\undefined%
      \relax%
    \else%
      \setlength{\unitlength}{\unitlength * \real{\svgscale}}%
    \fi%
  \else%
    \setlength{\unitlength}{\svgwidth}%
  \fi%
  \global\let\svgwidth\undefined%
  \global\let\svgscale\undefined%
  \makeatother%
  \begin{picture}(1,0.67947309)%
    \put(0,0){\includegraphics[width=\unitlength]{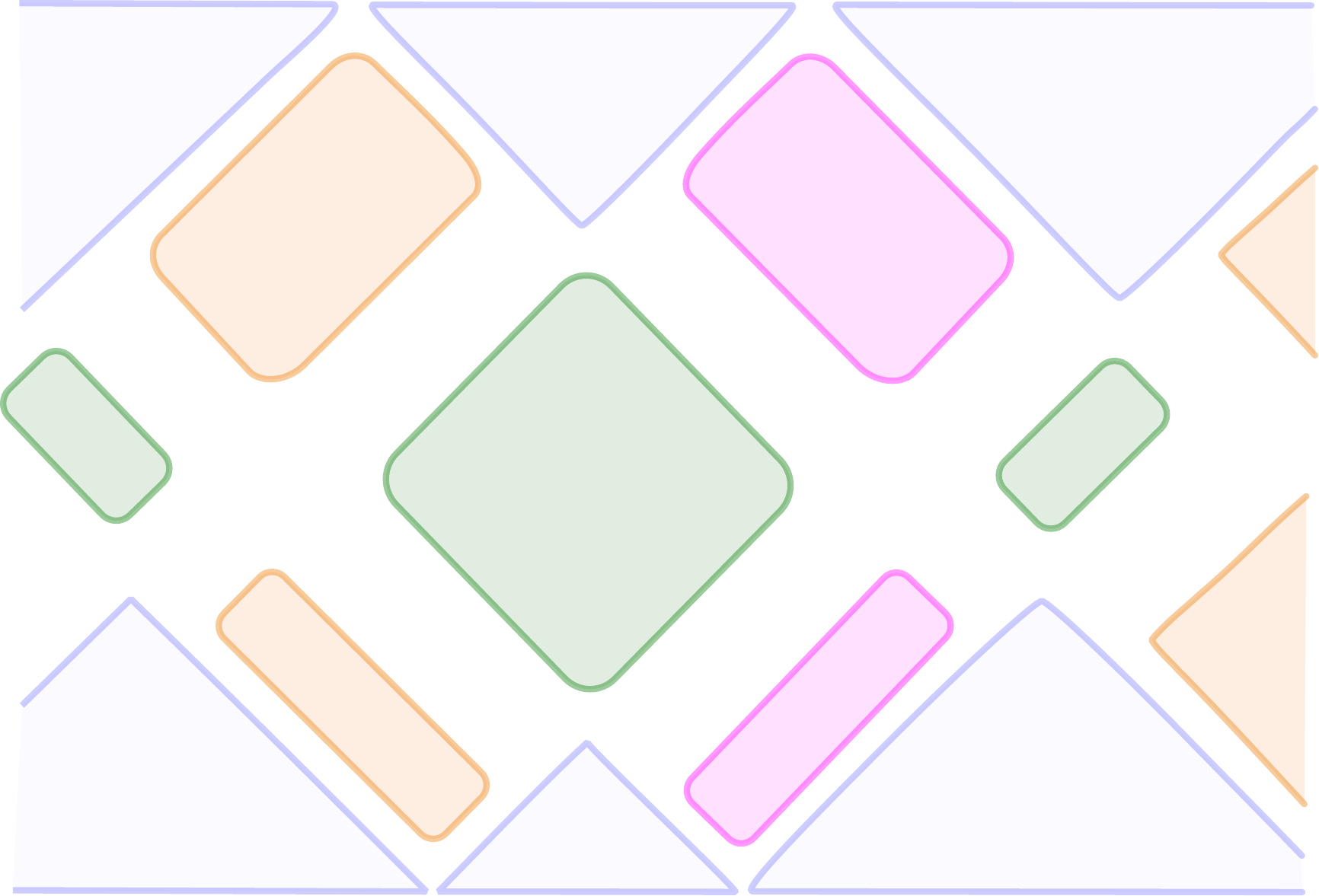}}%
    \put(0.02515416,0.67237921){\color[rgb]{0,0,0}\makebox(0,0)[lt]{\begin{minipage}{0.36649641\unitlength}\raggedright {\xymatrix@!C=5pt@!R=5pt{& {\begin{smallmatrix}1\end{smallmatrix}} && {\begin{smallmatrix}1\end{smallmatrix}} && {\begin{smallmatrix}1\end{smallmatrix}} && {\begin{smallmatrix}1\end{smallmatrix}} && {\begin{smallmatrix}1\end{smallmatrix}} && {\begin{smallmatrix}1\end{smallmatrix}} && {\begin{smallmatrix}1\end{smallmatrix}} && {\begin{smallmatrix}1\end{smallmatrix}}  \\{\begin{smallmatrix}3\end{smallmatrix}} && {\begin{smallmatrix}1\end{smallmatrix}} && {\begin{smallmatrix}5 \\ \textcolor{red}{6}\end{smallmatrix}} && {\begin{smallmatrix}1\end{smallmatrix}} && {\begin{smallmatrix}2\end{smallmatrix}}  &&{\begin{smallmatrix}5 \\ \textcolor{red}{4} \end{smallmatrix}} && {\begin{smallmatrix}1\end{smallmatrix}} && {\begin{smallmatrix}3\end{smallmatrix}} && {\begin{smallmatrix}1\end{smallmatrix}} \\ & {\begin{smallmatrix}2 \end{smallmatrix}} && {\begin{smallmatrix}4\\ \textcolor{red}{5}\end{smallmatrix}} && {\begin{smallmatrix} 4 \\ \textcolor{red}{5} \end{smallmatrix}} && {\begin{smallmatrix}1\end{smallmatrix}} && {\begin{smallmatrix}  9 \\\textcolor{red}{7} \end{smallmatrix}} && {\begin{smallmatrix}4\\ \textcolor{red}{3}\end{smallmatrix}} && {\begin{smallmatrix}2 \end{smallmatrix}} && {\begin{smallmatrix}2\end{smallmatrix}} \\ {\begin{smallmatrix}1\end{smallmatrix}} && {\begin{smallmatrix}7 \\ \textcolor{red}{9} \end{smallmatrix}} && {\begin{smallmatrix}3 \\ \textcolor{red}{4}\end{smallmatrix}} && {\begin{smallmatrix}3 \\ \textcolor{red}{4}\end{smallmatrix}} && {\begin{smallmatrix}4\\\textcolor{red}{3}\end{smallmatrix}} && {\begin{smallmatrix}7 \\ \textcolor{red}{5}\end{smallmatrix}} && {\begin{smallmatrix}7\\\textcolor{red}{5} \end{smallmatrix}} && {\begin{smallmatrix}1\end{smallmatrix}} && {\begin{smallmatrix}5 \\ \textcolor{red}{7}\end{smallmatrix}} \\& {\begin{smallmatrix}3 \\ \textcolor{red}{4}\end{smallmatrix}} &&  {\begin{smallmatrix}5 \\ \textcolor{red}{7}\end{smallmatrix}} && {\begin{smallmatrix}2 \\ \textcolor{red}{3}\end{smallmatrix}} && {\begin{smallmatrix}11 \\ \textcolor{red}{11}\end{smallmatrix}} && {\begin{smallmatrix}3 \\ \textcolor{red}{2}\end{smallmatrix}} && {\begin{smallmatrix}12\\ \textcolor{red}{8}\end{smallmatrix}} && {\begin{smallmatrix}3 \\ \textcolor{red}{2}\end{smallmatrix}} &&  {\begin{smallmatrix}2 \\ \textcolor{red}{3}\end{smallmatrix}}  \\{\begin{smallmatrix}8 \\ \textcolor{red}{7}\end{smallmatrix}} &&  {\begin{smallmatrix}2 \\ \textcolor{red}{3}\end{smallmatrix}} && {\begin{smallmatrix}3 \\ \textcolor{red}{5}\end{smallmatrix}} && {\begin{smallmatrix}7 \\ \textcolor{red}{8}\end{smallmatrix}} && {\begin{smallmatrix}8 \\ \textcolor{red}{7}\end{smallmatrix}} && {\begin{smallmatrix}5 \\ \textcolor{red}{3}\end{smallmatrix}} && {\begin{smallmatrix}5 \\ \textcolor{red}{3}\end{smallmatrix}} &&  {\begin{smallmatrix}5 \\ \textcolor{red}{5}\end{smallmatrix}} && {\begin{smallmatrix}1 \\ \textcolor{red}{\bf 2}\end{smallmatrix}}  \\&  {\begin{smallmatrix}5 \\ \textcolor{red}{5}\end{smallmatrix}} && {\begin{smallmatrix}1 \\ \textcolor{red}{\mathbf 2}\end{smallmatrix}} && {\begin{smallmatrix}10 \\ \textcolor{red}{13}\end{smallmatrix}} && {\begin{smallmatrix}5 \\ \textcolor{red}{5}\end{smallmatrix}} && {\begin{smallmatrix}13 \\ \textcolor{red}{10}\end{smallmatrix}} && {\begin{smallmatrix}2 \\ \textcolor{red}{\mathbf 1}\end{smallmatrix}} && {\begin{smallmatrix}8 \\ \textcolor{red}{7}\end{smallmatrix}} && {\begin{smallmatrix}2 \\ \textcolor{red}{3}\end{smallmatrix}}  \\{\begin{smallmatrix}3 \\ \textcolor{red}{2}\end{smallmatrix}} && {\begin{smallmatrix}2\\ \textcolor{red}{3}\end{smallmatrix}} && {\begin{smallmatrix}3 \\ \textcolor{red}{5}\end{smallmatrix}} && {\begin{smallmatrix}7 \\ \textcolor{red}{8}\end{smallmatrix}} && {\begin{smallmatrix}8 \\ \textcolor{red}{7}\end{smallmatrix}} && {\begin{smallmatrix}5 \\ \textcolor{red}{3}\end{smallmatrix}} && {\begin{smallmatrix}3 \\ \textcolor{red}{2}\end{smallmatrix}} &&  {\begin{smallmatrix}3 \\ \textcolor{red}{4}\end{smallmatrix}} && {\begin{smallmatrix}5 \\ \textcolor{red}{7}\end{smallmatrix}}  \\&  {\begin{smallmatrix}1\end{smallmatrix}} &&  {\begin{smallmatrix}5 \\ \textcolor{red}{7}\end{smallmatrix}} && {\begin{smallmatrix}2 \\ \textcolor{red}{3} \end{smallmatrix}} && {\begin{smallmatrix}11 \\ \textcolor{red}{11}\end{smallmatrix}} && {\begin{smallmatrix}3 \\ \textcolor{red}{2}\end{smallmatrix}} && {\begin{smallmatrix}7 \\ \textcolor{red}{5}\end{smallmatrix}} && {\begin{smallmatrix}1\end{smallmatrix}} &&  {\begin{smallmatrix}7 \\ \textcolor{red}{9}\end{smallmatrix}}  \\{\begin{smallmatrix}2\end{smallmatrix}} &&  {\begin{smallmatrix}  2\end{smallmatrix}} && {\begin{smallmatrix}3 \\ \textcolor{red}{4}\end{smallmatrix}} &&  {\begin{smallmatrix}3\\ \textcolor{red}{4} \end{smallmatrix}} && {\begin{smallmatrix}4 \\ \textcolor{red}{3}\end{smallmatrix}} && {\begin{smallmatrix}4 \\ \textcolor{red}{3}\end{smallmatrix}} && {\begin{smallmatrix}2\end{smallmatrix}} &&  {\begin{smallmatrix}2\end{smallmatrix}} && {\begin{smallmatrix}4\\ \textcolor{red}{5}\end{smallmatrix}}  \\&  {\begin{smallmatrix}3\end{smallmatrix}} &&  {\begin{smallmatrix} 1 \end{smallmatrix}} && {\begin{smallmatrix} 4\\\textcolor{red}{5} \end{smallmatrix}} &&  {\begin{smallmatrix}1\end{smallmatrix}} && {\begin{smallmatrix}5 \\ \textcolor{red}{4}\end{smallmatrix}} &&  {\begin{smallmatrix}1\end{smallmatrix}} &&     {\begin{smallmatrix}3\end{smallmatrix}} &&  {\begin{smallmatrix}1\end{smallmatrix}}  \\{\begin{smallmatrix}1\end{smallmatrix}} &&  {\begin{smallmatrix}1\end{smallmatrix}} &&     {\begin{smallmatrix}1\end{smallmatrix}} &&  {\begin{smallmatrix}1\end{smallmatrix}} &&     {\begin{smallmatrix}1\end{smallmatrix}} &&  {\begin{smallmatrix}1\end{smallmatrix}} &&     {\begin{smallmatrix}1\end{smallmatrix}} &&  {\begin{smallmatrix}1\end{smallmatrix}} &&    {\begin{smallmatrix}1\end{smallmatrix}} &&}}\end{minipage}}}%
  \end{picture}%
\endgroup%
}

  \caption{Frieze pattern of Example~\ref{ex:quiver-triangul-flip}. Red entries: after flip of diagonal 1}
   \label{fig:frieze-entries}
\end{figure}

\section{Mutating friezes}\label{sec:mutating}

The goal of this section is to describe the effect of the flip of a diagonal or equivalently the mutation at 
an indecomposable projective on the associated frieze. We give a formula for computing the effect 
of the mutation 
using the specialised Caldero Chapoton map. 
Let $\mathcal{T}$ be a triangulation of a polygon with associated quiver $Q$. 
The quiver $Q$ looks as in Figure~\ref{fig:quiver-divided}, where the subquivers $Q_b$, $Q_c$, 
$Q_d$, $Q_e$ may 
be empty. Let 
$T=\oplus_{x\in T} P_x$ and  
$B=\End_{\mathcal C} T$ be the associated cluster-tilted algebra. 

Take $a\in \mathcal{T}$ and let $\mathcal{T}'=\mu_a(\mathcal{T})$ be the triangulation obtained from flipping $a$, 
with quiver $Q'=\mu_a(Q)$ (Figure~\ref{fig:flipped-quiver}). 
\begin{figure}
 \includegraphics[width=45mm]{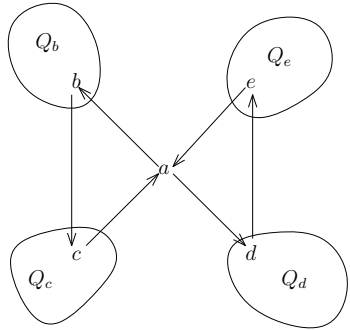}
 \caption{Quiver after flipping diagonal $a$}
 \label{fig:flipped-quiver}
\end{figure}

Let $B'$ be the algebra obtained through this, it is the cluster-tilted algebra for 
$T'=\oplus_{x\in \mathcal{T}'}P_x$. If $M$ is an indecomposable $B$-module, we write $M'$ for 
$\mu_{a} (M)$ in the sense of \cite{DWZ1}. If $M$ is an indecomposable $B$-module, the {\em entry of $M$} in the frieze $F({T})$ is the 
entry at the position of $M$ in the frieze.

\begin{defi}
Let $\mathcal{T}$ be a triangulation of a polygon, $a\in \mathcal{T}$ and $M$ an indecomposable 
object of $\mathcal{C}_f$. Then we define the {\em frieze difference (w.r.t. mutation at $a$)}  
$\delta_a: \ind \mathcal{C}_f \to \mathbb{Z}$ by 
 \[
 \delta_a(M)=\rho_{\mathcal{T}}(M)-\rho_{\mathcal{T}'}(M') \in \mathbb{Z} 
 \]
\end{defi}


In Section~\ref{ssec:mutating-regions} we first describe the effect mutation has on 
the regions in the frieze. This gives us all the necessary tools to compute the frieze difference 
$\delta_a$ (Section~\ref{ssec:computing-delta}). 

\subsection{Mutation of regions}\label{ssec:mutating-regions}
$\ $

Here we describe how mutation affects the regions (Section~\ref{ssec:regions}) of the frieze $F({T})$. 
Let $\mathcal{T},a,B$ and $\mathcal{T}',B'$ be as above. When mutating at $a$, the change in support of 
the indecomposable modules can be described explicitly in terms of the local quiver around 
$a$. This is what we will do here. We first describe the regions in the AR quiver of $\mathcal{C}_f$ 
for $B'$.

If $x$ is a diagonal or a boundary segment, we write 
\[
\mathcal X'=\{M \in \mbox{ind}\,B'\mid \Hom(P_x,M)\ne 0\}
\]
for the indecomposable modules supported on $x$.

After mutating $a$, the regions in the AR quiver are still determined by the projective indecomposables corresponding 
to the framing diagonals (or edges) $b,c,d,e$. The relative positions of $a,b,c,d$ and $e$ have changed, however it follows from \cite{DWZ1} that 
except for vertex $a$ the support of an indecomposable module at all other vertices remains the same.  
Therefore, the regions 
are now described as follows: 

$$\mathcal{B}'\cap \mathcal{E}' = \{ M \in \text{ind}\,B' \mid M \text{ is supported on } e\to a\to b \}$$

$$\mathcal{C}'\cap \mathcal{D}' = \{ M \in \text{ind}\,B' \mid M \text{ is supported on } c\to a \to d \}$$

$$\mathcal{B}'\cap \mathcal{C}' = \{ M \in \text{ind}\,B' \mid M \text{ is supported on } b\to   c \}$$

$$\mathcal{D}'\cap \mathcal{E}'= \{ M \in \text{ind}\,B' \mid M \text{ is supported on } d\to   e \}$$

$$\mathcal{B}'\cap \mathcal{D}' = \{ M \in \text{ind}\,B' \mid M \text{ is supported on } b\leftarrow  a  \to d \}$$

$$\mathcal{C}'\cap \mathcal{E}' = \{ M \in \text{ind}\,B' \mid M \text{ is supported on } c \to a  \leftarrow e \}$$

Under the mutation at $a$, if a module $M$ lies on one of the rays $\mathfrak{b}_a$, $\mathfrak{d}_a$ $\mathfrak{c}^a$ and $\mathfrak{e}^a$ 
then $M'$ is obtained from $M$ by removing support at vertex $a$.  The modules lying on the remaining four rays gain support at vertex $a$ after the mutation.

\subsection{The specialized Caldero Chapoton map under short exact sequences} 

$\ $

Let $B$ be a cluster-tilted algebra of type $A_n$.  We begin by showing three formulas for the number of submodules of a given $M\in\text{ind}\,B$ in terms of the number of submodules of certain quotients and submodules of $M$.  

\begin{lemma}\label{lem1}
Let $M \in\textup{ind}\,B$ such that $M = (\dots z \leftarrow x \to y \dots)$, and consider two indecomposable submodules $M_z, M_y$ of $M$ supported on $z,y$ respectively, such that there exists a short exact sequence 

$$0\to M_z\oplus M_y \to M \to S_x \to 0$$
ending in a simple module $S_x$.  Then 

$$s(M) = s(M_z) s(M_y) + s(\widetilde{M}_z) s(\widetilde{M}_y)$$
where $\widetilde{M}_z, \widetilde{M}_y$ are maximal quotients of $M_z, M_y$ respectively, that are not supported on $z, y$.  
\end{lemma}

\begin{proof}
Recall that $s(M)$ stands for the number of isoclasses of submodules of $M$, however in the following computations we only consider submodules up to isomorphisms.  Therefore, we omit the term isoclasses to simplify the notation.  Also, some of the modules defined above can be zero in which case they admit only one submodule, the zero submodule.

{\small

$$\begin{array}{llcr}
s(M)&= \#\Big\{ {\begin{array}{c} \text{submodules of } M\\ \text{ supported on } x\end{array}}\Big\}&+&\# \Big\{{\begin{array}{c}\text{submodules of } M\\ \text{ not supported on } x\end{array}} \Big\}
\end{array}$$}
The right hand summand can be simplified as follows.  

{\small $
\# \Big\{{\begin{array}{c}\text{submodules of } M\\ \text{ not supported on } x\end{array}} \Big\} \;\;= \;\;  \# \{\text{submodules of } M_z\oplus M_y\} \;\; = \;\; s(M_z)s(M_y)$}

\medskip
\noindent Now consider the left hand summand.  

{\small 
\begin{align*}
\#\Big\{ {\begin{array}{c} \text{submodules of } M\\ \text{ supported on } x\end{array}}\Big\} &= \#\Big\{ {\begin{array}{c} \text{submodules of } M\\ \text{ supported on } x, y, z \end{array}}\Big\} \\
  &= \#\Big\{ {\begin{array}{c} \text{submodules of } M_z\\ \text{ supported on } z\end{array}}\Big\}\,\times \,\#\Big\{ {\begin{array}{c} \text{submodules of } M_y\\ \text{ supported on } y\end{array}}\Big\}\\
  &=\#\Big\{ {\begin{array}{c} \text{quotients of } M_z\\ \text{ not supported on } z\end{array}}\Big\}\,\times\,\#\Big\{ {\begin{array}{c} \text{quotients of } M_y\\ \text{ not supported on } y\end{array}}\Big\}\\
  &=\#\{ \text{quotients of } \widetilde{M}_z\}\,\times\,\#\{ \text{quotients of } \widetilde{M}_y \}\\
&=\#\{ \text{submodules of } \widetilde{M}_z\}\,\times\,\#\{ \text{submodules of } \widetilde{M}_y \}\\
&=s(\widetilde{M}_z)s(\widetilde{M}_y) 
\end{align*}}
\end{proof}

The following statement is the dual version of the lemma above, so we state it without proof.  

\begin{lemma}\label{lem2}
Let $M \in\textup{ind}\,B$ such that $M = (\dots z \to x \leftarrow y \dots)$, and consider two indecomposable quotients $M_z, M_y$ of $M$ supported on $z,y$ respectively, such that there exists a short exact sequence 

$$0\to S_x \to M \to M_z\oplus M_y \to 0$$
ending in a simple module $S_x$.  Then 

$$s(M) = s(M_z) s(M_y) + s(\overline{M}_z) s(\overline{M}_y)$$
where $\overline{M}_z, \overline{M}_y$ are maximal submodules of $M_z, M_y$ respectively, that are not supported on $z, y$.  
\end{lemma}

The next lemma shows the final formula for $s(M)$.  

\begin{lemma}\label{lem3}
Let $M\in\textup{ind}\,B$ such that $M = (\dots x \to y \dots)$, and consider two indecomposable modules $M_x, M_y$ supported on $x, y$ respectively such that there exists a short exact sequence 

$$0\to M_y \to M \to M_x \to 0.$$
Then 

$$s(M)=s(M_y)s(M_x)-s(\widetilde{M}_x)s(\overline{M}_y)$$
where $\widetilde{M}_x$ is a maximal quotient of $M_x$ not supported on $x$, and $\overline{M}_y$ is a maximal submodule of $M_y$ not supported on $y$.  
\end{lemma}

\begin{proof}
In the following computation, we omit the term isoclasses when referring to the number of submodules (up to isomorphisms).  Also, some of the modules defined above can be zero in which case they admit only one submodule, the zero submodule.  Given a module $M$ as in the statement we have

{\small

$$\begin{array}{llcr}
s(M)&= \#\Big\{ {\begin{array}{c} \text{submodules of } M\\ \text{ supported on } x\end{array}}\Big\}&+&\# \Big\{{\begin{array}{c}\text{submodules of } M\\ \text{ not supported on } x\end{array}} \Big\}
\end{array}.$$}
The left hand summand can be reinterpreted as follows. 

{\Small 
\begin{align*}
\#\Big\{ {\begin{array}{c} \text{submodules of } M\\ \text{supported on } x\end{array}}\Big\}  &= \#\Big\{ {\begin{array}{c} \text{submodules of } M\\ \text{supported on } x, y \end{array}}\Big\} \\
  &= \#\Big\{ {\begin{array}{c} \text{submodules of } M_x\\ \text{supported on } x\end{array}}\Big\}\times \#\Big\{ {\begin{array}{c} \text{submodules of } M_y\\ \text{supported on } y\end{array}}\Big\}\\
  &=\#\Big\{ {\begin{array}{c} \text{quotients of } M_x\\ \text{not supported on } x\end{array}}\Big\}\times \Big[\#\{\text{submodules of } M_y\}-\#\Big\{ {\begin{array}{c}\text{submodules of } {M}_y\\ \text{not supported on } y\end{array}}\Big\}\Big]\\
  &=\#\{ \text{quotients of } \widetilde{M}_x\}\times\Big[\#\{\text{submodules of } M_y\}-\#\{ \text{submodules of } \overline{M}_y\}\Big]\\
&=\#\{ \text{submodules of } \widetilde{M}_x\}\,\times\,\big[s(M_y)-s(\overline{M}_y)\big] \\
&=s(\widetilde{M}_x) (s(M_y)-s(\overline{M}_y)) 
\end{align*}}

The right hand summand equals $s(M_y) s(\overline{M}_x)$, where $\overline{M}_x$ is the maximal submodule of $M_x$ not supported on $x$.  Using similar reasoning as above, we have $s(\overline{M}_x) = s(M_x) - s(\widetilde{M}_x)$.  Therefore, combining the two computations above we obtain the desired result. 

\begin{align*}
s(M) &= s(\widetilde{M}_x) (s(M_y)-s(\overline{M}_y)) + s(M_y)(s(M_x) - s(\widetilde{M}_x))\\
&= s(M_y)s(M_x)-s(\widetilde{M}_x)s(\overline{M}_y)
\end{align*} 
\end{proof}

\begin{remark}\label{remark01}
With the notation of Lemma~\ref{lem3} we emphasize two special cases when one of the modules $M_x, M_y$ is a simple module. 
\begin{itemize}

\item[(a)] If $M_y = S_y$ then $M=(\dots x \to y)$, and we have $s(M) = 2 s(M_x) - s(\widetilde{M}_x)$, because $s(S_y)=2$ and $s(\overline{S}_y) = s(0)=1$.  It will be useful to rewrite the above equation in the following way 
$$s(M)-s(M_x) = s(M_x) - s(\widetilde{M}_x) = s(\overline{M}_x)$$
where $\overline{M}_x$ is the maximal submodule of $M_x$ not supported at $x$. 

\item[(b)] Similarly, if $M_x =S_x$ then $M=(x\to y \dots)$, and we have $s(M) = 2 s(M_y) - s(\overline{M}_y).$  Also, we can rewrite this equation as
$$s(M)-s(M_y) = s(M_y) - s(\overline{M}_y)=s(\widetilde{M}_y)$$
where $\widetilde{M}_y$ is the maximal quotient of $M_y$ not supported at $y$.  
\end{itemize}
\end{remark}

We will also need the following result. 

\begin{lemma}\label{lem04}
Let $M_x \in \textup{ind}\,B$ such that $M_x = (x \dots )$, and consider two modules $M_{x_y}, M_{x^z}$ obtained by extending $M_x$ by simple modules $S_y, S_z$, respectively,  such that there exist short exact sequences as follows. 

$$0\to S_y \to M_{x_y} \to M_x \to 0 \hspace{2cm}  0\to M_x \to M_{x^z} \to S_z \to 0$$
Then 

$$s(M_{x_y}) - s(M_x) = 2 s (M_x) - s(M_{x^z}). $$
\end{lemma}

\begin{proof}
Applying Remark~\ref{remark01}(a) to the first short exact sequence with $M= M_{x_y}$, we see that 
$$s(M_{x_y}) = 2 s(M_x) - s(\widetilde{M}_x)$$
where $\widetilde{M_x}$ is the maximal quotient of $M_x$ not supported at $x$.  Now applying Remark~\ref{remark01}(b) to the second short exact sequence with $M = M_{x^z}$ and $M_y= M_x$, we obtain  
$$s(M_{x^z})-s(M_x) = s(\widetilde{M}_x).$$
Solving this equation for $s(\widetilde{M}_x)$ and substituting this resulting expression into the the first equation we obtain the desired formula.  
\end{proof}

\subsection{Mutation of frieze}\label{ssec:computing-delta}
$\ $

We next present the main result of this section, the effect of flip on the generalized 
Caldero Chapoton map, i.e. the description of the frieze difference 
$\delta_a$. 
We begin by introducing the necessary notation. 

Depending on the position of 
an indecomposable object $M$ we define several projection maps sending 
$M$ to objects on the eight rays from Section~\ref{ssec:diagonal-rays}.

%

Let $M\in\text{ind}\,B$, and let $\mathfrak{i}$ be one of the sectional paths defined in 
section~\ref{ssec:diagonal-rays}.  Suppose $M \not\in\mathfrak{i}$, then we denote by 
$M_i$ a module on $\mathfrak{i}$ if there exists a sectional path $M_i \to \dots \to M$ or 
$M\to \dots \to M_i$ in $\mathcal{C}_f$, otherwise we let $M_i=0$.  If $M \in\mathfrak{i}$ 
then we let $M_i=M$.  
In the case when it is well-defined, we call $M_i$ the \emph{projection} of $M$ onto 
the path $\mathfrak{i}$.

\begin{remark}\label{rm-projections}
Given an indecomposable module $M$ and its projection $M_i$, the module $M_i$ can be easily described in terms of the support of $M$ and the location of the ray $\mathfrak{i}$.  For example, if $M\in \mathcal{B}\cap \mathcal{C}$ then $M = (\dots b\to a \to c \dots)$.  Note that $M_{c^a}$ is not supported at $b$, and there exists an $B$-module homomorphism  $M_{c^a}\to M$ that does not factor through any other module lying on the ray $\mathfrak{c}^a$.  This implies that $M_{c^a} = (a\to c \dots )$ is the maximal submodule of $M$ that is not supported at $b$.  Similarly, $M_{b_a} = (\dots b \to a)$ is the maximal quotient of $M$ not supported at $c$.   

On the other hand if we consider $M\in \mathcal{B}\cap \mathcal{E}$ then $M = (\dots e\to b \dots )$.  As above, the projection $M_b$ onto the ray $\mathfrak{b}$ is the largest submodule of $M$ not supported at $e$.  However, $M_{b_a}$ is obtained from $M_b$ by extending it by a simple module $S_a$, thus $M_{b_a}= (a \leftarrow b \dots)$ and there exists a short exact sequence 
$$0\to S_a \to M_{b_a} \to M_b \to 0.$$
Note that unlike in the previous computations $M_{b_a}$ is neither a quotient nor a submodule of $M$.  

In this way given $M$ we can construct the corresponding $M_i$, provided that the projection is well defined.  
\end{remark}

It will be convenient to write these projections in a uniform way. 


\begin{defi}[Projections] 
If $(x,y)$ is one of the pairs 
$\{(b,c),(d,e),(b,e),(c,d)\}$, the region $\mathcal X\cap \mathcal Y$ 
has two paths along its boundary and 
two paths further backwards or forwards met along the two sectional paths through any vertex $M$ of 
$\mathcal X\cap \mathcal Y$.
We call the backwards projection onto the first path 
$\pi_1^-(M)$ and the projection onto the second path 
$\pi_2^-(M)$. The forwards 
projection onto the first path is denoted by $\pi_1^+(M)$ and the one onto the second path $\pi_2^+(M)$. 
\end{defi}
Figure~\ref{fig:projections-case-a} illustrates these projections in the case 
$(x,y)\in \{(b,c),(d,e)\}$.

\smallskip

The remaining two regions will be treated together with the surrounding paths. 

\begin{defi}
The {\em closure of $\mathcal C\cap \mathcal E$} is the 
Hom-hammock 
$$
\overline{\mathcal C\cap \mathcal E}=\ind(\Hom_{\mathcal C_f}(P_a[1],-)\cap \Hom_{\mathcal C_f}(-,S_a))
$$ 
in $\mathcal C_f$ starting at $P_a[1]$ and ending at $S_a$. 
Similarly, 
the {\em closure of $\mathcal B\cap \mathcal D$} is the Hom-hammock 
$$
\overline{\mathcal B\cap \mathcal D}=\ind(\Hom_{\mathcal C_f}(S_a,-)\cap \Hom_{\mathcal C_f}(-,P_a[1]))
$$
in $\mathcal C_f$ starting at $S_a$ and ending at $P_a[1]$. 
For $(x,y)\in \{(c,e),(b,d)\}$, the {\em boundary of $\overline{\mathcal X\cap \mathcal Y}$} 
(or of 
$\mathcal X\cap \mathcal Y$) is 
$\overline{\mathcal X\cap \mathcal Y}\setminus(\mathcal X\cap \mathcal Y)$.  
\end{defi}

Note that $\overline{\mathcal C\cap \mathcal E}$ is the union of 
$\mathcal C\cap \mathcal E$ with the surrounding rays and the shifted projectives 
$\{P_b[1],P_d[1]\}$. Analogously, 
$\overline{\mathcal B\cap \mathcal D}$ contains $\{P_c[1],P_e[1]\}$. 

\begin{figure}
\includegraphics[height=5.5cm]{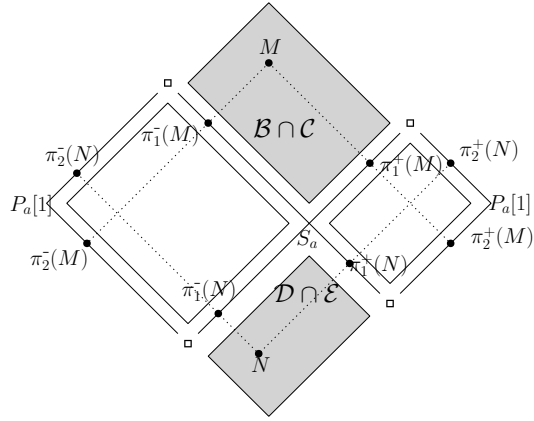}
\caption{Projections for $\mathcal B\cap \mathcal C$, $\mathcal D\cap \mathcal E$}\label{fig:projections-case-a}
\end{figure}

\begin{figure}
\includegraphics[height=5cm]{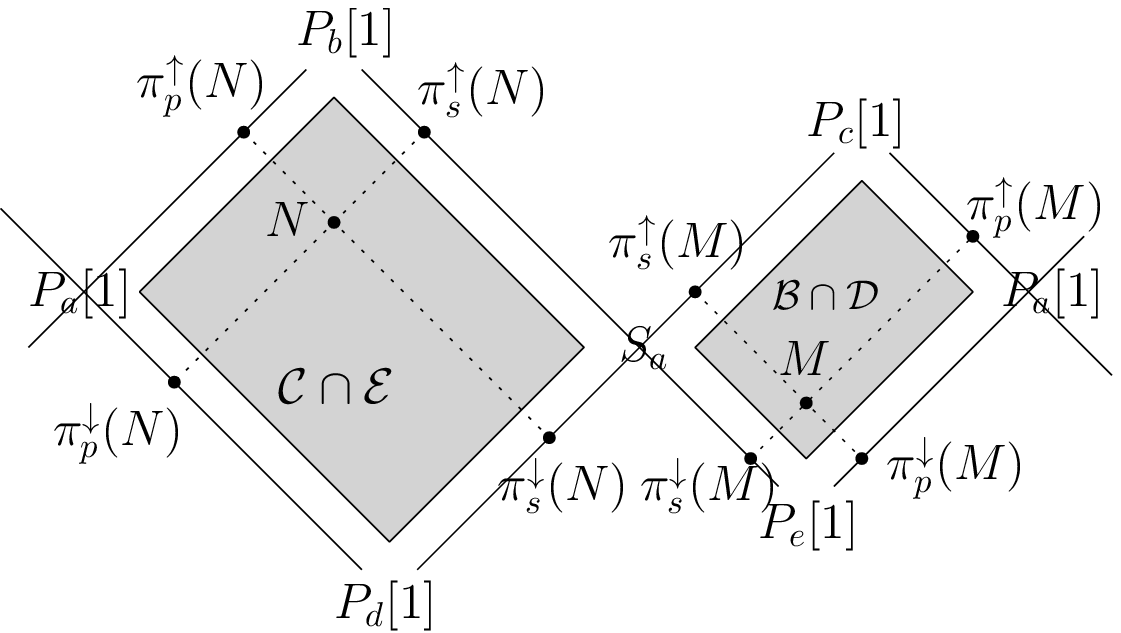}
\caption{}\label{fig:projections-bar}
\end{figure}
\begin{defi}[Projections, continued] \label{def:projections-bar}
If $M$ is a vertex of one of the two closures, 
we define 
four projections for $M$ onto the four different ``edges'' of the boundary of 
its region: 
We denote the 
projections onto the paths starting or ending next to $P_a[1]$ by 
$\pi_{p}^{\uparrow}$, $\pi_{p}^{\downarrow}$ 
and the projections onto the paths starting or ending next to $S_a$ 
by $\pi_{s}^{\uparrow}$ and $\pi_{s}^{\downarrow}$ respectively. 
We choose the upwards arrow to refer to the paths ending/starting near $P_b[1]$ or 
$P_c[1]$ and the downwards arrow to refer to paths ending/starting near 
$P_d[1]$ or $P_e[1]$.  
See Figure~\ref{fig:projections-bar}. 
\end{defi}

\begin{rem}
The statement of Theorem~\ref{thm01} is independent of the 
choice of  $\uparrow$ (paths near $P_b[1]$ or $P_c[1]$) and 
$\downarrow$ in Definition~\ref{def:projections-bar} 
as the formula is symmetric in these expressions. 
\end{rem}

\begin{ex}\label{ex-projections}
If $M\in \mathfrak e$, we have 
$\pi_p^{\uparrow}(M)=M$, $\pi_s^{\uparrow}(M)=P_b[1]$, $\pi_p^{\downarrow}(M)=P_a[1]$ 
and $\pi_s^{\downarrow}(M)=M_{\mathfrak e^a}$. 

For $S_a$  we have 
$\pi_s^{\uparrow}(S_a)=\pi_s^{\downarrow}(S_a)=S_a$ whereas the two modules 
$\pi_p^{\uparrow}(S_a)$ and $\pi_p^{\downarrow}(S_a)$ are  $\{P_b[1],P_d[1]\}$ or 
$\{P_c[1],P_e[1]\}$ depending on whether $S_a$ is viewed as an element of 
$\overline{\mathcal C\cap \mathcal E}$ or of $\overline{\mathcal B\cap \mathcal D}$. 

For $P_a[1]$, we have 
$\pi_p^{\uparrow}(P_a[1])=\pi_p^{\downarrow}(P_a[1])=P_a[1]$ whereas the two modules 
$\pi_s^{\uparrow}(P_a[1])$ and $\pi_s^{\downarrow}(P_a[1])$ are $\{P_b[1],P_d[1]\}$ or 
$\{P_c[1],P_e[1]\}$
These four shifted projectives evaluate to $1$ under $s$, and so in Theorem~\ref{thm01}, 
this ambiguity does not play a role. 
\end{ex}

With this notation we are ready to state the theorem. 
 
\begin{theorem}\label{thm01}
Let $M$ be an indecomposable object of $\mathcal{C}_f$. Then $\delta_a(M)$ is given by: \\
If $M\in (\mathcal B\cap \mathcal C) \cup (\mathcal D\cap \mathcal E)$ then
\[
 \delta_a(M) = (s(\pi_1^+(M))-s(\pi_2^+(M)))\,(s(\pi_{1}^-(M))-s(\pi_{2}^-(M));
\]
If $M\in (\mathcal B\cap \mathcal E) \cup (\mathcal C\cap \mathcal D)$ then
\[
 \delta_a(M) = -(s(\pi_2^+(M))-2s(\pi_1^+(M)))\,(s(\pi_{2}^-(M))-2s(\pi_{1}^-(M)) ;
\]
If $M\in \overline{\mathcal C\cap \mathcal E}\cup\overline{\mathcal B\cap \mathcal D}$ then
\[
 \delta_a(M) = s(\pi_{s}^{\downarrow}(M))s(\pi_{p}^{\downarrow}(M)) + s(\pi_{s}^{\uparrow}(M))s(\pi_{p}^{\uparrow}(M))  - 3\,s(\pi_p^{\downarrow}(M))s(\pi_p^{\uparrow}(M)); 
\]
If $M\in \mathcal F$ then
\[
 \delta_a(M) =0.
\]

\end{theorem}

\begin{proof}
By definition $\delta_a(M) = \rho_{\mathcal{T}} (M) - \rho_{\mathcal{T}'}(M')$, which in turn equals $s(M)-s(M')$.  
Now we consider various cases based on the location of an indecomposable object $M$ in $\mathcal{C}_f$. 
Note that $M$ satisfies exactly one of the four conditions in the definition of $\delta_a(M)$ stated in the theorem. 

\medskip
Suppose $M \in \mathcal{B}\cap \mathcal{C}$, then by definition $M = (\dots b \to a \to c \dots)$.  We also have a short exact sequence 
$$0 \to M_{c^a} \to M \to M_b\to 0 $$
in $\text{mod}\,B$ where $M_{c^a}, M_b$ are projections of $M$ onto the rays $\mathfrak{c}^a, \mathfrak{b}$ respectively.  Note that by Remark~\ref{rm-projections} we know that $M_{c^a} = (a \to c \dots)$ is a submodule of $M$ and $M_b = (\dots b)$ is a quotient of $M$.   

It follows from Section~\ref{ssec:mutating-regions} that $M'$ is obtained from $M$ by removing the support at vertex $a$, 
so $M'= (\dots b \to c \dots)$, and we have the following sequence in $\text{mod}\,B'$  
$$0\to M_c \to M' \to M_b \to 0$$ 
where $M_c = (c \dots)$ is the projection of $M$ onto the ray $\mathfrak{c}$.  Also, note that $M_c, M_b$ are $B$-modules, however they can also be thought of as modules over $B'$ because they are supported on the subquivers of $Q$ that are not affected by mutation at vertex $a$.   
By Lemma~\ref{lem3} we obtain formulas for $s(M)$ and $s(M')$ 

$$s(M) = s(M_b)s(M_{c^a}) - s(\widetilde{M}_b)s(\overline{M}_{c^a})\hspace{1cm} s(M') = s(M_b)s(M_{c}) - s(\widetilde{M}_b)s(\overline{M}_{c})$$
that correspond to the given short exact sequences.  By definition $\overline{M}_{c^a}$ is a maximal submodule of $M_{c^a}$ that is not supported on $a$, therefore $\overline{M}_{c^a}=M_c$.  Combining the two formulas we obtain 

\begin{align*}
s(M)-s(M') &= s(M_b) (s(M_{c^a})-s(M_c)) - s(\widetilde{M}_b)(s(M_c) - s(\overline{M}_c))\\
&= (s(M_b) - s(\widetilde{M}_b)) ( s(M_{c^a})-s(M_c))\\
&= (s(M_{b_a}) - s({M}_b)) ( s(M_{c^a})-s(M_c))\\
&= (s(\pi_1^+(M))-s(\pi_2^+(M)))\,(s(\pi_{1}^-(M))-s(\pi_{2}^-(M))
\end{align*}
where the second equality follows from Remark~\ref{remark01}(b) letting $M = M_{c^a}=(a \to c \dots)$ and $M_y=M_c$.   The third equality follows from Remark~\ref{remark01}(a) letting $M=M_{b_a}=(\dots b \to a)$ and $M_x=M_b$.  This shows that the theorem holds in the case $M \in \mathcal{B} \cap \mathcal{C}$.  Similar argument implies the result in the case $M \in \mathcal{D}\cap \mathcal{E}$. 

\medskip
Suppose $M \in \mathcal{B}\cap \mathcal{E}$, then we know that $M=(\dots e \to b \dots)$.  By projecting $M$ onto $\mathfrak{b}, \mathfrak{e}$ we obtain the modules $M_b, M_e$, that together with $M$ form a short exact sequence 
$$0\to M_b \to M \to M_e \to 0$$
in $\text{mod}\,B$.  Using the results from Section~\ref{ssec:mutating-regions} 
we know that $M'$ is obtained from $M$ by inserting support at vertex $a$, so $M' = (\dots e \to a \to b \dots)$.  Moreover, we obtain a short exact sequence 
$$0\to M'_{b^a} \to M' \to M_e \to 0$$
in $\text{mod}\,B'$, where $M'_{b^a}=(a\to b \dots)$ is a module over $B'$ obtained by projecting $M'$ onto the ray $\mathfrak{b}_a$ in $\text{mod}\,B'$.  By Lemma~\ref{lem3} we have the following formulas for $s(M)$ and $s(M')$ that correspond to the given short exact sequences.  
$$s(M)=s(M_e)s(M_b)-s(\widetilde{M}_e)s(\overline{M}_b) \hspace{1cm} s(M')=s(M_e)s(M'_{b^a})-s(\widetilde{M}_e)s(\overline{M}'_{b^a})$$
Observe that $\overline{M}'_{b^a}\cong M_b$, because it is the maximal submodule of $M'_{b^a}$ that is not supported on $a$.  Note that here we think of $M_b$ as a module over both $B$ and $B'$.  Therefore, we have 

\begin{align*}
s(M')-s(M) &= s(M_e)(s(M'_{b^a})-s(M_b)) - s(\widetilde{M}_e)(s(M_b)-s(\overline{M}_b)\\
&=(s(M_e)-s(\widetilde{M}_e))(s(M'_{b^a})-s(M_b))\\
&=(s(M'_{e_a})-s({M}_e))(s(M'_{b^a})-s(M_b))\\
&=(s(M_{e^a})-2s(M_e))\,(s(M_{b_a})-2 s(M_b))\\
&=(s(\pi_2^+(M))-2s(\pi_1^+(M)))\,(s(\pi_{2}^-(M))-2s(\pi_{1}^-(M)))
\end{align*}
where the second equality follows from Remark~\ref{remark01}(b) with $M=M'_{b^a}$ and $M_y = M_b$.  The third equality follows from Remark~\ref{remark01}(a) with $M = M'_{e_a}=(\dots e \to a) $, the projection of $M'$ onto the ray ${e_a}$ in $B'$, and $M_x = M_e$.  Next we justify the fourth equality.  The modules $M_{e^a}=(\dots e \leftarrow a)$, $M_{b_a}=(a \leftarrow b\to \dots) $ are projections of $M$ onto the rays $\mathfrak{e}^a, \mathfrak{b}_a$ respectively.  By Lemma~\ref{lem04} first setting $M_x = M_e$, $M_{x_y}=M'_{e_a}$, $M_{x^z}= M_{e^a}$ and then setting $M_x = M_b$, $M_{x_y}=M_{b_a}$, $M_{x^z}= M'_{b^a}$,  it follows that 
$$s(M'_{e_a})-s({M}_e)= s(M_{e^a})-2s(M_e) \hspace{2cm} s(M'_{b^a})-s(M_b)=s(M_{b_a})-2 s(M_b)$$
which implies the fourth equality.  This completes the proof of the theorem in the case $M \in \mathcal{B}\cap \mathcal{E}$ while the case $M\in \mathcal{C}\cap \mathcal{D}$ can be shown in a similar way.    

\medskip 
Suppose $M \in \overline{\mathcal{C} \cap \mathcal{E}}$.  First, we assume that $M\in \mathcal{C}\cap \mathcal{E}$ then $M=(\dots e \leftarrow a \to c \dots)$.  Thus there exists a short exact sequence 
$$0\to M_e \oplus M_c \to M \to S_a \to 0$$
in $\text{mod}\,B$ where $M_e, M_c$ are projections of $M$ onto the sectional paths $\mathfrak{e}, \mathfrak{c}$ respectively.  As discussed in section~\ref{ssec:mutating-regions} 
the module $M'$ is obtained from $M$ by reversing the arrows incident with $a$, so $M' = (\dots e \to a \leftarrow c \dots)$.  Therefore, we have a short exact sequence 
$$0\to S_a \to M' \to M_e\oplus M_c\to 0$$
in $\text{mod}\,B'$.  By Lemma~\ref{lem1} and Lemma~\ref{lem2} respectively, we obtain the following formulas for $s(M)$ and $s(M')$.
$$s(M)=s(M_e)s(M_c)+s(\widetilde{M}_e)s(\widetilde{M}_c) \hspace{1cm} s(M')= s(M_e)s(M_c) + s(\overline{M}_e)s(\overline{M}_c)$$
By Remark~\ref{remark01}(b) 
$$s(\widetilde{M}_e)=s(M_{e^a})-s(M_e) \hspace{1cm} s(\widetilde{M}_c) = s(M_{c^a})-s(M_c)$$
and, similarly, by the same remark we have 
$$s(\overline{M}_e) = 2 s(M_e)-s(M_{e^a}) \hspace{1cm} s(\overline{M}_c) = 2 s(M_c)-s(M_{c^a}).$$
Therefore, combining the results above we obtain 

\begin{align*}
s(M)-s(M') =& s(\widetilde{M}_e)s(\widetilde{M}_c)-s(\overline{M}_e)s(\overline{M}_c)\\
=& (s(M_{e^a})-s(M_e))(s(M_{c^a})-s(M_c)) - (2 s(M_e)-s(M_{e^a}))(2 s(M_c)-s(M_{c^a}))\\
=&s(M_{e^a})s(M_c) + s(M_{c^a})s(M_e) - 3 s(M_e) s(M_c)\\
=&s(\pi_{s}^{\downarrow}(M))s(\pi_{p}^{\downarrow}(M)) + s(\pi_{s}^{\uparrow}(M))s(\pi_{p}^{\uparrow}(M)) - 3\,s(\pi_p^{\downarrow}(M))s(\pi_p^{\uparrow}(M)) 
\end{align*}
which yields the desired conclusion.

If $M$ equals $P_a[1]$ or $S_a$ then $M' $ equals  $S_a$ or $P_a[1]$ respectively, and it follows from Example~\ref{ex-projections} that $\delta_a$ is well defined in this case.  Evaluating $s(M)-s(M')$ at $P_a[1], S_a$ respectively we obtain $-1, 1$ and it is easy to see that this agrees with the formula for $\delta_a(M)$ provided in the theorem.     

Now assume that $M \in \mathfrak{e}$ and $M\not= P_a[1]$, then $M = (\dots e)$ and $M' = (\dots e \to a)$ is obtained from $M_e$ by adding support at $a$.  By Lemma~\ref{lem04} we have that 
$$s(M)-s(M) = s(M_{e^a})-2 s(M).$$
On the other hand it follows from Example~\ref{ex-projections} that 
\begin{align*}
\delta_a(M) &= s(\pi_{s}^{\downarrow}(M))s(\pi_{p}^{\downarrow}(M)) + s(\pi_{s}^{\uparrow}(M))s(\pi_{p}^{\uparrow}(M)) - 3\,s(\pi_p^{\downarrow}(M))s(\pi_p^{\uparrow}(M)) \\
&= s(M_{e^a}) + s(M) - 3 s(M) \\
&= s(M) - s(M')
\end{align*}
which shows that the formula holds in this case.  

Now assume that $M \in \mathfrak{c}^a$ and $M\not= S_a$, then $M = (\dots c \leftarrow a )$ and $M' = (\dots c)$ is obtained from $M$ by removing support at $a$.  Note that $M'=M_c$ the projection of $M$ onto the ray $\mathfrak{c}$.  According to the statement of the theorem we have 
\begin{align*}
\delta_a(M) &= s(\pi_{s}^{\downarrow}(M))s(\pi_{p}^{\downarrow}(M)) + s(\pi_{s}^{\uparrow}(M))s(\pi_{p}^{\uparrow}(M)) - 3\,s(\pi_p^{\downarrow}(M))s(\pi_p^{\uparrow}(M)) \\
&= s(S_a) s(M_c) + s(M) s(P_b[1]) - 3 s(M_c) s(P_b[1]) \\
&= 2 s(M') + s(M)- 3 s(M')\\
&= s(M) - s(M')
\end{align*}
which shows that the formula holds in this case.  For $M$ lying on $\mathfrak{c}$ and $\mathfrak{e}^a$ we can apply the same argument as above.  Similar reasoning also shows that the theorem holds when $M \in \overline{\mathcal{B}\cap\mathcal{D}}$.  

Finally, if $M \in \mathcal{F}$ we know that $M$ equals $M'$, hence $\delta_a(M)=0$.  This completes the proof of the theorem. 
\end{proof}

Note, that given a frieze 
and an indecomposable $M$ in one of the six regions $\mathcal X\cap \mathcal Y$, 
it is easy to locate the entries required to compute the frieze difference $\delta_a(M)$. 
We simply need to find projections onto the appropriate rays in the frieze. 
In this way, we do not need to know the precise shape of the modules appearing in the formulas of Theorem~\ref{thm01}.

\begin{ex}
Let $\mathcal{C}_f$ be the category given in Example~\ref{ex:quiver-triangul-flip}.  We consider three possibilities for $M$ below.   

If $M = \begin{smallmatrix} 3\;\;\\8\;1\;9\\\;\;\;\;5\end{smallmatrix}$ then we know by Figure~\ref{fig:frieze-entries} that $s(M)= 12$ and $s(M')= 8$.   On the other hand, we see from Figure~\ref{fig:ARquiver} that $M \in \mathcal{B}\cap\mathcal{C}$.  Theorem~\ref{thm01} implies that 

\begin{align*}
\delta_a(M)=s(M)-s(M')&=(s(M_{b_a})-s(M_b))(s(M_{c^a})-s(M_c))\\
& = (s(\begin{smallmatrix} 3\\8\;1\end{smallmatrix})-s(\begin{smallmatrix} 3\\8\end{smallmatrix}))(s(\begin{smallmatrix}1\;9\\5\end{smallmatrix})-s(\begin{smallmatrix} 9\\5 \end{smallmatrix}))\\
&=(5-3)(5-3)=4.
\end{align*}

Similarly, if $M = \begin{smallmatrix}2\\3\;6\\8\;\;\;\end{smallmatrix}$, then $M \in \mathcal{B}\cap\mathcal{E}$ 
with $s(M)=7$ and $s(M')=9$.  The same theorem implies that

\begin{align*}
\delta_a(M)=s(M)-s(M') &= -(s(M_{e^a})-2s(M_{e}))\,(s(M_{b^a})-2s(M_b))\\
&=-(s(\begin{smallmatrix}1\\2\\6\end{smallmatrix})- 2s(\begin{smallmatrix}2\\6\end{smallmatrix})) 
(s(\begin{smallmatrix}3\\8\; 1\end{smallmatrix})-2s(\begin{smallmatrix}3\\8\end{smallmatrix}))\\
&=-(4-6)(5-6)=-2.
\end{align*}

Finally, if $M=\begin{smallmatrix}\;\;1\;7\\5\;2\\10\;\;\;\;\;\end{smallmatrix}$, then $M \in \mathcal{C}\cap\mathcal{E}$.  We also know that $s(M)=s(M')=11$.  By the third formula in Theorem~\ref{thm01}, we have 

\begin{align*}
\delta_a(M)=s(M)-s(M')&=s(M_{e^a})s(M_c) + s(M_{c^a})s(M_e) - 3 s(M_e) s(M_c)\\
&= s(\begin{smallmatrix}7\;1\\2\end{smallmatrix})s(\begin{smallmatrix}5\\10\end{smallmatrix}) + s(\begin{smallmatrix}1\\5\\10\end{smallmatrix}) s(\begin{smallmatrix}7\\2\end{smallmatrix}) -3 s(\begin{smallmatrix}7\\2\end{smallmatrix}) s(\begin{smallmatrix}5\\10\end{smallmatrix}) \\
&= 5\cdot 3 + 4\cdot 3 - 3 \cdot 3 \cdot 3 = 0.
\end{align*}
\end{ex}


\section*{Acknowledgements}
All authors thank BIRS at Banff where the work on this project began during the week of WINART.\\
K.B. acknowledges support from the Austrian Science Fund (projects DK-W1230, P 25141 and P25647), 
S.G. acknowledges support from the Swiss National Science Foundation, K.S. 
acknowledges support from the National Science Foundation Postdoctoral Fellowship MSPRF-1502881.

\bibliographystyle{alpha}
\bibliography{biblioEF}

\newcommand{\etalchar}[1]{$^{#1}$}
\def\cprime{$'$}
\begin{thebibliography}{BMR{\etalchar{+}}06a}

\bibitem[ABCJP10]{ASCJP}
Ibrahim Assem, Thomas Br{\"u}stle, Gabrielle Charbonneau-Jodoin, and Pierre-Guy
  Plamondon.
\newblock Gentle algebras arising from surface triangulations.
\newblock {\em Algebra Number Theory}, 4(2):201--229, 2010.

\bibitem[Ami09]{[Am]}
Claire Amiot.
\newblock Cluster categories for algebras of global dimension 2 and quivers
  with potential.
\newblock {\em Ann. Inst. Fourier (Grenoble)}, 59(6):2525--2590, 2009.

\bibitem[AR77]{AR-artin}
Maurice Auslander and Idun Reiten.
\newblock Representation theory of {A}rtin algebras. {V}. {M}ethods for
  computing almost split sequences and irreducible morphisms.
\newblock {\em Comm. Algebra}, 5(5):519--554, 1977.

\bibitem[BCI74]{BCI}
Duane Broline, Donald~W. Crowe, and I.~Martin Isaacs.
\newblock The geometry of frieze patterns.
\newblock {\em Geometriae Dedicata}, 3:171--176, 1974.

\bibitem[BIRS11]{BIRS}
Aslak~Bakke Buan, Osamu Iyama, Idun Reiten, and David Smith.
\newblock Mutation of cluster-tilting objects and potentials.
\newblock {\em Amer. J. Math.}, 133(4):835--887, 2011.

\bibitem[BKM16]{BKM}
Karin Baur, Alastair~D. King, and Robert~J. Marsh.
\newblock Dimer models and cluster categories of {G}rassmannians.
\newblock {\em Proc. Lond. Math. Soc. (3)}, 113(2):213--260, 2016.

\bibitem[BMR{\etalchar{+}}06a]{BMRRT}
Aslak~Bakke Buan, Robert Marsh, Markus Reineke, Idun Reiten, and Gordana
  Todorov.
\newblock Tilting theory and cluster combinatorics.
\newblock {\em Adv. Math.}, 204(2):572--618, 2006.

\bibitem[BMR06b]{BuanMarshReiten}
Aslak~Bakke Buan, Robert~J. Marsh, and Idun Reiten.
\newblock Cluster-tilted algebras of finite representation type.
\newblock {\em J. Algebra}, 306(2):412--431, 2006.

\bibitem[BMR07]{BMR}
Aslak~Bakke Buan, Robert~J. Marsh, and Idun Reiten.
\newblock Cluster-tilted algebras.
\newblock {\em Trans. Amer. Math. Soc.}, 359(1):323--332 (electronic), 2007.

\bibitem[BR87]{BR}
Michael C.~R. Butler and Claus~Michael Ringel.
\newblock Auslander-{R}eiten sequences with few middle terms and applications
  to string algebras.
\newblock {\em Comm. Algebra}, 15(1-2):145--179, 1987.

\bibitem[CC73a]{CoCo1}
John~H. Conway and Harold S.~M. Coxeter.
\newblock Triangulated polygons and frieze patterns.
\newblock {\em Math. Gaz.}, 57(400):87--94, 1973.

\bibitem[CC73b]{CoCo2}
John~H. Conway and Harold~S.M. Coxeter.
\newblock Triangulated polygons and frieze patterns.
\newblock {\em Math. Gaz.}, 57(401):175--183, 1973.

\bibitem[CC06]{CalderoChapoton}
Philippe Caldero and Fr{\'e}d{\'e}ric Chapoton.
\newblock Cluster algebras as {H}all algebras of quiver representations.
\newblock {\em Comment. Math. Helv.}, 81(3):595--616, 2006.

\bibitem[CCS06]{CalderoChapotonSchiffler}
Philippe Caldero, Fr{\'e}d{\'e}ric Chapoton, and Ralf Schiffler.
\newblock Quivers with relations arising from clusters ({$A\sb n$} case).
\newblock {\em Trans. Amer. Math. Soc.}, 358(3):1347--1364, 2006.

\bibitem[Cox71]{Coxeter}
Harold S.~M. Coxeter.
\newblock Frieze patterns.
\newblock {\em Acta Arith.}, 18:297--310, 1971.

\bibitem[CS16]{CSch}
Ilke {Canakci} and Ralf {Schiffler}.
\newblock {Cluster algebras and continued fractions}.
\newblock {\em ArXiv e-prints}, August 2016.

\bibitem[DL16]{DeLuo}
Laurent Demonet and Xueyu Luo.
\newblock Ice quivers with potential associated with triangulations and
  {C}ohen-{M}acaulay modules over orders.
\newblock {\em Trans. Amer. Math. Soc.}, 368(6):4257--4293, 2016.

\bibitem[DWZ08]{DWZ1}
Harm Derksen, Jerzy Weyman, and Andrei Zelevinsky.
\newblock Quivers with potentials and their representations. i. {M}utations.
\newblock {\em Selecta Math. (N.S.)}, 14(1):59--119, 2008.

\bibitem[FK10]{FuKeller}
Changjian Fu and Bernhard Keller.
\newblock On cluster algebras with coefficients and 2-{C}alabi-{Y}au
  categories.
\newblock {\em Trans. Amer. Math. Soc.}, 362(2):859--895, 2010.

\bibitem[FZ02]{FZ1}
Sergey Fomin and Andrei Zelevinsky.
\newblock Cluster algebras. {I}. {F}oundations.
\newblock {\em J. Amer. Math. Soc.}, 15(2):497--529 (electronic), 2002.

\bibitem[HJ16]{HJ}
Thorsten Holm and Peter J{\o}rgensen.
\newblock Generalised friezes and a modified {C}aldero--{C}hapoton map
  depending on a rigid object, {II}.
\newblock {\em Bull. Sci. Math.}, 140(4):112--131, 2016.

\bibitem[JKS16]{JKS}
Bernt~Tore Jensen, Alastair~D. King, and Xiuping Su.
\newblock A categorification of {G}rassmannian cluster algebras.
\newblock {\em Proc. Lond. Math. Soc. (3)}, 113(2):185--212, 2016.

\bibitem[KY11]{KY}
Bernhard Keller and Dong Yang.
\newblock Derived equivalences from mutations of quivers with potential.
\newblock {\em Adv. Math.}, 226(3):2118--2168, 2011.

\bibitem[Pal08]{Palu}
Yann Palu.
\newblock Cluster characters for 2-{C}alabi-{Y}au triangulated categories.
\newblock {\em Ann. Inst. Fourier (Grenoble)}, 58(6):2221--2248, 2008.

\bibitem[Sch14]{SchifflerBook}
Ralf Schiffler.
\newblock {\em Quiver representations}.
\newblock CMS Books in Mathematics/Ouvrages de Math\'ematiques de la SMC.
  Springer, Cham, 2014.

\end{thebibliography}

\end{document}